\documentclass[11pt,reqno]{amsart}
\usepackage {amssymb}
\usepackage {amsmath}
\usepackage {bbm}
\usepackage{amsthm}
\usepackage{mathtools}
\usepackage{graphicx}
\usepackage {amscd}
\usepackage {epic}
\usepackage{array}

\usepackage{mathrsfs}

\usepackage{etaremune}

\DeclareFontFamily{U}{dutchcal}{\hyphenchar\font=-1}
\DeclareFontShape{U}{dutchcal}{m}{n}{ <-> dutchcal-r}{}
\DeclareSymbolFont{dutchletters}{U}{dutchcal}{m}{n}

\DeclareMathSymbol{\mathdutchcal}{0}{dutchletters}{"41} 
\newcommand{\dcal}[1]{\text{\usefont{U}{dutchcal}{m}{n}#1}}

\usepackage[dvipsnames]{xcolor}
\usepackage[colorlinks,citecolor=OliveGreen,linkcolor=Mahogany,urlcolor=Plum,pagebackref]{hyperref}
\usepackage{cleveref}
\usepackage[alphabetic]{amsrefs}

\usepackage{enumerate}
\usepackage{tikz}
\usepackage{tikz-cd}
\usepackage{verbatim,color,geometry}
\usepackage[all]{xy}
\usepackage{enumitem}
\usepackage{bm}
\usepackage{longtable}
\usetikzlibrary{calc}
\usepackage{textcomp}
\usepackage{tablefootnote}
\usepackage{float}
\newcommand{\specialcell}[2][l]{%
  \begin{tabular}[#1]{@{}l@{}}#2\end{tabular}}
\newcommand{\mtc}[1]{\mathcal{#1}}
\newcommand{\mtf}[1]{\mathfrak{#1}}
\newcommand{\mts}[1]{\mathscr{#1}}
\newcommand{\ove}[1]{\overline{#1}}
\newcommand{\wt}[1]{\widetilde{#1}}

\newcommand{\cD}{\dcal{D}}
\newcommand{\cE}{\dcal{E}}
\newcommand{\cF}{\dcal{F}}

\newcommand{\cH}{\dcal{H}}

\newcommand{\cL}{\dcal{L}}
\newcommand{\cM}{\mathcal{M}}
\newcommand{\cN}{\mathcal{N}}
\newcommand{\cO}{\mathcal{O}}

\newcommand{\cS}{\mathcal{S}}

\newcommand{\cX}{\mathcal{X}}

\newcommand{\fm}{\mathfrak{m}}

\newcommand{\bP}{\mathbb{P}}
\newcommand{\bC}{\mathbb{C}}
\newcommand{\bR}{\mathbb{R}}
\newcommand{\bA}{\mathbb{A}}
\newcommand{\bQ}{\mathbb{Q}}
\newcommand{\bZ}{\mathbb{Z}}
\newcommand{\bD}{\mathbb{D}}

\newcommand{\bG}{\mathbb{G}}
\newcommand{\bF}{\mathbb{F}}

\newcommand{\bN}{\mathbb{N}}

\newcommand{\bH}{\mathbb{H}}

\newcommand{\hvol}{\widehat{\mathrm{vol}}}

\newcommand{\cI}{\dcal{I}}
\newcommand{\Bl}{\mathrm{Bl}}

\newcommand{\Id}{\mathrm{Id}}
\newcommand{\coker}{\mathrm{Coker}}
\newcommand{\Ker}{\mathrm{Ker}}
\newcommand{\Gr}{\mathrm{Gr}}

\newcommand{\tF}{\widetilde{F}}
\newcommand{\tY}{\widetilde{Y}}
\newcommand{\tS}{\widetilde{S}}
\newcommand{\tB}{\widetilde{B}}
\newcommand{\hE}{\widehat{E}}
\newcommand{\hY}{\widehat{Y}}
\newcommand{\hS}{\widehat{S}}
\newcommand{\hB}{\widehat{B}}
\newcommand{\bfA}{\mathbf{A}}
\newcommand{\bfP}{\mathbf{P}}
\newcommand{\Sym}{\mathrm{Sym}}

\newcommand{\sD}{\mathscr{D}}

\newcommand{\sG}{\mathscr{G}}

\newcommand{\sS}{\mathscr{S}}

\newcommand{\sV}{\mathscr{V}}
\newcommand{\sW}{\mathscr{W}}
\newcommand{\sX}{\mathscr{X}}

\newcommand{\sZ}{\mathscr{Z}}

\newcommand{\rB}{\mathrm{B}}

\DeclareMathOperator{\Aut}{Aut}

\DeclareMathOperator{\vol}{vol}

\DeclareMathOperator{\Bs}{Bs}

\DeclareMathOperator{\Spec}{Spec}
\DeclareMathOperator{\Nklt}{Nklt}

\DeclareMathOperator{\ch}{ch}
\DeclareMathOperator{\BB}{BB}

\DeclareMathOperator{\Pic}{Pic}

\DeclareMathOperator{\coeff}{coeff}
\DeclareMathOperator{\id}{id}
\DeclareMathOperator{\im}{Im}
\DeclareMathOperator{\ord}{ord}

\DeclareMathOperator{\Proj}{Proj}

\DeclareMathOperator{\sing}{sing}
\DeclareMathOperator{\Supp}{Supp}
\DeclareMathOperator{\diag}{diag}
\DeclareMathOperator{\Kss}{Kss}

\DeclareMathAlphabet{\mathbbb}{U}{bbold}{m}{n}

\newcommand{\PGL}{\mathrm{PGL}}

\newcommand{\Hom}{\mathrm{Hom}}
\newcommand{\RHom}{\mathrm{RHom}}

\newcommand{\tE}{\widetilde{E}}

\newcommand{\tX}{\widetilde{X}}
\newcommand{\tD}{\widetilde{D}}
\newcommand{\cP}{\dcal{P}}

\newcommand{\sF}{\mathscr{F}}

\newcommand{\fM}{\mathfrak{M}}
\newcommand{\fF}{\mathfrak{F}}
\newcommand{\fP}{\mathfrak{P}}

\newcommand{\K}{\mathrm{K}}
\newcommand{\Cl}{\mathrm{Cl}}
\newcommand{\can}{\mathrm{can}}
\newcommand{\ter}{\mathrm{ter}}
\newcommand{\Kst}{\mathrm{Kst}}
\newcommand{\Def}{\mathrm{Def}}
\newcommand{\Ext}{\mathrm{Ext}}
\newcommand{\plt}{\mathrm{plt}}
\newcommand{\uOmega}{\underline{\Omega}}

\newcommand{\Gor}{\mathrm{Gor}}

\newcommand{\ADE}{\mathrm{ADE}}

\newcommand{\rk}{\mathrm{rk}}
\newcommand{\bL}{\mathbb{L}}

\newcommand{\sslash}{\mathbin{/\mkern-6mu/}}

\newcommand{\YL}[1]{{\textcolor{blue}
{[Yuchen: #1]}}}

\newcommand{\AP}[1]{{\textcolor{red}{Andrea: #1}}}

\numberwithin{equation}{section}

\newtheorem{prop} {Proposition} [section]
\newtheorem{thm}[prop] {Theorem} 
\newtheorem{theodef}[prop]{Theorem-Definition}
\newtheorem{lem}[prop] {Lemma}
\newtheorem{cor}[prop]{Corollary}
\newtheorem{prop-def}[prop]{Proposition-Definition}

\newtheorem{conj}[prop]{Conjecture}

\newtheorem{theorem}[prop]{Theorem}
\newtheorem{lemma}[prop]{Lemma}

\theoremstyle{definition}
 
\newtheorem{rem}[prop] {Remark}

\newtheorem{defn}[prop]{Definition}

\newtheorem{exam}[prop]{Example}

\newtheorem{remark}[prop]{Remark}

\title{The boundary of K-moduli of prime Fano threefolds of genus twelve}

\author{Anne-Sophie Kaloghiros}
\address{Department of Mathematics, Brunel University of London, Uxbridge UB8 3PH, United Kingdom}
\email{anne-sophie.kaloghiros@brunel.ac.uk}

\author{Yuchen Liu}
\address{Department of Mathematics, Northwestern University, 2033 Sheridan Rd, Evanston, IL 60208, USA}
\email{yuchenl@northwestern.edu}

\author{Andrea Petracci}
\address{Dipartimento di Matematica, Università di Bologna, Piazza di Porta San Donato 5, Bologna 40126, Italy}
\email{a.petracci@unibo.it}

\author{Junyan Zhao}
\address{Department of Mathematics, University of Maryland, 4176 Campus Dr, College Park, MD 20742, USA}
\email{jzhao81@umd.edu}

\date{\today}

\begin{document}

\begin{abstract}
We study the K-moduli stack of prime Fano threefolds of genus twelve, known as $V_{22}$. We prove that its boundary, which parametrizes singular members, is purely divisorial and consists of four irreducible components corresponding to the four families of Prokhorov’s one-nodal $V_{22}$.

A key ingredient is a modular relation between Fano threefolds $X$ and their anticanonical K3 surfaces $S$. 
We prove that the forgetful morphism from the moduli of Fano--K3 pairs $(X,S)$ where $X$ is a K-semistable degeneration of $V_{22}$ to the moduli space of genus $12$ polarized K3 surfaces $(S,-K_X|_S)$ is an open immersion. 
In particular, the K-moduli of $V_{22}$ is governed by the moduli of their anticanonical K3 surfaces, providing a modular realization of Mukai's philosophy. 
Along the way, we develop a general deformation framework for Fano threefolds of large volume, which may be useful beyond the study of K-moduli.
\end{abstract}

\maketitle
\tableofcontents

\section{Introduction}

%\AS{I have added an introduction' file: these are suggestions of reformulations and simplifications for Sections 1 until 1.4. The introduction could be shortened: most of what is in Section 1.5 appears earlier and could be made more explicit when it first appears. That would not necessarily be better from the point of view of readability. I didn't want to interfere with the introduction file itself. Please comment out when read.}

K-stability was introduced by differential geometers \cite{Tia97, Don02} to characterize the existence of K\"ahler--Einstein metrics on Fano varieties. The algebraic reformulation of the theory has led to substantial developments, chief among them the construction of the \emph{K-moduli space} %a projective good moduli space 
that parametrizes K-polystable Fano varieties; see \cite{Xu25}. 

Over the past decade, significant progress has been made in explicit K-stability in dimension three -- that is in understanding K-stability of smooth Fano threefolds and the associated K-moduli spaces; see e.g.\ \cite{LX19, AZ23, ADL21, ACC+}. Yet, relatively few result are known about K-moduli of prime Fano threefolds (those with Picard rank and Fano index one) because current techniques provide little control over their degenerations. In particular, a description of the boundary of the K-moduli space is known for no family of prime Fano threefolds. 

An especially interesting family of prime Fano threefolds is that of genus twelve: these were first constructed by Iskovskikh and are known as \emph{Fano threefolds of type $V_{22}$}, they have trivial intermediate Jacobian and form a $6$-dimensional family. Further, while a general Fano threefold of type $V_{22}$ is K-stable, some smooth $V_{22}$ are strictly K-semistable; see \cite{Tia97}. The precise description of which smooth $V_{22}$ are K-polystable is not known, and is the object of Donaldson's conjecture \cite{Donaldson08}. %\footnote{\YL{if we mention Donaldson then Tian should be cited too.}}\footnote{\AS{Do you mean the construction of non K-ps $V_{22}$ in Corollary 1.3 of Tian 97? I am not aware of a precise conjecture of Tian's}}

\smallskip

The goal of this paper is to describe the boundary of the K-moduli stack of $V_{22}$. The methods we develop highlight the role of anticanonical K3 surfaces in controlling degenerations of $V_{22}$. 

\subsection{K-moduli of $V_{22}$ and Mukai's philosophy}

Let $\cM^\K$ denote the K-moduli stack of $V_{22}$ with reduced structure (cf.~\Cref{defn:different K-moduli stacks}), which parametrizes K-semistable Fano threefolds admitting a $\bQ$-Gorenstein deformation to a smooth $V_{22}$. 

In \cite{Pro16}, Prokhorov classified all $\bQ$-Gorenstein degenerations of smooth $V_{22}$ with a single ordinary double point, called \emph{one-nodal $V_{22}$}, into four families. Subsequently, it was proved in \cite{DK25} that a general member of each of these families is K-polystable. Our first main theorem describes the boundary of the K-moduli stack $\cM^{\K}$: we show that every singular K-semistable degeneration of $V_{22}$ lies in the closure of the locus of one-nodal $V_{22}$.% \AP{Shouldn't we write $V_{22}$'s ?}\textcolor{blue}{After double checking with Anne-Sophie, we agree that we should treat it as same form in singular and plural. } \AS{I don't like using $V_{22}$s or $V_{22}$ for plural. I prefer Fano threefolds of type $V_{22}$- in the book we used ``Fano threefolds $V_{22}$"- but it's too long, so i suppose $V_{22}$ is fine so long as we are consistent.}

\begin{thm}\label{thm:main theorem 1}
Every singular K-semistable $V_{22}$ is a degeneration of one-nodal $V_{22}$. In particular, the boundary of $\cM^{\K}$ parametrizing singular members is purely divisorial and consists of four irreducible components, each corresponding to one of the four families of one-nodal $V_{22}$. Moreover, every K-semistable $V_{22}$ with isolated singularities has at worst nodal singularities.
\end{thm}

Our starting point is the moduli continuity method which bounds the singularities of K-semistable degenerations via local volumes; see \cite{LX19, LZ25, Liu25}. Unlike in the cases of the deformation families studied in \cite{LX19,LZ25}, there is no a-priori meaningful compactification of the moduli of smooth $V_{22}$. For instance, the GIT quotient associated to the Grassmannian construction of $V_{22}$ (cf.\ \cite[Section 3.1]{GLT15}) is difficult to analyze and its boundary has no satisfactory modular interpretation (cf.\ Remark \ref{rmk:picrank}). In addition, the singularity estimates by themselves are not strong enough to control the singularities of K-semistable degenerations using classification results. This calls for a complementary perspective. 

The central idea of this paper follows Mukai’s philosophy, further developed by Beauville, that the geometry and deformation theory of a Fano threefold are governed by its anticanonical K3 sections (cf. \cite{Muk88,Muk92,Muk02,Bea04}). %We study the moduli of Fano--K3 pairs and show that the deformation theory of the ambient K-semistable Fano threefold is controlled by that of its anticanonical K3 surface. 
Concretely,  let $\cP^{\K,\ADE}$ denote the moduli stack of Fano--K3 pairs $(X,S)$ where $X$ is a K-semistable $V_{22}$ and $S\in |-K_X|$ is an ADE K3 surface (cf.~\Cref{defn:different K-moduli stacks}), and let $\cF_{22}$ denote the moduli stack of polarized K3 surfaces of degree $22$ (cf.~\Cref{defn:moduli of K3}). These stacks are related by the diagram 
\[
\begin{tikzcd}[ampersand replacement=\&]
	\& {\cP^{\K,\ADE}} \&\&\& {(X,S)} \& \\
	{\cM^\K} \&\& {\cF_{22}} \& X \&\& {(S,-K_X|_S)}
	\arrow["\Psi"', from=1-2, to=2-1]
	\arrow["\Phi", from=1-2, to=2-3]
	\arrow[maps to, from=1-5, to=2-4]
	\arrow[maps to, from=1-5, to=2-6]
\end{tikzcd},
\]
where $\Psi$ is surjective (cf.~\Cref{thm:volume bound}) and $\Phi$ is birational (cf.~\cite[Theorem 1.3]{BKM25}), facts that go back to Mukai. Our second main result strengthens this statement and shows that $\Phi$ is an open immersion, thereby providing a precise modular realization of Mukai's philosophy. %  The precise relationship between the deformation theories of the K3 surface and the Fano threefold is established as follows. 

\begin{thm}\label{thm:Fano-K3}
The forgetful morphism
\[
\Phi:\cP^{\K,\ADE} \longrightarrow \cF_{22}, \qquad (X,S) \ \mapsto\ \big(S, -K_X|_S\big),
\]
is an open immersion. Its image is contained in the union of the Brill--Noether general locus and the four Noether--Lefschetz divisors corresponding to Prokhorov's degeneration of types~{\rm{I--IV}}.
\end{thm}

% the complement of the seven Noether--Lefschetz divisors
% \[
% \cD_{1,0}^{22},\ \ \cD_{2,0}^{22},\ \ \cD_{3,0}^{22},\ \ \cD_{4,0}^{22},\ \ 
% \cD_{7,2}^{22},\ \ \cD_{8,2}^{22},\ \ \cD_{10,4}^{22}.
% \]
In particular, a K-semistable 
$V_{22}$ is uniquely determined by its anticanonical polarized K3 surfaces, and the geometry of the K-moduli stack $\cM^{\K}$ is governed by the moduli of polarized K3 surfaces. As a consequence, the K-moduli stack $\cM^\K$ is smooth (cf. \Cref{cor:reduced K-moduli of V22 is smooth}).

It is natural to expect that Theorem \ref{thm:Fano-K3} can be extended to any smoothable Gorenstein canonical Fano threefold. In Section \ref{sec:Reconstruction}, we explore another formulation of the reconstruction principle: starting from a polarized K3 surface lying in one of the remaining Noether--Lefschetz divisors of $\cF_{22}$, we construct a Gorenstein canonical Fano threefold containing it as an anticanonical divisor. This leads to Conjecture \ref{conj:Fano-K3-inj}, which predicts that the moduli stack of Fano--K3 pairs $(X,S)$, with $X$ a Gorenstein canonical degeneration of $V_{22}$, maps isomorphically onto $\cF_{22}$. The constructions carried out there provide further evidence that the geometry of $V_{22}$ is entirely encoded by its anticanonical K3 surface.

\smallskip
As a byproduct of our study of the K-moduli of $V_{22}$, we prove the rationality of the moduli of degree $22$ K3 surfaces, which we believe was already known to Mukai. In fact, in \cite[Corollary~0.5]{Muk88}, he states that $\cF_{22}$ is unirational and attributes this to Iskovskikh \cite{Isk77}. Although the ingredients of the argument are known, we are unaware of a reference in the literature; therefore, we record it here for completeness.

\begin{thm}\label{thm:rationality of F22}
The moduli stack $\cF_{22}$ of polarized K3 surfaces of degree $22$ and its coarse moduli space $\fF_{22}$ are rational.
\end{thm}

%\AP{Can we add, to be more precise, that $\fF_{22}$ is the coarse moduli space of the stack $\cF_{22}$? At first sight, these two different F's seem equal.} I changed to "both are rational"

\smallskip

\subsection{General deformation framework for K-moduli of Fano threefolds}\label{sec:deformation-intro}

The goal of this subsection is to isolate the deformation-theoretic input that allows us to control K-semistable degenerations of Fano threefolds via their anticanonical K3 surfaces. Since existing approaches are not sufficient to treat prime Fano threefolds, we develop a framework based on the moduli continuity method that systematically combines deformation theory and K-stability. This framework applies directly to Fano threefolds of volume at least $22$, where K-stability forces singularities into a tightly constrained list (cf.~\Cref{thm:volume bound}). We expect this method to have wider applications.

A fundamental input is a deformation result generalizing Beauville's theorem for smooth Fano threefolds and smooth K3 surfaces (cf.~\cite{Bea04}) to a mildly singular weak Fano--K3 setting. Let $X$ be a Gorenstein terminal weak Fano threefold, and let $S \in |-K_X|$ be a K3 surface with ADE singularities. Let $\Lambda \subseteq \Pic(S)$ denote the saturation of the image of $\Pic(X) \to \Pic(S)$. We prove that the forgetful morphism
\[
\Def_{(X,S)} \longrightarrow \Def_{(S,\Lambda)}
\]
is smooth of relative dimension $h^{1,2}(X)$ (see \Cref{thm:beauville} and \Cref{thm:smooth dominant forgetful map}). This Beauville-type theorem serves as the main structural mechanism of our approach, showing that the deformation of the threefolds is largely governed by their anticanonical K3 surfaces together with the induced lattices. As a consequence, we derive the invariance of $h^{1,2}$ in families of Gorenstein terminal weak Fano threefolds (see Corollary \ref{cor:h12}), a result of independent interest.

K-stability imposes additional constraints on possible degenerations. Recent results \cite{LX19,LZ25,Liu25} show that K-semistable degenerations of Fano threefolds with sufficiently large volume admit only hypersurface singularities of specific types; see \Cref{thm:volume bound}. Building on this, we prove that the anticanonical divisor is very ample (cf.~\Cref{thm:Kss-very-ample}) and that Fano threefolds singular along a line are K-unstable (cf.~\Cref{thm:sing-line}). These results imply that K-semistable degenerations admit partial smoothings and enjoy well-behaved deformation theory (cf.~\Cref{thm:smoothable}).

Taken together, these results allow us to control K-semistable degenerations via their anticanonical K3 surfaces, forming the main technical input for the proofs of Theorems \ref{thm:main theorem 1} and \ref{thm:Fano-K3}.

\smallskip

\subsection{Outline of the proof} 

We now outline the main ideas and strategy underlying the proofs of Theorems \ref{thm:main theorem 1} and \ref{thm:Fano-K3}. The two results are closely intertwined:  understanding the structure of the forgetful morphism is the key input in describing the boundary of $\cM^{\K}$.

Let $X$ be a singular member of $\cM^{\K}$. By \cite{Liu25}, $X$ has either isolated $cA_{\leq 2}$-singularities or non-isolated $A_\infty$ or $D_\infty$ singularities (cf. \Cref{defn:non-isolated-singularities}). We show that every such $X$ deforms to a one-nodal $V_{22}$, and that the analysis of these degenerations simultaneously determines the structure of the forgetful morphism $\cP^{\rm K,ADE}\to \cF_{22}$.

The argument relies on the deformation framework developed in Section \ref{sec:deformation-intro}, which allows us to control deformations of $X$ via the pair $(X,S)$, where $S\in |-K_X|$ is an anticanonical K3 surface. We analyze separately the isolated and non-isolated cases. The isolated case follows directly from this deformation theory, while the non-isolated case requires further geometric input and ultimately yields the structural description of the forgetful morphism.

First, suppose that $X$ has only isolated singularities, and hence is  terminal. In Theorem \ref{cor:deforms to a one-nodal}, we show that $X$ must be nodal.  Indeed, blowing up a singular point $p \in X$ yields a Gorenstein terminal weak Fano threefold $\tX$ with exceptional divisor $E$. If $S$ is a general anticanonical K3 surface $S$ passing through $p$, its strict transform $\tS$ is an anticanonical K3 surface of $\tX$. If $p$ were not nodal, then $\Pic(\tX)$ would have rank two, and be generated by the pullback of $-K_X$ and $E$. Applying \Cref{thm:beauville} to $(\tX,\tS)$ produces a deformation $(\tX_t,\tS_t)$ such that $\tS_t$ corresponds to a general K3 surface in the nodal divisor $\cD_{0,-2}^{22}$. Passing to the anticanonical model gives a deformation $X_t$ of $X$, which must be a smooth $V_{22}$ by the injectivity of the forgetful map over the terminal locus. This contradicts a minimal log discrepancy argument, and hence $X$ is nodal. By \cite{Nam97}, $X$ deforms to a one-nodal $V_{22}$.

Next, suppose that $X$ has non-isolated singularities. We show in \Cref{appendix B} that there are such examples. As a first step, we show in \Cref{thm:sing-line} that the singular locus of $X$ cannot contain a line. Blowing up the one-dimensional singular locus produces a Gorenstein terminal weak Fano threefold $\tX$. If $S\subset X$ is a general anticanonical K3 surface, its strict transform $\tS \subset \tX$ contains exceptional $(-2)$-curves that do not arise from restrictions of line bundles on $\tX$. By \Cref{thm:beauville}, these curve classes disappear under general deformation, so $\tX$ deforms to a Gorenstein terminal weak Fano threefold whose anticanonical model $X_t$ is terminal. 

If $X_t$ itself is a degeneration of $V_{22}$, then we reduce to the terminal case treated above. Otherwise, by \cite{Nam97}, $X_t$ deforms to another smooth Fano threefold of volume $22$, namely to a member of families \textnumero 2.15, \textnumero 2.16, or \textnumero 3.6. The first two possibilities are excluded by \cite{LZ25} and \Cref{appendix A}, which show that their K-moduli stacks form disjoint connected components. For family \textnumero 3.6, a dimension count shows that the corresponding pair moduli stack has codimension at least two in $\cP^{\rm K,ADE}$. Using purity of the exceptional locus and smoothness of $\cF_{22}$, we deduce that the forgetful map $\cP^{\rm K,ADE} \to \cF_{22}$ is an open immersion. This structural input allows us to control the anticanonical K3 surfaces arising from $X$: in particular, a general anticanonical K3 surface of $X$ lies in the Noether–Lefschetz divisor corresponding to one-nodal $V_{22}$ of Prokhorov type~I, and hence $X$ is a degeneration of such varieties, completing the proof of Theorem~\ref{thm:main theorem 1}. Finally, combining the open immersion with the description of its image yields Theorem~\ref{thm:Fano-K3}.

\smallskip
\subsection{History and prior work}

We briefly review the history of the study of $V_{22}$ and the progress and current status of the K-moduli of Fano threefolds.

\subsubsection{History of $V_{22}$}

The varieties $V_{22}$ occupy a distinguished position among Fano threefolds. They are one of the four deformation families of smooth Fano threefolds with $b_2=1$ and $b_3=0$, alongside $\bP^3$, the smooth quadric threefold $Q^3$, and the quintic del Pezzo threefold $V_5$. Unlike the first three, which are rigid and admit explicit descriptions, the $V_{22}$ form a nontrivial $6$-dimensional moduli family, despite having trivial intermediate Jacobian.

%A smooth Fano threefold $X$ is called \emph{prime} if $\Pic(X) \simeq \bZ \cdot [-K_X]$. In his early work, Fano asserted without proof that such varieties always contain lines (cf.~\cite{Fano37}). Based on this assertion, Roth claimed that prime Fano threefolds exist only for $g \leq 10$ (cf.~\cite{Roth50,Roth55}). This picture was clarified by Iskovskikh in his classification of prime Fano threefolds~\cite{Isk78}. Using the method of double projection from a line, he proved Fano’s assertion and showed that, in addition to the known families with $g \leq 10$, there also exist prime Fano threefolds of genus $12$. These are the varieties now known as $V_{22}$.

The $V_{22}$ were first discovered by Iskovskikh in his classification of \emph{prime} Fano threefolds; see \cite{Isk78}. In this work, Iskovskikh used anticanonical K3 surfaces as auxiliary tools to study linear systems on Fano threefolds. Mukai later introduced a complementary perspective, in which K3 surfaces play a central role in the construction of Fano threefolds. This viewpoint can be regarded as a higher-dimensional analogue of the reconstruction of K3 surfaces from canonical curves; see \cite{SD74}. In \cite{Muk88,Muk02}, Mukai realized both polarized K3 surfaces and prime Fano threefolds as linear sections of homogeneous varieties, and formulated a reconstruction principle starting from Brill--Noether general polarized K3 surfaces. This viewpoint was recently placed on a firm foundation by Bayer, Kuznetsov, and Macrì (cf.~\cite{BKM24,BKM25}), who proved that a smooth prime K3 surface of genus $12$ admits a unique embedding, up to isomorphism, as an anticanonical divisor in a smooth $V_{22}$.

From the perspective of K-stability, the $V_{22}$ also play a significant role. It was once expected that a Fano manifold with finite automorphism group should admit a K\"ahler--Einstein metric. However, Tian \cite{Tia97} showed that certain $V_{22}$ without nontrivial holomorphic vector fields do not admit K\"ahler--Einstein metrics. This phenomenon led Tian to introduce the notion of K-stability as a criterion for the existence of K\"ahler--Einstein metrics. Subsequently, Donaldson \cite{Donaldson07,Donaldson08} proved that the Mukai–Umemura threefold $V_{22}^{\mathrm{MU}}$ is K-polystable. More recently, it was shown in \cite{CS23,Fuj23} that every smooth $V_{22}$ admitting a faithful $\bG_m$-action is K-polystable. It is now widely expected that every smooth $V_{22}$ is K-semistable; see e.g.\ \cite[Problem 10]{XZ26}.

\subsubsection{K-moduli of Fano threefolds}

One of the most successful approaches to studying K-moduli of Fano varieties is the \emph{moduli continuity method}. Roughly speaking, one starts with a concrete parameter space for a given family of varieties—often arising from a Hilbert scheme or a GIT construction—and compares it with the K-moduli space via the theory of K-stability. The key steps are to identify a candidate compact moduli space and to control the possible K-semistable degenerations, typically using a priori estimates on the local volumes of singularities. This strategy has been successfully applied in several settings; see, for instance, \cite{MM93, OSS16, SS17, LX19, Liu22, ADL21, LZ25,Zha26}.

A prominent example is the work of \cite{LX19} on cubic threefolds. In that case, cubic threefolds admit a natural GIT compactification as hypersurfaces in $\bP^4$. The key point is that strong control on the singularities of K-semistable degenerations implies that any K-semistable limit must again be a cubic threefold. This allows one to identify the K-moduli space with the corresponding GIT quotient.

Another successful application appears in the work of \cite{ADL21} on quartic K3 surfaces. There, the authors study the K-moduli of pairs $(\bP^3,cS)$ with $S\in|\cO_{\bP^3}(4)|$ and analyze its wall-crossing behavior as the coefficient $c$ varies. In particular, they show that the resulting K-moduli spaces interpolate between the GIT moduli and the Baily--Borel compactification of quartic K3 surfaces. This reflects the classical relation between Fano varieties and K3 surfaces and suggests that it admits a modular interpretation.

For the Fano threefolds $V_{22}$ considered in this paper, however, applying these ideas presents additional challenges. Unlike cubic threefolds, the varieties $V_{22}$ do not admit a natural description as hypersurfaces or complete intersections, and no explicit GIT compactification of their moduli is currently known. Furthermore, existing singularity estimates for K-semistable degenerations do not appear strong enough to ensure that the limits remain within the same geometric class of varieties. On the other hand, the moduli space $\cF_{22}$ of polarized K3 surfaces of degree $22$ contains many Noether--Lefschetz divisors, leading to a much richer boundary structure. 

The main technical contribution of this paper is the development of a deformation-theoretic framework that allows us to control degenerations of $V_{22}$ via their anticanonical K3 surfaces, thereby providing a systematic bridge between the K-moduli of $V_{22}$ and the moduli of polarized K3 surfaces.

%The results of this paper suggest that a complete determination of the K-moduli of $V_{22}$ should be accessible. We plan to address this problem in future work, including establishing the K-semistability of smooth $V_{22}$, building on the deformation-theoretic framework developed here.

\subsection{Conventions and notations}

We adopt the following conventions throughout this paper.
\begin{itemize}
    \item We work over the field $\bC$ of complex numbers.
    \item We follow the conventions of \cite{KM98} regarding singularities of varieties and log pairs.
    \item Throughout this paper, we use calligraphic letters $\cM,\cF,\cP$ to denote moduli stacks, and the corresponding fraktur letters $\fM,\fF,\fP$ to denote their good or coarse moduli spaces. We use script letters such as $\sX,\sS$ to denote the total spaces of families.
    \item We do not distinguish between an object parametrized by a stack and the corresponding point of the stack. For instance, when we write $X\in \cM^{\K}$, we mean that $X$ is a variety represented by a $\bC$-point of $\cM^{\K}$.
\end{itemize}

For the reader’s convenience, we collect here the notation for the moduli stacks and spaces most frequently used throughout the paper.

\renewcommand{\arraystretch}{1.5}
\begin{center}
\begin{longtable}{| p{.13\textwidth} | p{.75\textwidth} |}
    \hline \textbf{Notation} & \textbf{Definition/Description}   \\ 
    \hline $\cM^\K_{3,22} $&   K-moduli stack of Fano threefolds of volume 22\\ \hline
    $\cM^\K $&  K-moduli stack of $V_{22}$ with reduced stack structure\\ \hline
     $\fM^\K $&  K-moduli space of $V_{22}$\\ \hline
    $\cP^{\K}_{3,22,1}(c)$ &  K-moduli stack of threefold pairs $(X,cS)$, $0<c<1$, such that $\vol(X)=22$ and $S$ is an integral divisor satisfying $K_X+S\sim_{\bQ}0$  \\ \hline
    
    $\cP^{\K, \ADE}$ &  moduli stack of pairs $(X,S)$ such that $X\in \cM^{\K}$ and $S\in |-K_X|$ is ADE\\ \hline
    $\cF_{d} $&   moduli stack of primitively polarized ADE K3 surfaces of degree $d$\\ \hline
    $\overline{\fF}_{d}^{\BB} $& Baily--Borel compactification of the coarse moduli space $\fF_d$ of $\cF_d$\\ \hline
     $\cF_{d,\Lambda} $&   Noether--Lefschetz locus in $\cF_{d} $ associated to the lattice $\Lambda$\\ \hline
     $\cF_{(\Lambda,h)} $&   moduli stack of $(\Lambda,h)$-polarized K3 surfaces\\
     \hline

    \caption{Notations and descriptions of moduli stacks}
    \label{tab:notations}
\end{longtable}

\end{center}

\subsection*{Acknowledgements}
We are grateful to Chenyang Xu for fruitful discussions at an early stage of this project. We thank Philip Engel and Jakub Witaszek for helpful conversations, and Gavril Farkas, Klaus Hulek, Yuri Prokhorov, and Alessandro Verra for answering our questions. We also thank Dori Bejleri and Patrick Brosnan for valuable feedback. Finally, we are especially grateful to Alexander Kuznetsov for his detailed comments on the draft.

ASK was supported by EPSRC grant EP/V056689/1.
YL is partially supported by NSF CAREER Grant DMS-2237139 and an AT\&T Research Fellowship from Northwestern University.
AP acknowledges support from INdAM--GNSAGA and from the European Union -- NextGenerationEU under the National Recovery and Resilience Plan (PNRR), Mission 4 ``Education and Research,'' Component 2 ``From Research to Business,'' Investment 1.1, PRIN 2022, \emph{Geometry of algebraic structures: moduli, invariants, deformations}, DD No.~104 (2/2/2022), proposal code 2022BTA242, CUP J53D23003720006.
JZ is supported by the UMD Postdoctoral Travel Grant and the Simons Travel Grant.

\medskip
\section{Preliminaries}

\subsection{Geometry of Fano threefolds}

\begin{defn}
    A \emph{log Fano pair} (resp. \emph{log weak Fano pair}) $(X,D)$ consists of a normal projective variety $X$ and a boundary divisor $D$ such that the log anticanonical divisor $-K_X-D$ is an ample (resp. a nef and big) $\mathbb{Q}$-Cartier $\bQ$-divisor, and $(X,D)$ has klt singularities. If $D=0$, then $X$ is called a \emph{$\mathbb{Q}$-Fano variety} (resp. \emph{weak $\mathbb{Q}$-Fano variety}).

    By \cite{BCHM10}, if $X$ is a weak $\bQ$-Fano variety, the anti-canonical divisor $-K_X$ is always big and semiample and its ample model $\overline{X} :=\Proj R(-K_X)$ is a $\bQ$-Fano variety. We call $\overline{X}$ the \emph{anticanonical model} of $X$. 
    
  If a (weak) $\bQ$-Fano variety $X$ is Gorenstein, then it necessarily has canonical singularities. In this case, we say that $X$ is a \emph{Gorenstein canonical} (weak) Fano variety, or (\emph{weak}) \emph{Fano variety} for abbreviation.
\end{defn}

\begin{defn}
Let $X$ be a Gorenstein canonical weak Fano threefold. The \emph{volume} of $X$, denoted by $\vol(X)$, is $(-K_X)^3$, and the \emph{genus} of $X$, denoted by $g(X)$, is $\frac{(-K_X)^3}{2} + 1$.
\end{defn}

\begin{theorem}\label{thm:generalelephant}\textup{(General elephants, cf.~\cite{Rei83,Sho79})}
    Let $X$ be a Gorenstein canonical weak Fano threefold. Then $|-K_X|\neq \emptyset$, and a general element $S\in|-K_X|$ is a K3 surface with at worst ADE singularities.
\end{theorem}

\begin{thm}[\cite{RS09}]\label{thm:NL-Cl}
Let $X$ be a Gorenstein canonical Fano threefold such that $|-K_X|$ is very ample. Then for a very general $S\in |-K_X|$, the restriction map $\Cl(X) \to \Cl(S)$ is an isomorphism. 
\end{thm}

\begin{defn}
A smooth Fano threefold of Picard rank $1$, Fano index $1$, and genus $12$ is called a \emph{smooth $V_{22}$}. More generally, a Fano threefold is called a $V_{22}$ if it appears as a $\bQ$-Gorenstein degeneration of smooth $V_{22}$.
\end{defn}

By the description of prime Fano threefolds due to Mukai, a smooth $V_{22}$ can be realized as the smooth zero locus of a global section of the vector bundle $(\wedge^2 \mathcal U^*)^{\oplus 3}$ on $\Gr(3,7)$, where $\mathcal U$ denotes the universal subbundle on $\Gr(3,7)$.

\begin{theorem}[One-nodal $V_{22}$; \textup{\cite[Theorem 1.2]{Pro16}}]\label{thm:Prokhorov's type I-IV} Every $V_{22}$ with a single $A_1$-singularity as its singular locus, denoted by $X$, is the midpoint of the following Sarkisov link \[\begin{tikzcd}[ampersand replacement=\&]
	\& Y \&\& {Y^+} \\
	Z \&\& X \&\& {Z^+}
	\arrow["\chi", dashed, from=1-2, to=1-4]
	\arrow["f"{description}, from=1-2, to=2-1]
	\arrow["\pi"{description}, from=1-2, to=2-3]
	\arrow["{\pi^+}"{description}, from=1-4, to=2-3]
	\arrow["{f^+}"{description}, from=1-4, to=2-5]
\end{tikzcd},\] where $\pi$ and $\pi^+$ are small $\bQ$-factorializations, and $\chi$ is a flop. The morphisms $f$
and $f^+$ are extremal contractions described as one of the following four cases:
\renewcommand{\arraystretch}{1.5}
\begin{longtable}{| >{\centering\arraybackslash}p{.04\textwidth} 
                  | >{\centering\arraybackslash}p{.04\textwidth} 
                  | >{\centering\arraybackslash}p{.33\textwidth} 
                  | >{\centering\arraybackslash}p{.04\textwidth} 
                  | >{\centering\arraybackslash}p{.33\textwidth} |}
    \hline 
    \textup{\textnumero} & $Z$ & $f$ & $Z^+$ & $f^+$ \\ \hline 
    \rm I & $\bP^3$ & \rm blowup along a smooth rational quintic curve not in a quadric & $\bP^3$ & \rm blowup along a smooth rational quintic curve not in a quadric \\ \hline
    \rm II & $Q^3$ & \rm blowup along a non-degenerate smooth rational quintic curve & $\bP^2$ & \rm  a conic bundle with discriminant curve of degree 3 \\ \hline
    \rm III & $V_5$ & \rm blowup along a smooth
rational quartic curve & $\bP^1$ & \rm a del Pezzo fibration of degree 6\\ \hline
    \rm IV & $\bP^2$ & \rm $\bP\cE\rightarrow \bP^2$ for a stable bundle $\cE$ of character $\ch(\cE)=(2,0,-4)$ & $\bP^1$ & \rm a del Pezzo fibration of degree 5 \\ \hline
    \caption{Prokhorov’s Type I–IV degenerations}
    \label{tab:description of Prokhorov degeneration}
\end{longtable}
\noindent where $Q^3$ is the smooth quadric threefold in $\bP^4$, and $V_5$ is the smooth quintic del Pezzo threefold.
\end{theorem}

\begin{rem}
The one-nodal degenerations of all prime Fano threefolds are classified in \cite{KP25}.
\end{rem}

\begin{thm}[\textup{\cite[Theorem 1.1]{DK25}}]
   A general one-nodal $V_{22}$ is K-polystable.
\end{thm}

\begin{lem}\label{lem:line-bundle-extension}
Let $\sX\to T$ be a $\bQ$-Gorenstein family of weak Fano varieties over a smooth pointed variety $(0\in T)$. Let $L$ be a line bundle on the central fiber $\sX_0$. Then, after an \'etale base change $(0'\in T')\to (0\in T)$, there exists a unique line bundle $\cL'$ on $\sX'=\sX\times_T T'$ such that
\[
\cL'|_{\sX'_{0'}} \simeq L
\]
under the natural identification $\sX'_{0'} \simeq \sX_0$.
\end{lem}

\begin{proof}
By Artin's representability theorem \cite[\href{https://stacks.math.columbia.edu/tag/0D2C}{Tag 0D2C}]{stacks-project}, the relative Picard functor $\Pic_{\sX/T}$ is represented by an algebraic space locally of finite presentation over $T$. Since each fiber $\sX_t$ is weak Fano, Kawamata--Viehweg vanishing gives
\[
H^1(\sX_t,\cO_{\sX_t})\ =\  H^2(\sX_t,\cO_{\sX_t})\ =\  0
\] for every $t\in T$ up to shrinking the base. The vanishing of $H^1$ implies that $\Pic_{\sX/T}\to T$ is unramified, while the vanishing of $H^2$ implies that it is formally smooth (see \cite[Theorem~9.5.11 and Proposition~9.5.19]{FGA}). Hence $\Pic_{\sX/T}\to T$ is formally \'etale. Since it is locally of finite presentation, it follows that it is \'etale. Therefore, after an \'etale base change $(0'\in T')\to (0\in T)$, the given line bundle $L$ determines a unique section of $\Pic_{\sX'/T'}\to T'$ extending the point corresponding to $L\in \Pic(\sX_0)$. This section corresponds to a unique line bundle $\cL'$ on $\sX'$ extending $L$, as desired.
\end{proof}

\begin{lemma}\label{lem:Picard rank one}
    For any $\bQ$-Fano degeneration $X$ of $V_{22}$, one has $\Pic(X)\simeq \bZ\cdot[-rK_X]$, where $r$ is the Gorenstein index of $X$.
\end{lemma}

\begin{proof}
    This follows directly from Lemma \ref{lem:line-bundle-extension}. %the extension of line bundles in $\bQ$-Gorenstein families of weak Fano varieties (see e.g. \cite[Section 2.7]{Zhu20b}).%\footnote{\YL{add a lemma about extension of line bundles.}}
\end{proof}

\subsection{Moduli of K3 surfaces}

\begin{defn}
  A \emph{K3 surface} is a normal projective surface $S$ with at worst ADE singularities satisfying $\omega_S\simeq \mathcal{O}_S$ and $H^1(S, \mathcal{O}_S) =0$. A \emph{polarization} (resp. \emph{quasi-polarization}) on a K3 surface $S$ is an ample (resp. big and nef) line bundle $L$ on $S$. We call the pair $(S,L)$ a \emph{polarized} (resp. \emph{quasi-polarized}) \emph{K3 surface of degree $d$}, where $d=(L^2)$. Since $d$ is always an even integer, we write $d = 2g-2$ and call $g$ the \emph{genus} of $(S,L)$.
\end{defn}

\begin{defn}
   Let $\Lambda_{\rm K3}\cong U^{\oplus 3}\oplus E_8(-1)^{\oplus 2}$ be a fixed copy of the unique even unimodular lattice of signature $(3,19)$, called the \emph{K3 lattice}.
\end{defn}

Let $(S,L)$ be a polarized K3 surface. Then there are three cases based on the behavior of the linear system $|L|$.

\begin{theorem}[cf. \cite{May72, SD74}]\label{Mayer}
Let $(S,L)$ be a polarized K3 surface of genus $g\geq 3$. Then one of the following holds.%\footnote{\YL{when $g=2$, there is a special case in (3) where $|L|$ has a base point which is the contraction of $E$.}}
\begin{enumerate}
    \item \textup{(Generic case)} The linear series $|L|$ is very ample, and the embedding $\phi_{|L|}:S\hookrightarrow |L|^{\vee}$ realizes $S$ as a degree $2g-2$ surface in $\mathbb{P}^{g}$. In this case, a general member of $|L|$ is a smooth non-hyperelliptic curve.
    \item \textup{(Hyperelliptic case)} The linear series $|L|$ is base-point-free, and the induced morphism $\phi_{|L|}$ realizes $S$ as a double cover of a normal surface of degree $g-1$ in $\mathbb{P}^{g}$. In this case, a general member of $|L|$ is a smooth hyperelliptic curve, and $|2L|$ is very ample.
    \item \textup{(Unigonal case)}  The linear series $|L|$ has a base component $E$, which is a smooth rational curve. The linear series $|L-E|$ defines a morphism $S\rightarrow \mathbb{P}^{g}$ whose image is a rational normal curve in $\mathbb{P}^{g}$. In this case, a general member of $|L-E|$ is a union of disjoint elliptic curves, and $|2L|$ is base-point-free.
\end{enumerate}
\end{theorem}

\begin{defn}\label{defn:moduli of K3}
   For any even integer $d\ge 2$, let $\Lambda_d=\langle \ell_d\rangle\subseteq \Lambda_{\rm K3}$ be the rank-one sublattice generated by a vector $\ell_d$ with $(\ell_d^2)=d$. 
   The \emph{moduli pseudo-functor $\cF_{d}$ of polarized K3 surfaces of degree $d$} assigns to a base scheme $S$ the set of isomorphism classes of pairs
\[
\left\{(f:\mts{X}\rightarrow S;\varphi)\left| \begin{array}{l} \mts{X}\to S\textrm{ is a proper flat morphism, each geometric fiber}\\ \textrm{$\mts{X}_{\bar{s}}$ is an ADE K3 surface, and $\varphi:\Lambda_d\longrightarrow\Pic_{\mts{X}/S}(S)$ is }\\ \textrm{a group homomorphism such that the induced map }\\ \textrm{$\varphi_{\bar{s}}:\Lambda_d\rightarrow \Pic(\mts{X}_{\bar{s}})$ is an isometric primitive embedding of}\\ \textrm{lattices and that $\varphi_{\bar{s}}(\ell_d)\in \Pic(\mts{X}_{\bar{s}})$ is an ample class.} \end{array}\right.\right\}.
\]
\end{defn}

\begin{theorem}[Moduli of polarized K3 surfaces; cf. \cite{Dol96,AE25}]\label{isommoduli}
    The moduli pseudo-functor $\cF_{d}$ is represented by a $19$-dimensional smooth separated Deligne--Mumford stack, still denoted by $\cF_{d}$. 
    Moreover, $\cF_{d}$ admits a normal quasi-projective coarse moduli space $\fF_{d}$ whose analytification is isomorphic to $\mathbb{D}_{d}/\Gamma_d$, where
\[
\mathbb{D}_{d}:=\mathbb{P}\{\,w\in \Lambda_d^{\perp}\otimes\mathbb{C}\mid (w^2)=0,\ (w\cdot\bar w)>0\,\}
\quad\text{and}\quad
\Gamma_d:=\{\gamma_d\in \mathrm{O}(\Lambda_{\rm K3})\mid \gamma|_{\Lambda_d}=\Id_{\Lambda_d}\}.
\]
\end{theorem}

\begin{thm}[Baily--Borel compactification; cf. \cite{BB66}]\label{thm:BB_K3}
There exists a normal projective variety $\ove{\fF}_d^{\BB}$, called the \emph{Baily--Borel compactification}, together with an open immersion
$\fF_d \hookrightarrow \ove{\fF}_d^{\BB}$
such that:
\begin{enumerate}
    \item $\ove{\fF}_d^{\BB}$ is the Proj of the graded ring of $\Gamma_d$-automorphic forms on $\bD_d$, and the inclusion identifies $\fF_d$ with a Zariski open dense subset;
    \item the boundary $\ove{\fF}_d^{\BB}\setminus \fF_d$ is a finite union of locally closed strata of dimension $0$ and $1$;
    \item these strata are in bijection with $\Gamma_d$-equivalence classes of primitive isotropic sublattices of $\Lambda_d^{\perp}$ of rank $1$ (giving $0$-dimensional cusps) and rank $2$ (giving $1$-dimensional cusps).
\end{enumerate}
\end{thm}

\begin{defn}\label{defn:NL divisor}
Let $\Lambda\subseteq \Lambda_{\rm K3}$ be a primitive hyperbolic sublattice containing $\ell_d$. 
The \emph{Noether--Lefschetz locus associated to $(\Lambda,\ell_d)$} is the closed substack 
\[
\cF_{d,\Lambda}\subseteq \cF_d
\]
defined as the closure of the locus of polarized K3 surfaces such that the Picard lattice of its minimal resolution contains $\Lambda$ as a primitive sublattice. If $\rk(\Lambda)=2$ and the Gram matrix of $\Lambda$ with respect to $\ell_d$ and another vector $v$ is
\[
\begin{pmatrix}
d & h\\
h & m
\end{pmatrix},
\]
then $\cF_{d,\Lambda}$ is called a \emph{Noether--Lefschetz divisor} and is denoted by $\cD^d_{h,m}$. 
\end{defn}

\begin{lemma}\label{lem:K3 lattice of Type1-4}
The Gram matrices of a very general anticanonical K3 surface on a one-nodal $V_{22}$ of types {\rm I–IV} are
\[
\Lambda_{\mathrm{I}}=\begin{pmatrix}22&11\\11&4\end{pmatrix},\quad
\Lambda_{\mathrm{II}}=\begin{pmatrix}22&9\\9&2\end{pmatrix},\quad
\Lambda_{\mathrm{III}}=\begin{pmatrix}22&6\\6&0\end{pmatrix},\quad
\Lambda_{\mathrm{IV}}=\begin{pmatrix}22&5\\5&0\end{pmatrix}.
\]
\end{lemma}

\begin{proof}
This follows by direct computation from Prokhorov’s description of the four types. 
For instance, by \Cref{thm:Prokhorov's type I-IV}, a Type~I $V_{22}$ $X$ is the anticanonical model of the blowup $\bP^3$ along a smooth quintic rational curve. Hence a very general anticanonical K3 surface $S\in |-K_X|$ is isomorphic to a quartic K3 surface in $\bP^3$ containing such a curve. The classes $-K_X|_S$ and $\cO_{\bP^3}(1)|_S$ then generate $\Pic(S)$ with intersection form $\Lambda_{\mathrm{I}}$. The remaining cases are analogous.
\end{proof}

\begin{lem}[\textup{\cite[Lemma 1.7]{GLT15}}]\label{lem:equivalent condition of BN general}
Let $(S,L)$ be a primitively polarized ADE K3 surface of degree $22$. Then $(S,L)$ is not Brill--Noether general if and only if it is contained in one of the following eleven NL divisors:
    \[\cD_{1,0}^{22}, \ \  \cD_{2,0}^{22}, \ \ \cD_{3,0}^{22}, \ \  \cD_{4,0}^{22}, \ \ \cD_{5,0}^{22},  \ \ \cD_{6,0}^{22}, \ \
    \cD_{7,2}^{22}, \ \ \cD_{8,2}^{22}, \ \ \cD_{9,2}^{22}, \  \ \cD_{10,4}^{22}, \  \ \cD_{11,4}^{22}.
    \]
\end{lem}

\begin{lemma}\label{lem:line bundle}
Let $\pi: \sS \to T$ be a flat family of ADE surfaces over a smooth pointed curve $0\in T$. Let $\sD$ be a $\bQ$-Cartier Weil divisor on $\sS$ which is Cartier on a general fiber. Then $\sD$ is Cartier. If, moreover, each fiber $\sS_t$ has irregularity $0$, then $\sD_t \sim 0$ for some $t\in T$ if and only if $\sD \sim_T 0$.
\end{lemma}

\begin{proof}

For the first statement, notice that by \cite[Theorem A.1]{HLS24} it suffices to show that $\sD_0$ is Cartier. Up to a finite base change of $(0\in T)$, there exists a simultaneous resolution $\sigma:\widetilde{\sS}\to \sS$ such that fiberwise it is a minimal resolution (see e.g. \cite[Theorem 4.28]{KM98}). As $\sD_t$ is Cartier for a general point $t\in T$, then $\sigma^*\sD$ is still a Weil divisor, and hence Cartier, and so is $\sigma_0^*\sD_0$. Thus $\cD_0$ is Cartier by the following lemma.\begin{lem}
Let $S$ be a klt surface, let $f \colon \tS \to S$ be its minimal resolution, and let $D$ be a $\bQ$-Cartier divisor on $S$. Then $f^*D$ is Cartier if and only if $D$ is Cartier.
\end{lem}

\begin{proof}
One direction is immediate: if $D$ is Cartier, then so is $f^*D$. Conversely, suppose that $f^*D$ is Cartier. Since $f \colon \tS \to S$ is the minimal resolution of a klt surface, the $\bQ$-divisor $\Delta:=f^* K_{S} - K_{\tS}$ is effective and the pair $(\tS, \Delta)$ is klt.
% \[
% -K_{\tS} \ \sim_{\bQ,f}\ \sum a_i E_i =:\Delta,
% \]
%where the $E_i$ are the $f$-exceptional divisors and $0 \leq a_i < 1$. %Choose an effective $f$-exceptional $\bQ$-divisor $G = \sum b_i E_i$ such that $-G$ is $f$-ample and the pair
%\[
%(\tS, \Delta), \quad \text{where } \quad  \Delta \coloneqq G - \sum a_i E_i,
%\]
%Then the pair $(\tS, \Delta)$ is klt as $K_{\tS} + \Delta = f^* K_S$. %Then $f \colon (\tS,\Delta) \to S$ is a klt pair and 
Thus $f^*D$ is a nef Cartier divisor over $S$ such that $f^* D - K_{\tS} -\Delta$ is nef over $S$. By the relative base-point free theorem, it follows that $|b f^*D|$ is base-point free over $S$ for any integer $b \gg 0$. Since $S$ is the relative ample model of $f^*D$ over $S$, we know that $bD$ is Cartier for any $b\gg0$. This implies that $D$ is Cartier.
\end{proof}
For the second statement, consider the relative Picard functor $\Pic_{\sS/T}$, which is represented by an algebraic space locally of finite type over $T$. If $\sD_t \sim 0$ for some $t\in T$, then the corresponding point of $\Pic_{\sS/T}$ lies in the identity over $t$. Since $H^1(\sS_t,\cO_{\sS_t})=0$ for every $t\in T$, the morphism $\Pic_{\sS/T}\to T$ is unramified, and this identity extends uniquely over a Zariski open neighborhood $U$ of $t$, so that $\sD|_{\pi^{-1}(U)} \sim_U 0$. Since $T$ is a smooth curve and $\pi$ is flat, the complement $T\setminus U$ consists of finitely many points, and triviality extends across these fibers. Hence $\sD \sim_T 0$. The converse is immediate.
\begin{comment}
    For the first statement, we first assume that a general fiber of $\pi$ is smooth. Note that $\sS$ has only isolated singularities, all contained in the central fiber $\sS_0$, and these are rational double points. Applying the local Grothendieck--Lefschetz theorem (see \cite[Section~1]{Rob76}) at each singular point shows that every $\bQ$-Cartier Weil divisor on $\sS$ is Cartier. 

In general case, by keep blowing up reduced 1-dimensional singular locus of $\sX$, we get a proper birational morphism $\sigma:\widetilde{\sS}\to \sS$ such that the induced family $\widetilde{\pi}:\widetilde{\sS}\to T$ has smooth general fiber. Then the pullback $\sigma^*\sD$ is also a $\bQ$-Cartier Weil divisor on $\widetilde{\sS}$. By the case where the general fiber is smooth, we conclude that $\sigma^*\sD$ is Cartier on $\widetilde{\sS}$. Then $\sD$ is also Cartier by the relative base-point free theorem applied to the morphism $\sigma$ and the Cartier divisor $\sigma^*\sD$.
\end{comment}
\end{proof}

\subsection{K-stability and K-moduli theory}

\begin{defn}
    A $\mathbb{Q}$-Fano variety $X$ (resp. weak $\bQ$-Fano variety) is called \emph{$\mathbb{Q}$-Gorenstein smoothable} if there exists a projective flat morphism $\pi:\mts{X}\rightarrow T$ over a pointed smooth curve $(0\in T)$ such that the following conditions hold:
    \begin{itemize}
        \item $-K_{\mts{X}/T}$ is $\mathbb{Q}$-Cartier and $\pi$-ample (resp. $\pi$-big and $\pi$-nef);
        \item $\pi$ is a smooth morphism over $T^\circ:=T\setminus \{0\}$; and
        \item $\mts{X}_0\simeq X$.
    \end{itemize}
\end{defn}

\begin{defn}
Let $(X,D)$ be a log Fano pair, and $E$ a prime divisor on a normal projective variety $Y$, where $\pi:Y\rightarrow X$ is a birational morphism. Then the \emph{log discrepancy} of $(X,D)$ with respect to $E$ is $$A_{X,D}(E):=1+\coeff_{E}(K_Y-\pi^{*}(K_X+D)).$$ We define the \emph{S-invariant} of $(X,D)$ with respect to $E$ to be $$S_{X,D}(E):=\frac{1}{(-K_X-D)^n}\int_{0}^{\infty}\vol_Y(-\pi^{*}(K_X+D)-tE)dt,$$ and the \emph{$\beta$-invariant} of $(X,D)$ with respect to $E$ to be $$\beta_{X,D}(E):=A_{X,D}(E)-S_{X,D}(E).$$
\end{defn}

\begin{theodef} \textup{(cf. \cite{Fuj19,Li17,BX19, LWX21})} A log Fano pair $(X,D)$ is 
\begin{enumerate}
    \item K-semistable if and only if $\beta_{X,D}(E)\geq 0$ for any prime divisor $E$ over $X$;
    \item K-stable if and only if $\beta_{X,D}(E)>0$ for any prime divisor $E$ over $X$;
    \item K-polystable if and only if it is K-semistable and any $\mathbb{G}_m$-equivariant K-semistable degeneration of $(X,D)$ is isomorphic to itself.
\end{enumerate}
A weak $\mathbb{Q}$-Fano variety $X$ is \textup{K-(semi/poly)stable} if its anti-canonical model $\ove{X}:=\Proj R(-K_X)$ is K-(semi/poly)stable.

\end{theodef}

The following theorem is usually called the \emph{K-moduli Theorem}, which is attributed to many people (cf. \cite{ABHLX20,BHLLX21,BLX22,BX19,CP21,Jia20,LWX21,LXZ22,Xu20,XZ20,XZ21}).

\begin{theorem}[K-moduli Theorem for Fano varieties]\label{thm:kmoduli}
Fix numerical invariants $n\in \bN$ and $V\in \bQ_{>0}$. Consider the moduli pseudo-functor $\cM^{\K}_{n,V}$ sending a base scheme $S$ to

\[
\left\{\mts{X}/S\left| \begin{array}{l} \mts{X}\to S\textrm{ is a proper flat morphism, each geometric fiber}\\ \textrm{$\mts{X}_{\bar{s}}$ is an $n$-dimensional K-semistable $\bQ$-Fano variety of}\\ \textrm{volume $V$, and $\mts{X}\to S$ satisfies Koll\'ar's condition}\end{array}\right.\right\}.
\]
Then there is an Artin stack, still denoted by $\mtc{M}^{\K}_{n,V}$, of finite type over $\mathbb{C}$ with affine diagonal which represents the moduli functor. The $\mathbb{C}$-points of $\cM^{\K}_{n,V}$ parameterize K-semistable $\mathbb{Q}$-Fano varieties $X$ of dimension $n$ and volume $V$. Moreover, the Artin stack $\cM^{\K}_{n,V}$ admits a good moduli space $\fM^{\K}_{n,V}$, which is a projective scheme, whose $\mathbb{C}$-points parameterize K-polystable $\mathbb{Q}$-Fano varieties.
\end{theorem}

The above K-moduli theorem admits a counterpart for log Fano pairs in full generality; see \cite[Chapter~7]{Xu25}. The definition of families of pairs is rather subtle; since this lies outside the scope of this paper, we refer the interested reader to \cite{Kol23} for a detailed treatment. We therefore state the following special form of the K-moduli theorem for log Fano pairs without giving a precise definition.

\begin{theorem}[K-moduli Theorem for log Fano pairs]\label{thm:kmoduli2}
Fix numerical invariants $n\in \bN$ and $r,V\in \bQ_{>0}$. For any rational number $c\in (0,\frac{1}{r})$, the moduli pseudo-functor $\cP^\K_{n,V,r}(c)$ sending a base scheme $S$ to
\[
\left\{(\mts{X},\cD)/S\left| \begin{array}{l} (\mts{X},c\cD)\to S\textrm{ is a family of log Fano pairs, each geometric fiber}\\ \textrm{$(\mts{X}_{\bar{s}},c\cD_{\bar{s}})$ is an $n$-dimensional K-semistable log Fano pair such}\\ \textrm{that $(-K_{\mts{X}_{\bar{s}}}^n)=V$ and $\cD\sim_{S,\bQ}-rK_{\mts{X}/S}$ is an integral divisor}\end{array}\right.\right\}
\]
is represented by an Artin stack, still denoted by $\cP^\K_{n,V,r}(c)$, of finite type over $\mathbb{C}$ with affine diagonal. Moreover, the stack $\cP^\K_{n,V,r}(c)$ admits a projective good moduli space $\fP^\K_{n,V,r}(c)$.
\end{theorem}

\begin{rem}
To simplify notation, if a pair $(X,cS)$ is parametrized by the K-moduli stack $\cP^\K_{n,V,r}(c)$, we will also say that $(X,S)$ is parametrized by $\cP^\K_{n,V,r}(c)$, and refer to $(X,S)$ as a $c$-K-semistable pair.
\end{rem}

As several different K-moduli stacks will appear throughout the paper, we first introduce the most frequently used ones. For a summary, see \Cref{tab:notations}.

\begin{defn}\label{defn:different K-moduli stacks}
We define the following stacks.
\begin{enumerate}
  \item Let $\mtc{M}^{\K}$ be the reduced closed substack of $\mtc{M}^{\K}_{3,22}$ given by the scheme-theoretic closure of the smooth open substack parametrizing smooth K-semistable $V_{22}$ (cf. \cite[Tag 0509]{stacks-project}). In particular, $\mtc{M}^{\K}$ is an irreducible component of $\mtc{M}^{\K}_{3,22}$.
%\item For any label $\star$ of a smooth family of Fano threefolds of volume $V$ (cf.\ \cite{Fano}), let $\mtc{M}^{\K}_{\textup{\textnumero}\star}$ be the irreducible component of $\mtc{M}^{\K}_{3,V}$ with its reduced stack structure whose general point parametrizes a smooth Fano threefold of family~\textnumero$\star$. In particular, the stack $\cM^{\K}$ defined in (1) coincides with $\cM^{\K}_{\textup{\textnumero 1.10}}$.
%\item Let $\cM^{\Gor}$ be the moduli stack (with reduced stack structure) of Gorenstein canonical Fano threefolds that admit a $\bQ$-Gorenstein smoothing to $V_{22}$. Then $\cM^\K$ is an open substack of $\cM^{\Gor}$ by \Cref{thm:volume bound}(1). Let $\cM^{\term}$ be the open substack of $\cM^{\Gor}$ parametrizing $X$ with terminal singularities, and let $\cM^{\K,\term}$ be $\cM^\K\cap \cM^{\term}$.
    \item Let $\cP^{\K,\ADE}$ be the moduli stack, endowed with the reduced structure, parametrizing pairs $(X,S)$ with $X\in \cM^{\K}$ and $S\in |-K_X|$ an ADE K3 surface. It is the reduced structure of an open substack of $\cP^\K_{3,22,1}(c)$. Let $\fP^{\K,\ADE}$ denote the good moduli space of $\cP^{\K,\ADE}$.
\end{enumerate}
\end{defn}

\begin{comment}

\begin{defn}
Let $x\in X$ be an $n$-dimensional klt singularity. Let $\pi:Y\rightarrow X$ be a birational morphism such that $E\subseteq Y$ is an exceptional divisor whose center on $X$ is $\{x\}$. Then the \emph{volume} of $(x\in X)$ with respect to $E$ is $$\vol_{x,X}(E):=\lim_{m\to \infty}\frac{\dim\mtc{O}_{X,x}/\{f:\ord_E(f)\geq m\}}{m^n/n!},$$ and the \emph{normalized volume} of $(x\in X)$ with respect to $E$ is $$\widehat{\vol}_{x,X}(E):=A_{X}(E)^n\cdot\vol_{x,X}(E).$$ We define the \emph{local volume} of $x\in X$ to be $$\widehat{\vol}(x,X):=\ \inf_{E}\ \widehat{\vol}_{x,X}(E),$$ where $E$ runs through all the prime divisors over $X$ whose center on $X$ is $\{x\}$.
\end{defn}

\begin{theorem}[cf. \cite{Liu18}]\label{volume}
Let $X$ be an $n$-dimensional K-semistable $\bQ$-Fano variety, and $x\in X$ be a point. Then we have the inequality $$(-K_X)^n \ \leq\  \left(\frac{n+1}{n}\right)^n\cdot \widehat{\vol}(x,X).$$
\end{theorem}

\end{comment}

\begin{defn}\label{defn:non-isolated-singularities}
A three dimensional hypersurface singularity $(x\in X)$ is called of \emph{$A_\infty$-type} (resp. \emph{$D_\infty$-type}) if it is locally analytically isomorphic to $
        0\in V(xy-z^2) \subseteq \mathbb{C}^4$ (resp. 
$0\in V(xy - zw^2) \subseteq \mathbb{C}^4$).
\end{defn}

\begin{lemma}\label{lem:conic bundle}
    Let $X$ be a threefold with $D_{\infty}$-singularities at a point $p$ along a curve $C$. Then the exceptional divisor $E$ of $\Bl_{C}X\rightarrow X$ is smooth and is a conic bundle over $C$, and the fiber of $E\rightarrow C$ over $p$ is a reducible conic, i.e. the nodal union of two distinct lines. 
\end{lemma}

\begin{proof}
    We may assume $X = V(x^2+y^2+z^2w)\subset \mathbb C^4_{x,y,z,w}$, $C = V(x,y,z)\simeq \mathbb A^1_w$, and $p=(0,0,0,0)$. We compute the blowup of \(X\) along \(C\). Blow up \(\mathbb C^4\) along the ideal \((x,y,z)\). On the standard affine charts of 
\(\mathrm{Bl}_C(\mathbb C^4)\), the coordinates are:
\begin{itemize}
\item \(U_u:\; y=xv,\ z=xs\),
\item \(U_v:\; x=yu,\ z=ys\),
\item \(U_s:\; x=zu,\ y=zv\).
\end{itemize} Since \(x^2+y^2+z^2w\in (x,y,z)^2\), the strict transform of \(X\) is obtained by dividing by the common square factor. Substituting, we obtain the equations: \[
\widetilde X\cap U_u = V(1+v^2+s^2w),\quad
\widetilde X\cap U_v = V(u^2+1+s^2w),\quad
\widetilde X\cap U_s = V(u^2+v^2+w).
\] Thus the exceptional divisor \(E\subset \widetilde X\) is cut out in 
\(\mathbb P^2_{[u:v:s]}\times C\) by the leading form
\[
E = V(u^2+v^2+s^2w),
\]
so \(E\to C\) is a conic bundle, with smooth fibers for \(w\neq 0\) and a reducible fiber \(u^2+v^2=0\) over \(w=0\). The smoothness of $E$ follows immediately from the Jacobian criterion.
\end{proof}

As seen in the proof, the exceptional divisor of the blowup along $A_{\infty}$-singularities is a ruled surface.

\begin{thm}[{cf. \rmfamily\cites{LX19,LZ25,Liu25}}]\label{thm:volume bound}
Let $X$ be a K-semistable weak $\bQ$-Fano threefold with $\vol(X)\ge 22$. Then the following hold:
\begin{enumerate}
    \item[\textup{(1)}] if $X$ is $\bQ$-Gorenstein smoothable, then it is Gorenstein canonical;
    \item[\textup{(2)}] if $X$ is Gorenstein canonical, then it has either isolated $cA_{\le 2}$-singularities, or $A_{\infty}$-singularities, or $D_{\infty}$-singularities.
\end{enumerate}
\end{thm}

\medskip
\section{Deformations of Fano threefolds with large volume}

In this section, we develop a general deformation framework for K-semistable Fano threefolds $X$ with $\vol(X)\ge 22$. More precisely, we prove:
\begin{itemize}
    \item (\Cref{thm:Kss-very-ample}) the very ampleness of $-K_X$;
    \item (\Cref{thm:beauville}) a Beauville-type result relating deformations of pairs to deformations of K3 surfaces;
    \item (\Cref{thm:smoothable}) an equivalence between the smoothability and the Gorenstein property of $X$; and
    \item (\Cref{thm:sing-line}) the exclusion of singularities along a line.
\end{itemize}

\subsection{Very ampleness of anti-canonical divisors}

\begin{thm}\label{thm:Kss-very-ample}
Let $X$ be a K-semistable Gorenstein canonical Fano threefold of volume $2g-2 \geq 20$. Then $-K_X$ is very ample, and $X$ is projectively normal in $\bP H^0(X,-K_X)$.
\end{thm}

\begin{proof}
Let $S \in |-K_X|$ be a general elephant. Then $(S, -K_X|_S)$ is a polarized K3 surface of genus $g$. Consider the short exact sequence
\[
0 \longrightarrow \mathcal{O}_X 
\longrightarrow \mathcal{O}_X(-K_X) 
\longrightarrow \mathcal{O}_S(-K_X|_S) 
\longrightarrow 0,
\]
which induces a surjection
\[
H^0(X,-K_X) \twoheadrightarrow H^0(S,-K_X|_S)
\]
by Kawamata--Viehweg vanishing. We analyze the geometry according to the behavior of the linear system $|-K_X|_S|$.

\smallskip

\noindent\textbf{Case 1: $(S,-K_X|_S)$ hyperelliptic.}
Then $|-K_X|_S|$ is base-point free, hence so is $|-K_X|$. The morphism
\[
\phi \colon X \longrightarrow \bP^{g+1}
\]
defined by $|-K_X|$ is finite of degree $d \geq 2$. Let $Y \coloneqq  \phi(X)$, which is a non-degenerate threefold of degree $(2g-2)/d \leq g-1$. By \cite{EH87}, the only possibility is $d = 2$, and $Y$ is a variety of minimal degree in $\bP^{g+1}$. By the classification of varieties of minimal degree, $Y$ is either a rational normal scroll or a cone over a scroll of lower dimension.

\smallskip
\noindent \emph{Subcase: $Y$ a rational normal scroll.}
Write $Y = \bP \mathcal{E}$ with projection $\pi \colon Y \to \bP^1$, where
\[
\mathcal{E} \simeq \mathcal{O}_{\bP^1}(a) \oplus \mathcal{O}_{\bP^1}(b) \oplus \mathcal{O}_{\bP^1}(c), 
\qquad a \leq b \leq c, \quad a+b+c = g-1 \geq 10.
\]
Then $c \geq 4$, hence $\mathcal{O}_Y(1) \otimes \pi^*\mathcal{O}_{\bP^1}(-4)$ is effective. In particular, $-K_X \sim_{\bQ} E + 4F$, where $F$ is the pullback of the fiber class of $Y$ and $E$ is some nonzero effective Weil divisors. This, in particular, implies $\alpha(X) \leq \frac{1}{4}.$ Since $X$ is K-semistable, \cite[Theorem 3.5]{FO18} gives $\alpha(X) \geq \frac{1}{4}$, so $\alpha(X) = \frac{1}{4}$. However, by the proof of \cite[Proposition 3.1]{Jia17b}, this forces $-K_X \sim_{\bQ} 4F$, a contradiction.

\smallskip
\noindent \emph{Subcase: $Y$ a cone over a rational normal surface scroll $T$.}
Write $T = \bP \mathcal{E}$ with
\[
\mathcal{E} \simeq \mathcal{O}_{\bP^1}(a) \oplus \mathcal{O}_{\bP^1}(b), 
\qquad a \leq b, \quad a+b = g-1 \geq 10.
\]
The same argument implies $\alpha(X) \leq \frac{1}{5}$, contradicting K-semistability. The case of a cone with higher-dimensional vertex is analogous.

\smallskip

\noindent\textbf{Case 2: $(S,-K_X|_S)$ unigonal.}
We have the commutative diagram
\[
\begin{tikzcd}[ampersand replacement=\&]
S \&\& \bP H^0(S,-K_X|_S) \simeq \bP^g \\
X \&\& \bP H^0(X,-K_X) \simeq \bP^{g+1}.
\arrow["{|-K_X|_S|}", dashed, from=1-1, to=1-3]
\arrow[hook, from=1-1, to=2-1]
\arrow[hook, from=1-3, to=2-3]
\arrow["{|-K_X|}", dashed, from=2-1, to=2-3]
\end{tikzcd}
\]
The image of $S$ in $\bP^g$ is a rational normal curve $R$ of degree $g$, which is a hyperplane section of the image $T$ of $X$ in $\bP^{g+1}$. Thus $T$ is either a cone over $R$ or a rational normal surface scroll. In either case, the same argument as above yields $\alpha(X) \leq \frac{1}{5}$, contradicting K-semistability.

\smallskip

Therefore, $(S, -K_X |_S)$ is neither hyperelliptic nor unigonal, and hence $-K_X|_S$ is very ample, and the morphism defined by $|-K_X|$ is birational. %\AP{Why is it birational?}\textcolor{blue}{This is because $S$ is the inverse image of a hyperplane section, so the degree of the morphism is the same as the degree of the restriction of the morphism on $S$}. 
We now prove that the section ring $R(-K_X) \coloneqq  \bigoplus_{m \geq 0} H^0(X,-mK_X)$ is generated in degree $1$ by induction on $m$. Choose a basis $g_0,\dots,g_N \in H^0(X,-K_X)$ such that $g_0$ defines $S$. The sequence
\[
0 \longrightarrow \mathcal{O}_X 
\stackrel{\cdot g_0}{\longrightarrow} 
\mathcal{O}_X(-K_X) 
\longrightarrow 
\mathcal{O}_S(-K_X|_S) 
\longrightarrow 0
\]
induces
\[
0 \longrightarrow \Sym^m H^0(X,-K_X)
\stackrel{\cdot g_0}{\longrightarrow}
\Sym^{m+1} H^0(X,-K_X)
\longrightarrow 
\Sym^{m+1} H^0(S,-K_X|_S)
\longrightarrow 0.
\] Let $g \in H^0(X,-(m+1)K_X)$ be an arbitrary element. Its restriction $\overline{g}$ lies in $H^0(S,-(m+1)K_X|_S)$, which is a homogeneous polynomial of degree $m+1$ in $\overline{g}_1,\dots,\overline{g}_N$, since $(S,-K_X|_S)$ is projectively normal by \cite[Proposition 2]{May72}. Using the exact sequence
\[
0 \longrightarrow H^0(X,-mK_X)
\stackrel{\cdot g_0}{\longrightarrow}
H^0(X,-(m+1)K_X)
\longrightarrow 
H^0(S,-(m+1)K_X|_S)
\longrightarrow 0,
\]
the induction hypothesis implies that $g$ is a homogeneous polynomial of degree $m+1$ in $g_0,\dots,g_N$. Hence $R(-K_X)$ is generated in degree $1$, so $-K_X$ is very ample and the anticanonical image of $X$ is projectively normal.
\end{proof}
\begin{comment}
\JZ{I commented it out, since changing it is not necessary; but if you prefer to change it, please do, and I am totally fine with it.}

\AP{I guess another (completely equivalent) possibility would be to combine the two short exact sequences at the end of this proof into a morphism of short exact sequences (as shown below)
and then use snake lemma: the vertical arrow on right is surjective for every $m$ because $(S, -K_X |_S)$ is projectively normal, and the vertical arrow on the left is surjective because induction hypothesis; hence the vertical arrow in the middle is surjective. I am happy in both situations.}
\[
\begin{tikzcd}
0 \ar[r]
&
\Sym^m H^0(X,-K_X) \ar[r, "\cdot g_0"] \ar[d]
&
\Sym^{m+1} H^0(X,-K_X) \ar[r] \ar[d]
&
\Sym^{m+1} H^0(S,-K_X|_S) \ar[r] \ar[d]
&
0
\\
0 \ar[r]
&
H^0(X,-mK_X) \ar[r, "\cdot g_0"]
&
H^0(X,-(m+1)K_X) \ar[r]
&
H^0(S,-(m+1)K_X|_S) \ar[r]
&
0
\end{tikzcd}
\]
\end{comment}
\begin{rem}
\begin{enumerate}
    \item It is worth noting that in the above statement we did not assume that $X$ is smoothable. 
    However, if $X$ is smoothable, then the Gorenstein canonical condition is automatically satisfied by \cite{LZ25, Liu25}. Moreover, as we shall see later in Theorem \ref{thm:smoothable}, if $\vol(X)\geq 22$ then $X$ is smoothable.

    \item The very ampleness of $-K_X|_S$ also implies the very ampleness of $-K_X$ in a more direct way. 
    Consider the degeneration $X \rightsquigarrow X_0 \coloneqq  C_p(S, -K_X|_S)$ to the projective cone over $(S,-K_X|_S)$. The very ampleness of $-K_X|_S$ implies that $-K_{X_0}$ is very ample. By Kawamata--Viehweg vanishing, we have $h^0(X,-K_X) = h^0(X_0,-K_{X_0}),$ which shows that very ampleness is an open property in deformation. %\AP{Don't we need also the higher antiplurigenera?} \textcolor{blue}{no because we take the anticanonical cone, so $-K$ is enough}
    \item It is proven in a recent paper \cite[Corollary~1.5]{ACD+} that a Gorenstein canonical Fano threefold with a non-very ample anticanonical divisor is K-semistable only if its volume is less than $14$. The proof relies on the classification of hyperelliptic and unigonal Gorenstein canonical Fano threefolds, together with estimates of stability thresholds.
\end{enumerate}
\end{rem}

\begin{comment}
\begin{lem}
Let $X$ be a Gorenstein canonical Fano threefold. Let $S\in |-K_X|$ be an ADE K3 surface. If $-K_X|_S$ is very ample, so is $-K_X$.
\end{lem}

\begin{proof}
Consider the degeneration $X\rightsquigarrow X_0 \coloneqq  C_p(S, -K_X|_S)$. Since $X_0$ is a projective cone, $-K_X|_S$ being very ample implies that $-K_{X_0}$ is very ample. Then by Kawamata-Viehweg vanishing we know that $h^0(X,-K_X)=h^0(X_0, -K_{X_0})$ which implies that very ampleness is open in deformation. 
\end{proof}

Let $X$ be a Gorenstein canonical Fano threefold with a K3 surface $S\in |-K_X|$. Let $C\in |-K_X|_S|$ be a general curve. We may consider the pencil of K3 surfaces in $|\cO_X(-K_X) \otimes I_{C}|$. Suppose $X'$ is a small deformation of $X$ such that $S\hookrightarrow X'$ as an anti-canonical divisor. 

\end{comment}

\subsection{Beauville type results}

In this subsection, we prove Theorem \ref{thm:beauville}  relating deformation of weak Fano--K3 pairs to deformation of lattice quasi-polarized K3 surfaces, generalizing \cite{Bea04}. This is the technical core of our deformation framework.

\begin{lem}\label{lem:term-elephant}
Let $X$ be a Gorenstein terminal weak Fano threefold of volume $\vol(X)> 2$. Then a general divisor $S\in |-K_X|$ is a smooth K3 surface.
\end{lem}

\begin{proof}
This can be deduced from \cite{Shi89}. For the reader's convenience, we include a proof. If $|-K_X|$ is base-point free, then by Bertini’s theorem we can choose a general divisor $S\in |-K_X|$ that is smooth and avoids the singular locus of $X$. Suppose instead that $|-K_X|$ is not base-point free. Let $S\in |-K_X|$ be a general member, which is an ADE K3 surface by Theorem~\ref{thm:generalelephant}. Since the restriction map
\[
H^0(X,-K_X)\longrightarrow H^0(S,-K_X|_S)
\]
is surjective, the polarized surface $(S,-K_X|_S)$ is a unigonal K3 surface of degree at least $4$. Hence we may write $-K_X|_S \sim E + gF$, where $g\ge 3$ is the genus of $X$, $F$ is the fiber of an elliptic fibration on $S$, and $E$ is a section disjoint from the singular locus of $S$. In particular, the base locus of $|-K_X|$ coincides with that of $|-K_X|_S|$, which is exactly $E$ with its reduced scheme structure. It follows that a general member of $|-K_X|$ is smooth, hence a smooth K3 surface.
\end{proof}

\begin{thm}\label{thm:beauville}
Let $X$ be a Gorenstein terminal weak Fano threefold with $\vol(X)>2$, and let $S \in |-K_X|$ be an ADE K3 surface. Let $\Lambda \subseteq \Pic(S)$ be the saturation of the image of the restriction map $\Pic(X) \to \Pic(S)$. Then the deformation functors $\Def_{(X,S)}$ and $\Def_{(S,\Lambda)}$ are unobstructed, prorepresentable, and admit algebraic miniversal deformation spaces. Moreover, the natural forgetful morphism
\[
\pi : \Def_{(X,S)} \ \longrightarrow \ \Def_{(S,\Lambda)}
\]
is formally smooth with fibers of dimension $h^2(X,\Omega_X^1)$.
\end{thm}

\begin{remark}\label{rmk:beauville}
In Theorem~\ref{thm:beauville}, $\Def_{(X,S)}$ denotes the deformation functor of the pair $(X,S)$ (i.e.\ the closed embedding $S \hookrightarrow X$), and $\Def_{(S,\Lambda)}$ denotes the deformation functor of $S$ together with a $\bZ$-basis $L_1,\dots,L_r$ of the lattice $\Lambda$. To define the forgetful morphism $\pi$, by Smith normal form we may choose $L_1,\dots,L_r$ and positive integers $a_1,\dots,a_r$ such that $L_1^{\otimes a_1},\dots,L_r^{\otimes a_r}$ form a $\bZ$-basis of $\im\big(\Pic(X)\to \Pic(S)\big)$. Thus there exist $M_1,\dots,M_r \in \Pic(X)$ with $M_i|_S = L_i^{\otimes a_i}$.

Given a deformation $(\sX,\sS)/B$ of $(X,S)$, Lemma~\ref{lem:line-bundle-extension} shows that each $M_i$ extends uniquely to $\sX/B$, hence $L_i^{\otimes a_i}$ extends to $\sS/B$. Since the obstruction to extending $L_i^{\otimes a_i}$ is a multiple of that for $L_i$, it follows that $L_i$ also extends; uniqueness follows from $H^1(S,\cO_S)=0$. Therefore, $\pi$ sends $(\sX,\sS)/B$ to $\sS/B$ together with the extended line bundles $L_1,\dots,L_r$.
\end{remark}

\begin{proof}[Proof of Theorem~\ref{thm:beauville}]
Denote by $\bL_X$ and $\bL_S$ the cotangent complexes of $X \to \Spec \bC$ and of $S \to \Spec \bC$, respectively. Let $\bL_X(\log S)$ be the logarithmic cotangent complex of the pair $(X,S)$, i.e.\ the cotangent complex $\bL_\epsilon$ of the morphism $\epsilon : X \to [\bA^1/\bG_m]$ induced by the effective Cartier divisor $S$. Denote by $\iota \colon S \hookrightarrow X$ the inclusion.
Consider the morphisms
\begin{equation*}
    S \ \overset{\iota}{\hookrightarrow} \ X \ \overset{\epsilon}{\longrightarrow} \ [\bA^1/\bG_m] \ \overset{p}\longrightarrow \ \rB \bG_m \ \longrightarrow\ \Spec \bC
\end{equation*}
where $p$ is induced by $\bA^1 \to \Spec \bC$.
The composition $p \circ \epsilon$ is induced by the the line bundle $\cO_X(S) = \omega_X^\vee$.
By \cite[Appendix~B.2]{chp} $L \epsilon^* \bL_{[\bA^1 / \bG_m]}$ is isomorphic to $\iota_* \cO_S[-1]$.
Thus, by shifting the conormal distinguished triangle 
\[
L \epsilon^* \bL_{[\bA^1 / \bG_m]} \ \longrightarrow \  \bL_X  \ \longrightarrow \  \bL_\epsilon  \ \longrightarrow \  L \epsilon^* \bL_{[\bA^1 / \bG_m]}[1]
\]
of $\epsilon$, we obtain the distinguished triangle
\begin{equation}\label{eq:Beauville-tria1}
    \bL_X  \ \longrightarrow \  \bL_X(\log S)  \ \longrightarrow \ \iota_*\cO_S  \ \longrightarrow \ \bL_X[1]
\end{equation}
which should be thought of as a generalization of the residue sequence.

By tensoring \eqref{eq:Beauville-tria1} with $\cO_X(-S)$ and by shifting, we obtain the distinguished triangle
\begin{equation} \label{eq:beauville_triangle_shifted_residue}
    \iota_* \cO_S(-S)[-1]  \ \longrightarrow \  \bL_X(-S)  \ \longrightarrow \  \bL_X(\log S)(-S)  \ \longrightarrow \  \iota_* \cO_S(-S).
\end{equation}
Since $\iota$ is a regular embedding, it is well known that $\bL_\iota$ is isomorphic to $\iota^* \cO_X(-S)[1]$. Hence, by considering the conormal triangle of $S \hookrightarrow X \to \Spec \bC$, by pushing it forward to $X$ and by shifting, we obtain the distinguished triangle
\begin{equation} \label{eq:beauville_conormal_iota}
    \iota_* \cO_S(-S)[-1]  \ \longrightarrow \  \iota_* L \iota^* \bL_X [-1]  \ \longrightarrow \ \iota_* \bL_S[-1]  \ \longrightarrow \  \iota_* \cO_S(-S).
\end{equation}
By derived-tensoring the sequence
\[
0  \ \longrightarrow \  \cO_X(-S)  \ \longrightarrow \  \cO_X  \ \longrightarrow \  \iota_* \cO_S  \ \longrightarrow \  0
\]
by $\bL_X$ and then shifting, we obtain the distinguished triangle
\begin{equation}
    \label{eq:beauville_easy_triangle}
    \iota_* L \iota^* \bL_X [-1]  \ \longrightarrow \  \bL_X(-S)  \ \longrightarrow \  \bL_X  \ \longrightarrow \  \iota_* L \iota^* \bL_X.
\end{equation}
By considering the distinguished triangles \eqref{eq:beauville_easy_triangle}, \eqref{eq:beauville_triangle_shifted_residue}, 
\eqref{eq:beauville_conormal_iota},
the octahedron axiom gives a distinguished triangle
\begin{equation}\label{eq:Beauville-tria2}
    \bL_X(\log S)(-S)  \ \longrightarrow \  \bL_X  \ \longrightarrow \  \iota_*\bL_S  \ \longrightarrow \  \bL_X(\log S)(-S)[1].
\end{equation}
The two distinguished triangles \eqref{eq:Beauville-tria1} and \eqref{eq:Beauville-tria2} will be crucial to study the two forgetful maps
\[\begin{tikzcd}[ampersand replacement=\&]
\& \Def_{(X,S)} \\
\Def_X \& \& \Def_S
\arrow["\pi", from=1-2, to=2-3]
\arrow[ from=1-2, to=2-1]
\end{tikzcd}
\]

We start with showing that the forgetful map $\Def_{(X,S)}\to \Def_X$ is formally smooth. Since $\Def_X$ is unobstructed by \cite{Nam97, Min01,San18}, this will imply that $\Def_{(X,S)}$ is unobstructed. Applying $\RHom(\,\cdot\,,\cO_X)$ to \eqref{eq:Beauville-tria1} yields the exact sequence
\begin{equation}\nonumber
\begin{tikzpicture}[descr/.style={fill=white,inner sep=1.5pt}]
  \matrix (m) [
    matrix of math nodes,
    row sep=1em,
    column sep=2.5em,
    text height=1.5ex,
    text depth=0.25ex
  ]
  {
    \Ext^1(\bL_X(\log S),\cO_X) & \Ext^1(\bL_X,\cO_X) & \Ext^2(\iota_*\cO_S,\cO_X) \\
    \Ext^2(\bL_X(\log S),\cO_X) & \Ext^2(\bL_X, \cO_X). & {} \\
  };
  \path[overlay,->,>=latex]
    (m-1-1) edge (m-1-2)
    (m-1-2) edge (m-1-3)
    (m-1-3) edge[out=355,in=175]  (m-2-1)
   (m-2-1) edge (m-2-2);
\end{tikzpicture}
\end{equation}
By Serre duality, we obtain
\[
\Ext^2(\iota_*\cO_S,\cO_X)
 \ \simeq \ \Ext^1(\cO_X,\cO_S\otimes \omega_X)^\vee
\ \simeq \ H^1(S,\omega_X^{-1}|_S)^\vee
=0,
\]
where the last equality follows from Kawamata–Viehweg vanishing.
%\AP{I would not use Serre duality, I would use the following:
%\begin{align*}
%\Ext^2(\iota_*\cO_S,\cO_X) &= h^2 \ R \Gamma \ R \cH om (\iota_*\cO_S,\cO_X) \\
%&\simeq h^2 \ R \Gamma \ R \cH om (\omega_X \to \cO_X,\cO_X) \\
%&\simeq h^2 \ R \Gamma (\cO_X \to \omega_X^\vee) \\
%&\simeq h^2 \ R \Gamma (\omega_X^\vee \vert_S[-1]) \\
%&\simeq H^1 (S, \omega_X^\vee \vert_S).
%\end{align*}
%}\textcolor{blue}{I think Serre duality is faster and clearer.}
%\AP{OK! :)}
This vanishing implies that the map $\Def_{(X,S)}\to \Def_X$ is formally smooth, as the map between tangent spaces (resp.\ between obstruction spaces) is surjective (resp.\ injective). We note that this is true even under the weaker assumption that $X$ is a Gorenstein canonical weak Fano threefold, despite the fact that $\Def_X$ or $\Def_{(X,S)}$ may be obstructed.

To show that $\Def_{(X,S)}$ is prorepresentable, it suffices to check $\Aut(X,S)$ is finite; see e.g. \cite[Theorem 2.6.1]{Ser06}. Let $X'$ be the anticanonical model of $X$, and let $S'$ be the image of $S$. Then $(X',S')$ is a plt log Calabi–Yau pair with $X'$ a Gorenstein canonical Fano threefold. Any automorphism of $(X,S)$ descends to an automorphism of $(X',S')$, so we have an inclusion
\[
\Aut(X,S)\hookrightarrow \Aut(X',S').
\]
By \cite[Theorem~2.10]{ADL21}, $\Aut(X',S')$ is finite, hence $\Aut(X,S)$ is finite. 

Since $H^2(X,\cO_X)=0$ by Kawamata–Viehweg vanishing, the miniversal deformation of $X$ is effective and algebraizable by \cite[Theorems~2.5.13(ii) and 2.5.14]{Ser06}. By \Cref{lem:line-bundle-extension}, the line bundle $\cO_X(S)\simeq \omega_X^{\vee}$ extends uniquely after an \'etale base change. Therefore $\Def_{(X,S)}$ is algebraizable via the relative linear system over the algebraic miniversal deformation of $X$.

We now consider $\Def_{(S,\Lambda)}$. Since $X$ is projective, the ample cone of $S$ intersects $\Lambda_\bR$ nontrivially. Hence there exists a very irrational ample class $h\in \Lambda_\bR$ (cf. \cite[Definition 4.1]{AE25}), and the unobstructedness and algebraizability of $\Def_{(S,\Lambda)}$ follow from \cite[Proof of Theorem~5.5]{AE25}.

Next, we study the forgetful map
\[
\pi:\Def_{(X,S)}\ \longrightarrow \ \Def_S
\]
and its tangent map
\[
d\pi:\Ext^1(\bL_X(\log S),\cO_X)
\ \longrightarrow \  \Ext^1(\bL_S,\cO_S).
\]
By applying $\RHom(\,\cdot\,,\omega_X)$ to \eqref{eq:Beauville-tria2} and by using $\omega_X\simeq \mtc{O}_X(-S)$ and the identification
\[
\Ext^k_X(\iota_*\bL_S,\omega_X)
\ \simeq \ \Ext^{k-1}_S(\bL_S,\omega_S) \ \simeq \ \Ext^{k-1}_S(\bL_S,\cO_S)
\] by Grothendieck duality, we get the exact sequence
\begin{equation}\label{eq:Beauville-seq1-diagram}
\begin{tikzpicture}[descr/.style={fill=white,inner sep=1.5pt}]
  \matrix (m) [
    matrix of math nodes,
    row sep=1em,
    column sep=2.5em,
    text height=1.5ex,
    text depth=0.25ex
  ]
  {
    \Hom_S(\bL_S,\cO_S) & \Ext^1(\bL_X,\omega_X) & \Ext^1(\bL_X(\log S),\cO_X) \\
    \Ext^1_S(\bL_S,\cO_S) & \Ext^2(\bL_X,\omega_X). & {} \\
  };
  \path[overlay,->,>=latex]
    (m-1-1) edge (m-1-2)
    (m-1-2) edge (m-1-3)
    (m-1-3) edge[out=355,in=175] node[descr] {$d\pi$} (m-2-1)
   (m-2-1) edge node[descr] {$\partial$} (m-2-2);
\end{tikzpicture}
\end{equation}
Since $X$ is a Gorenstein terminal threefold and $S$ is an ADE surface, both have lci singularities, and hence the cotangent complexes $\bL_X$ and $\bL_S$ have Tor amplitude $[-1,0]$. Let $\cE_X:=\cH^{-1}(\bL_X)$ and $\cE_S:=\cH^{-1}(\bL_S)$, which are coherent sheaves supported on the singular loci, which are unions of finitely many points. By Serre duality, for $i\leq 1$ and $j\leq 2$ one has
\[
\Ext^i(\cE_S,\cO_S) \ \simeq  \ H^{2-i}(S,\cE_S)^{\vee} \ =\ 0,
\qquad
\Ext^j(\cE_X,\omega_X)\ \simeq  \ H^{3-j}(X,\cE_X)^{\vee} \ =\ 0,
\]
and hence, by using the distinguished triangles
\[
\cE_X[1]  \ \longrightarrow \  \bL_X  \ \longrightarrow \  \Omega_X^1  \ \longrightarrow \ \cE_X[2], \qquad \cE_S[1]  \ \longrightarrow \  \bL_S  \ \longrightarrow \  \Omega_S^1  \ \longrightarrow \ \cE_S[2],
\]
we obtain
\[
\Ext^i(\bL_S,\cO_S)\ \simeq \ \Ext^i(\Omega_S^1,\cO_S),
\qquad
\Ext^j(\bL_X,\omega_X)\ \simeq \ \Ext^j(\Omega_X^1,\omega_X).
\] Let $\sigma:\tS\to S$ be the minimal resolution. Since $T_S\simeq \sigma_*T_{\tS}$, one has 
\[ \Hom(\Omega^1_S,\cO_S)\ \simeq \  
H^0(S,T_S)\ \simeq \  H^0(S, \sigma_*T_{\tS}) \ \simeq \ H^0\big(\tS,T_{\tS}\big)\ =\ 0,
\]
and the sequence \eqref{eq:Beauville-seq1-diagram} reduces to
\[
0\ \longrightarrow \ 
\Ext^1(\Omega_X^1,\omega_X)
\ \longrightarrow \ 
\Ext^1(\bL_X(\log S),\cO_X)
\ \stackrel{d\pi}{\longrightarrow} \ 
\Ext^1(\Omega_S^1,\cO_S)
\ \stackrel{\partial}{\longrightarrow} \ 
\Ext^2(\Omega_X^1,\omega_X).
\] 
Hence we have
\[
\im(d\pi) \;=\; \Ker(\partial) \;=\; \im(\partial^\vee)^{\perp},
\]
where
\[
\partial^\vee: H^1(X,\Omega_X^1) \ \longrightarrow \ H^1(S,\Omega_S^1)
\]
denotes the dual of $\partial$ under Serre duality.

Let $(\sX,\sS)\to (0\in T)$ be an algebraic miniversal deformation of $(X,S)$ over a smooth pointed variety. After possibly shrinking $T$, we may assume that $T$ is irreducible and that each fiber $(X_t,S_t)$ is a Gorenstein terminal weak Fano threefold together with an anticanonical ADE K3 surface. For each $t\in T$, denote by $\Lambda_t \subseteq  \Pic(S_t)$ the saturation of the image of the restriction map $\Pic(X_t)\to \Pic(S_t)$. For simplicity, we write $\pi_t$, $\partial_t$, and $\partial_t^\vee$ for the corresponding maps obtained by replacing $(X,S)$ with $(X_t,S_t)$. In particular, the tangent map of the forgetful morphism
\[
\pi_t:\Def_{(X_t,S_t)} \ \longrightarrow \  \Def_{S_t}
\]
is given by
\[
d\pi_t:\Ext^1(\bL_{X_t}(\log S_t),\cO_{X_t})\ 
\longrightarrow\ 
\Ext^1(\bL_{S_t},\cO_{S_t}).
\] Again, the target of $d\pi_t$ is isomorphic to $\Ext^1(\Omega^1_{S_t},\cO_{S_t})$, which is a vector space of dimension $20$ by \cite{BW74}. Thus we have
\begin{equation}\label{eq:beauville_dimIm+dimIm_constant}
    \dim \im(d\pi_t) + \dim\im (\partial_t^\vee) = 20
\end{equation}
for each $t \in T$.

\begin{lemma}\label{lem:const-rank}
The map $d\pi_t$ has constant rank for all $t\in T$.
\end{lemma}

\begin{proof}
Since both $\Def_{(X,S)}$ and $\Def_S$ are unobstructed smooth germs, by taking minors of Jacobians it follows that the function
\[
t \;\longmapsto\; \dim \im(d\pi_t)
\]
is lower semicontinuous on $T$. Therefore, by \eqref{eq:beauville_dimIm+dimIm_constant} it suffices to show that
\[
\dim \im(\partial_t^\vee) \;\ge\; \dim \im(\partial^\vee)
\]
for general $t\in T$. 

We first consider the case where $S$ is smooth. Consider the commutative diagram
\[
\begin{tikzcd}
\Omega_X^1 \arrow[r] \arrow[d] &
\iota_*\Omega_S^1 \arrow[d] \\
\underline{\Omega}_X^1 \arrow[r] &
\iota_*\underline{\Omega}_S^1,
\end{tikzcd}
\]
where $\underline{\Omega}_X^p$ and $\underline{\Omega}_S^p$ denote the $p$-th Du Bois complex of $X$ and $S$ respectively; see \cite{DB81}. Since $\iota$ is a closed immersion, the pushforward functor $\iota_*$ is exact on coherent sheaves, and hence
\[
\rm R\iota_* \underline{\Omega}_S^1 \ \simeq \  \iota_* \underline{\Omega}_S^1.
\]
As $S$ is smooth, we have $\underline{\Omega}_S^1  \simeq \Omega_S^1$, and therefore
\[
\iota_* \underline{\Omega}_S^1 \ \simeq \ \iota_* \Omega_S^1.
\]
\begin{lem}\label{lem:h11-Pic}
Let $Y$ be a weak $\bQ$-Fano variety. Then $\bH^1(Y, \uOmega_Y^1) \simeq H^2(Y, \bC) \simeq \Pic(Y)\otimes_\bZ \bC$.
\end{lem}

\begin{proof}
Since $Y$ has rational singularities and $H^2(Y,\cO_Y)=0$ by Kawamata--Viehweg vanishing, we know that $\bH^0(Y, \uOmega_Y^2) = \bH^2(Y, \uOmega_Y^0)=0$. Thus the result follows from the Hodge--Du Bois decomposition \[
H^2(Y,\bC) \ \simeq\  \bH^0(Y,\uOmega_Y^2) \oplus \bH^1(Y,\uOmega_Y^1) \oplus \bH^2(Y,\uOmega_Y^0),
\] the exponential sequence, and the fact that $H^i(Y,\cO_Y) = 0$ for $i>0$ by Kawamata--Viehweg vanishing.
\end{proof}
By Lemma \ref{lem:h11-Pic} we have that
\[
\im(\partial^\vee)
\;\subseteq\;
\im\!\bigl(\bH^1(X,\underline{\Omega}_X^1)\to H^1(S,\Omega_S^1)\bigr)
\;\cong\;
\Lambda_{\bC}.
\] On the other hand, for general $t\in T$, it follows from \cite[Main Theorem (2)]{Min01} that $X_t$ is nodal. In particular, $X_t$ is $1$-Du Bois by \cite[Theorem 1.1]{MOPW23}, i.e. $\Omega_{X_t}^1 \simeq \underline{\Omega}_{X_t}^1$. Therefore, by Lemma \ref{lem:h11-Pic}
\[
\im(\partial_t^\vee)
\;=\;
\im\!\bigl(\bH^1(X_t,\underline{\Omega}_{X_t}^1)
\to
H^1(S_t,\Omega_{S_t}^1)\bigr)
\;\cong\;
\Lambda_{t,\bC}.
\]
Since line bundles on $X$ extend to line bundles on $X_t$ after an \'etale base change by \Cref{lem:line-bundle-extension},  we obtain
\begin{equation}\label{eq:beauville-sm}
\dim \im(\partial^\vee)
\;\le\;
\dim \Lambda_{\bC}
\;\le\;
\dim \Lambda_{t,\bC}
\;=\;
\dim \im(\partial_t^\vee),
\end{equation} where to get the inequality in the middle we use \Cref{lem:line bundle}.
This concludes the proof of Lemma~\ref{lem:const-rank} when $S$ is smooth.

Now consider the case where $S$ has ADE singularities. By \Cref{lem:term-elephant}, there exists a family
\[
f:(X\times B,\sS_B)\to B
\]
over a smooth pointed curve $(0\in B)$, obtained from a general deformation of $(X,S)=(X,S_0)$, such that for each $b\in B\setminus\{0\}$, the fiber $S_b\coloneqq (\sS_B)_b$ is a smooth anticanonical K3 surface in $X$. As $\Omega_{\sS_B/B}^1$ is flat over $B$ by \cite[Theorem 2.5]{FL24} and $h^1(S_b,\Omega_{S_b}^1)=20$ for all $b\in B$, by Grauert's theorem, the sheaf $R^1 f_* \Omega_{\sS_B/B}^1$ is locally free, and its fiber over $b$ is naturally identified with $H^1(S_b,\Omega_{S_b}^1)$. Since $R^1 f_* \Omega_{X\times B/B}^1$ is locally free, the maps $\partial_b^\vee$ assemble into a morphism of locally free sheaves
\[
R^1 f_* \Omega_{X\times B/B}^1
\ \longrightarrow\ 
R^1 f_* \Omega_{\sS_B/B}^1.
\]
In particular, the function
\[
b \ \longmapsto\ \dim \operatorname{Im}(\partial_b^\vee)
\]
is lower semicontinuous on $B$. Therefore, for a general point $b \in B$, one has
\[
\dim \operatorname{Im}(\partial^\vee)
\;\le\;
\dim \operatorname{Im}(\partial_b^\vee).
\] Combining this inequality with the smooth case yields
\[
\dim \operatorname{Im}(\partial^\vee)
\;\le\;
\dim \operatorname{Im}(\partial_t^\vee)
\]
for general $t \in T$. This completes the proof of Lemma~\ref{lem:const-rank}.
\end{proof}
We have now shown that $d\pi_t$ has constant rank for $t\in T$. Since the image of the forgetful map lies in $\Def_{(S,\Lambda)}$ by extension of line bundles on weak Fano varieties, as noted by Remark~\ref{rmk:beauville}, it follows that 
\[
\operatorname{Im}(d\pi) \ \subseteq\  T^1_{(S,\Lambda)} \ \coloneqq \ T_{[(S,\Lambda)]}\Def_{(S,\Lambda)}.
\]
It remains to prove that $\dim \operatorname{Im}(d\pi)
=
\dim T^1_{(S,\Lambda)}$, which would imply $\operatorname{Im}(d\pi)=T^1_{(S,\Lambda)}$. For $t\in T$ such that $S_t$ is smooth, the proof of \Cref{lem:const-rank} (in particular the equalities in \eqref{eq:beauville-sm}) gives
\[
\dim \operatorname{Im}(d\pi)
\ =\ 
\dim \operatorname{Im}(d\pi_t)
\ =\ 
20 - \dim \operatorname{Im}(\partial_t^\vee)
 \ =\ 
20 - \dim \Lambda_{t,\bC}
\ =\ 
\dim T^1_{(S_t,\Lambda_t)}.
\]
Moreover, by \cite[Proof of Theorem~5.5]{AE25}, we know that
\[
\dim T^1_{(S,\Lambda)} \ =\  20 - \dim \Lambda_{\bC}.
\]
Thus it suffices to show that under a one-parameter smoothing $(X\times B, \sS_B)\to B$ of $(X,S)$ %\AP{(which exists by Lemma~\ref{lem:term-elephant})}
, the lattices $\{\Lambda_b\}_{b\in B}$ form a local system over $B$. This follows from Lemma~\ref{lem:line bundle}, which shows that the kernels of the restriction maps
\[
\Pic(X)\to \Pic(S)
\qquad\text{and}\qquad
\Pic(X)\to \Pic(S_b)
\]
coincide. Hence $\Lambda_b$ has constant rank under deformation.

Finally, since $\pi$ is smooth, its fiber dimension equals $\dim \Ker(d\pi)$. By Serre duality, this dimension is
\[
\dim \Ext^1(\Omega_X^1,\omega_X)
=
h^2(X,\Omega_X^1).
\]
This concludes the proof of Theorem~\ref{thm:beauville}.
\end{proof}

Set $h^{1,2}(X)\coloneqq h^2(X,\Omega_X^1)$. As a direct consequence of \Cref{thm:beauville}, we prove the invariance of $h^{1,2}$ in families of Gorenstein terminal weak Fano threefolds.

%\AP{I believe that the expression $\bQ$-Gorenstein in the two corollaries below is useless because the fibres of $\sX \to T$ are Gorenstein. Am I right? Maybe, we can put parentheses around ``$\bQ$-Gorenstein''.}\textcolor{blue}{if you take base $T$ to be only Q-Gor not Gorenstein, then fiber Gor. ter. does not imply total space Gorenstein}
%\AP{I think it depends on the definition of $\bQ$-Gorenstein morphism: I think we should allow $T$ to be arbitrary, even non-reduced. Am I right?
%And I would call a morphism $\sX \to T$ $\bQ$-Gorenstein if it is flat, of finite presentation, with normal geometric fibers (this is enough for us) and \'etale-locally induced by a deformation of the index-1 cover of the fiber which is equivariant with respect to the action of the cyclic group of order equal to the Gorenstein index.
%I think this is written badly, but I hope it is clear what I mean.
%} \textcolor{blue}{It makes sense, I think I meant the following: by $\bQ$-Gorenstein family, we mean a locally stable family in the sense of Kollar, not only saying that fiberwise is Gorenstein. This is a standard(?) terminology I guess? maybe we should add it in section 2.}

%\AP{What I wanted to say is that any flat (finitely presented) morphism with Gorenstein geometric fibers is necessarily a $\bQ$-Gorenstein family, hence writing ``$\bQ$-Gorenstein'' is a bit useless.}

\begin{cor}\label{cor:h12}
Let $\pi \colon \sX \to T$ be a family of Gorenstein terminal weak Fano threefolds of volume $>2$. Then the function $t\mapsto h^{1,2}(\sX_t)$ is locally constant on $T$.
\end{cor}

The following corollary is of significant interest in its own right, though we shall only invoke a weaker form of it later in the text.

\begin{cor}\label{cor:lattice-local-system}
Let $\pi:(\sX,\sS)\to T$ be a family of Gorenstein terminal weak Fano threefolds of volume $>2$, together with an anticanonical ADE K3 surface. Then there exists an \'etale locally constant subsheaf of $\Pic_{\sS/T}$ whose fiber over each closed point $t\in T$ is the saturation of the image of the restriction map
\[
\Pic(\sX_t)\longrightarrow \Pic(\sS_t).
\]
\end{cor}

\begin{proof}
We view $\Pic_{\sS/T}$ as an \'etale sheaf on the small \'etale site $T_{\textup{\'et}}$. 
For any closed point $t\in T$, there exists an \'etale neighborhood $U_t \to T$ such that every line bundle in the saturation $\Lambda_t \subseteq \Pic(\sS_t)$ of the image of $\Pic(\sX_t)\to \Pic(\sS_t)$ extends to $\sX|_{U_t}$ by \Cref{rmk:beauville}. 
This defines a locally constant subsheaf $\Lambda_{U_t} \subseteq \Pic_{\sS|_{U_t}/U_t}$.

We claim that for every $t'\in U_t$, the fiber $(\Lambda_{U_t})_{t'} \subseteq \Pic(\sS_{t'})$ is saturated. 
Suppose not. Then there exist $t'\in U_t$, a line bundle $L_{\sS_{t'}}\in \Pic(\sS_{t'})$, and an integer $m\ge 2$ such that $L_{\sS_{t'}}^{\otimes m}$ lies in $(\Lambda_{U_t})_{t'}$, while $L_{\sS_{t'}}$ does not. 
By \Cref{rmk:beauville}, the line bundle $L_{\sS_{t'}}$ extends to an \'etale neighborhood of $t'$, and hence defines a Weil divisor $\cL$ on $\sS|_{U_t}$ by taking closure. 
Since $L_{\sS_{t'}}^{\otimes m}$ extends to a Cartier divisor on $\sS|_{U_t}$ and $\Pic_{\sS/T}\to T$ is unramified, it follows that $\cL$ is $\bQ$-Cartier. 
By \Cref{lem:line bundle}, $\cL$ is in fact Cartier, hence defines a line bundle on $\sS|_{U_t}$, contradicting the saturation of $\Lambda_t$.

Now let $t_1,t_2\in T$ and set $U:=U_{t_1}\cap U_{t_2}$. 
For any closed point $t_0\in U$, the fibers of $\Lambda_{U_{t_1}}$ and $\Lambda_{U_{t_2}}$ at $t_0$ coincide. 
Indeed, both are saturated subgroups of $\Pic(\sS_{t_0})$, and by the proof of  \Cref{thm:beauville}, the rank of the image of $\Pic(\sX_t)\to \Pic(\sS_t)$ is locally constant on $T$. 
Therefore the two subsheaves $\Lambda_{U_{t_1}}|_U$ and $\Lambda_{U_{t_2}}|_U$ agree inside $\Pic_{\sS|_U/U}$.

It follows that the family $\{\Lambda_{U_t}\}_{t\in T}$ glues to a globally defined \'etale subsheaf $\Lambda_T \subset \Pic_{\sS/T}$. 
By construction, $\Lambda_T$ is locally constant, and its fiber over each $t\in T$ is precisely the saturation of the image of $\Pic(\sX_t)\to \Pic(\sS_t)$.
\end{proof}

\subsection{Non-isolated singularities}
In this subsection, we study the deformation theory of K-semistable Fano threefolds with non-isolated singularities. 
Although the deformation theory of Fano threefolds with non-terminal singularities is generally subtle (see e.g.\ \cite{Pet20, KP21, Pet22, chp}), we show that such varieties are smoothable when the volume is large.

\begin{thm}\label{thm:smoothable}
Let $X$ be a K-semistable Gorenstein canonical Fano threefold with $\vol(X) \geq 22$. Then $X$ is $\bQ$-Gorenstein smoothable. If, in addition, $X$ is not terminal, then one of the following holds:
\begin{enumerate}
    \item[\textup{(1)}] $X$ admits a small deformation to a singular Gorenstein terminal Fano threefold; or
    \item[\textup{(2)}] $X$ admits a smoothing whose general fiber has strictly higher Picard rank.%\footnote{\YL{In this case, one can further show that the corresponding smoothing component of the moduli stack is smooth at $X$ by Beauville.}}.
\end{enumerate}
\end{thm}

%\AP{In (2) can we put parentheses around $\bQ$-smoothing?}

We begin with an exclusion of singularities along a \emph{line}, i.e. a smooth rational curve of degree 1 with respect to the anticanonical divisor.

\begin{thm}\label{thm:sing-line}
Let $X$ be a Gorenstein canonical Fano threefold with $\vol(X)\geq 22$ that is singular along a line $\ell$. Then $X$ is K-unstable.
\end{thm}

\begin{proof}
Suppose for contradiction that $X$ is K-semistable. Then by \Cref{thm:volume bound}, $X$ has either $A_{\infty}$- or $D_{\infty}$-singularities along $\ell$, and by \Cref{thm:Kss-very-ample}, the divisor $-K_X$ is very ample.

Let $\mu \colon \wt{X} \to X$ be the blowup along $\ell$, with exceptional divisor $E$. Then $\mu$ is crepant, and by \Cref{lem:conic bundle}, $E$ is smooth and the morphism $E \to \ell$ is a conic bundle. Since $(-K_X \cdot \ell) = 1$, the curve $\ell$ is a line in $\bP H^0(X,-K_X) \simeq \bP^{g+1}$, where $g\coloneqq g(X)\geq 12$. Consequently, $-K_{\wt{X}} - E$ is base-point free, as it is the restriction of a base-point free divisor on $\Bl_{\ell} \bP^{g+1}$. In particular, $-K_E = (-K_{\wt{X}} - E)|_E$ is nef. Thus $E$ is isomorphic to $\bF_n$ for some $n \leq 2$, or to a blowup thereof, and hence $h^0(E,-K_E) \leq 9$. We have the commutative diagram
\[
\begin{tikzcd}[ampersand replacement=\&]
\wt{X} \&\& \bP^{g-1} \\
X \&\& \bP^{g+1}
\arrow["{|-K_{\wt{X}}-E|}", from=1-1, to=1-3]
\arrow["\mu"', from=1-1, to=2-1]
\arrow["{|-K_X|}"', hook, from=2-1, to=2-3]
\arrow["{\pi_{\ell}}"', dashed, from=2-3, to=1-3]
\end{tikzcd}
\]
where $\pi_{\ell}$ is projection from the line $\ell$. Consider the short exact sequence
\[
0 \ \longrightarrow \ \mtc{O}_{\wt{X}}(-K_{\wt{X}} - 2E)
\ \longrightarrow\  \mtc{O}_{\wt{X}}(-K_{\wt{X}} - E)
 \ \longrightarrow\  \mtc{O}_E(-K_E)
\ \longrightarrow\ 0,
\]
which induces the left-exact sequence
\[
0 \ \longrightarrow\ H^0(\wt{X}, -K_{\wt{X}} - 2E)
\ \longrightarrow \ H^0(\wt{X}, -K_{\wt{X}} - E)
\ \longrightarrow  \ H^0(E, -K_E).
\] Since
\[
h^0(E,-K_E)\ \leq \ 9, 
\qquad
h^0(\wt{X}, -K_{\wt{X}} - E)
\ \geq \  h^0(\bP^{g+1}, \mtc{I}_{\ell}(1)) \ =\ g,
\]
we deduce that $h^0(\wt{X}, -K_{\wt{X}} - 2E) \geq g-9\geq 3$, and hence the pseudo-effective threshold satisfies $\tau(-K_X;E) \geq 2$. Note that $-K_{\wt{X}} - tE$ is nef for $0 \leq t \leq 1$, and for $1 \leq t \leq 2$ we have
\[
-K_{\wt{X}} - tE 
\ =\  (2-t)(-K_{\wt{X}} - E) + (t-1)(-K_{\wt{X}} - 2E).
\] Let $\mtf{f}$ denote the fiber class of the conic bundle $E \to \ell$. Using $(-K_X \cdot \ell)=1$, one computes
\[
(-K_{\wt{X}} - E)|_E \sim -K_E, 
\qquad
E|_E \sim K_E + \mtf{f}, 
\qquad
-K_{\wt{X}}|_E \sim \mtf{f}.
\] Therefore the intersection numbers are
\[
(-K_{\wt{X}}^2 \cdot E) \ =\  0, 
\ \ \ 
(-K_{\wt{X}} \cdot E^2) \ =\  (K_E + \mtf{f}) \cdot \mtf{f} \ =\  -2,
\ \ \ 
(E^3) \ =\ (K_E+\mtf{f})^2\ =\ K_E^2 - 4 \ =\  4 - k,
\]
where $k$ is the number of points on $\ell$ at which $X$ has a $D_{\infty}$-singularity. Hence
\[
(-K_{\wt{X}} - tE)^3 \ =\  (2g-2) - 6t^2 - (4-k)t^3,
\]
and therefore
\[
\int_0^1 \vol(-K_X - tE)\, dt
\ =\  \int_0^1 \left(2g-2 - 6t^2 - (4-k)t^3\right)\, dt
\ =\  2g-5 + \frac{k}{4}.
\] For $1 \leq t \leq 2$, we have
$\vol(-K_X - tE) \geq (2-t)^3 (-K_{\wt{X}} - E)^3$, so
\[
\int_1^2 \vol(-K_X - tE)\, dt 
\ \geq\  (2g-12+k)\int_1^2 (2-t)^3\, dt
\ =\    \frac{2g-12+k}{4}.
\] Thus
\[
S_X(E) \ \geq\  1+ \frac{g-12+k}{4(g-1)} \ \geq\  1 \ =\  A_X(E).
\]
Since $X$ is assumed K-semistable, equality must hold, hence $g=12$ and $k=0$, and $X$ has only $A_{\infty}$-singularities along $\ell$. Moreover, all the inequalities above become equalities, and hence one has
\[
\vol(-K_{\tX} - tE) \ =\ \begin{cases}
    22 - 6t^2 -4t^3 & \textrm{when }t\in [0,1];\\
    12(2-t)^3 & \textrm{when }t\in [1,2].
\end{cases}  
\]
Thus we have 
\[
\left.\frac{d}{dt}\right|_{t=1^-} \vol(-K_{\tX} - tE)\ =\ -24\ \neq\  -36\  =\   \left.\frac{d}{dt}\right|_{t=1^+} \vol(-K_{\tX} - tE).
\]
This shows that $\vol(-K_{\tX}-tE)$ is not $C^1$ at $t=1$, a contradiction to \cite{BFJ09}. 
% for every $t\in [1,2]$ by the continuity of the volume function. Denote by $A\coloneqq -K_{\tX}-E$ and $B\coloneqq  -K_{\tX}-2E$. Then we have $\vol(A+tB) = \vol(A)$ for any $t\geq 0$. Note that $A$ is big as $$\big((-K_{\wt{X}}-E)^3\big) \ =\  2g-2-6-(4-k) \ =\  2g-12 + k \ >\ 0$$ Thus by \cite[Theorem A]{BFJ09} we have 
% \[
% 0 \ =\  \left.\frac{d}{dt}\right|_{t=0}\vol(A+tB)  \ =\ 3(\langle A^2\rangle. B) \ =\   3  (A^2 . B)\footnote{Can contradict by computing intersection numbers.},
% \]
% where the last equality follows from the nefness of $A$ and \cite[Proposition 2.12]{BFJ09}. Let $\phi: \tX \to X'$ be the ample model of $A$, which is a birational morphism. Then $(A^2 . B) = 0$ implies that any effective divisor in $|B|$ must be $\phi$-exceptional. In particular, we have $h^0(\tX, B) \leq 1$, contradicting to our previous inequality $h^0(\tX, B) \geq 3$.
\end{proof}

%\footnote{An $A_{\infty}$-singularity along a line is not smoothable based on Fantechi and Petracci; need to decide whether to add it or not}

\begin{proof}[Proof of \Cref{thm:smoothable}]
By \Cref{thm:volume bound}(2), $X$ has only isolated $cA_{\leq 2}$-singularities, $A_\infty$-singularities, or $D_\infty$-singularities. If $X$ is terminal, then it is $\bQ$-Gorenstein smoothable by \cite{Nam97}. Hence we assume that $X$ is not terminal. Then $\dim X_{\sing} = 1$. Let $C = \bigsqcup_i C_i $ be the one-dimensional singular locus of $X$, where each $C_i$ is a smooth curve. Let $\mu\colon \tX \to X$ be the blow-up of $C$, which is a terminalization, and let $E_i$ denote the exceptional divisor over $C_i$. By Theorem~\ref{thm:Kss-very-ample}, the linear system $|-K_X|$ is very ample, hence $|-K_{\tX}|$ is base-point free. Let $S\in |-K_X|$ be a general K3 surface. Then the strict transform $\tS\subseteq \tX$ is a smooth K3 surface by Bertini's theorem.

Applying Theorem~\ref{thm:beauville} to the pair $(\tX,\tS)$, we see that a very general deformation $(\tX_t,\tS_t)$ satisfies that $\tS_t$ is a very general deformation of $\tS$ in the moduli stack of $\Lambda$-quasi-polarized K3 surfaces (cf. \cite[Definition 4.2]{AE25}), where $\Lambda\subseteq \Lambda_{\rm K3}$ is the saturation of 
\[
\mathrm{Im}\bigl(\Pic(\tX) \to \Pic(\tS)\bigr).
\] Let $d_i \coloneqq  (-K_X\cdot C_i)$; by Theorem~\ref{thm:sing-line}, we have $d_i \geq 2$. Since $S$ is a general element of the base-point free linear system $|-K_{X}|$, we may assume that $S$ intersects each $C_i$ transversally at $d_i$ points $(p_{i,j})_{j=1}^{d_i}$ and $X$ has only $A_\infty$-singularities at each $p_{i,j}$.
Then $\tS$ contains exactly $d_i$ exceptional curves $(e_{i,j})_{j=1}^{d_i}$ lying over $C_i\cap S$, whose images $(p_{i,j})_{j=1}^{d_i}$ in $S$ are $A_1$-singularities. Since $E_i \to C_i$ is a conic bundle with a smooth fiber $e_{i,j}$ over each $p_{i,j}$ by \Cref{lem:conic bundle}, we conclude that the curve class $[e_{i,j}]\in N_1(\tX)$ is independent of the choice of $j$. 
% Denote by $X^{\circ}\subset X$ the open $\bQ$-factorial locus. Since $X$ is a klt threefold, we know that $X\setminus X^{\circ}$ is a finite set. In particular, we have $S\subset X^{\circ}$. Let $\tX^{\circ}:= \mu^{-1}(X^{\circ})\subset \tX$ so that $\tS\subset \tX^{\circ}$. Since every fiber of $\mu$  has dimension $\leq 1$, we know that $\tX^{\circ}$ is a big open subset of $\tX$. Then it is clear that 
% \[
% \Pic(\tX^{\circ})_{\bQ} = \mu^* \Pic(X^{\circ})_{\bQ}  \oplus \bigoplus_i \bQ[E_i|_{\tX^{\circ}}].
% \]
% Since
% \[
% \Cl(\tX) \ =\  \mu^*\Cl(X) \oplus \bigoplus_i \bZ[E_i],
% \]
Thus for any $i$ and any $L\in \Pic(\tX)$ (hence any $\beta\in \Lambda$), the intersection number $(L\cdot e_{i,j})$ (hence $(\beta\cdot e_{i,j})$) is independent of the choice of $j$. %In particular, for any $\beta\in \Lambda$ and any $i$, the intersection number $(\beta\cdot e_{i,j})$ is independent of $j$. 
%Therefore, for each $i$, the curves $e_{i,j}$ do not deform individually into a very general small deformation $\widetilde{S}_t$ whenever $d_i \ge 2$; rather, the sum of the curve classes $\sum_{j} [e_{i,j}]$ deforms as a whole. 

Let $X_t$ be the anticanonical model of $\widetilde{X}_t$. We will show that $X_t$ is terminal. If not, there exists a prime divisor $E_t\subset \tX_t$ that is contracted by the morphism $\widetilde{X}_t \to X_t$. Then, by \Cref{thm:volume bound}(2) and \Cref{thm:sing-line}, the image of $E_t$ is a smooth curve of degree $d \ge 2$. Consequently, $\widetilde{S}_t \cap E_t$ is a disjoint union of $d$ rational curves $e_{t,1}, \dots, e_{t,d}$. By the very generality of $\widetilde{S}_t$, one has $\Pic(\widetilde{S}_t) \simeq \Lambda$, and moreover $\Pic(\widetilde{S}_t)_{\bQ}$ is generated by the image of $\Pic(\widetilde{X}_t)_{\bQ}$ by \Cref{cor:lattice-local-system}. However, this is impossible, since by the same argument as above, the curves $e_{t,1}, \dots, e_{t,d}$ have identical intersection numbers with any line bundle $L_t \in \Pic(\widetilde{X}_t)$. Therefore, the morphism $\widetilde{X}_t \to X_t$ is small. Since $\widetilde{X}_t$ is Gorenstein terminal, the same holds for $X_t$. It follows from \cite{Nam97} that $X_t$ is $\bQ$-Gorenstein smoothable, and hence so is $X$.

For the final statement, we may assume that $X_t$ is smooth. Then $\tX_t \simeq X_t$, and $|-K_{X_t}|$ is very ample by Theorem~\ref{thm:Kss-very-ample}. By \Cref{thm:NL-Cl}, for a very general $\tS_t\in |-K_{\tX_t}|$ we have
\[ \Lambda  \ \simeq\ 
\Pic(\tS_t) \ \simeq\ \Pic(\tX_t) \ \simeq\  \Pic(X_t).
\]
 On the other hand, for a very general $S$, it follows again from \Cref{thm:NL-Cl} that
\[
\rk \Pic(X_t)
\ =\ 
\rk(\Lambda)
\ >\ 
\rk\!\bigl(\mathrm{Im}(\Pic(X)\to \Pic(S))\bigr)
\ =\ 
\rk \Pic(X).
\]
This completes the proof.
\end{proof}

\begin{rem}\label{cor:small deformation small anti-can morphism}
In the proof of \Cref{thm:smoothable}, it follows from \Cref{cor:h12} that $h^{1,2}(\widetilde{X}_t)=h^{1,2}(\widetilde{X})$. Since $\widetilde{X}_t \to X_t$ is small and $X_t$ is Gorenstein terminal, the induced morphism from $\widetilde{S}_t$ to its image in $X_t$ is an isomorphism. Let $\widetilde{\Lambda}_t$ (resp.~$\Lambda_t$) denote the saturation in $\Lambda_{\mathrm{K3}}$ of 
\[
\mathrm{Im}\bigl(\Pic(\widetilde{X}_t)\to \Pic(\widetilde{S}_t)\bigr)
\quad
\bigl(\text{resp.~}\mathrm{Im}\bigl(\Pic(X_t)\to \Pic(\widetilde{S}_t)\bigr)\bigr).
\]
Then we obtain the following commutative diagram:
\[
\begin{tikzcd}[ampersand replacement=\&]
	\Def_{(\widetilde{X}_t,\widetilde{S}_t)} \&\& \Def_{(\widetilde{S}_t,\widetilde{\Lambda}_t)} \\
	\Def_{(X_t,\widetilde{S}_t)} \&\& \Def_{(\widetilde{S}_t,\Lambda_t)}
	\arrow["{\widetilde{\phi}_t}", from=1-1, to=1-3]
	\arrow["{\psi_t}", from=1-1, to=2-1]
	\arrow["{\widetilde{\psi}_t}", from=1-3, to=2-3]
	\arrow["{\phi_t}", from=2-1, to=2-3]
\end{tikzcd},
\] where $\psi_t$ is defined by taking anticanonical ample models; see e.g.\ \cite[Corollary~2.16]{San17}. 
Since Mori dream spaces admit only finitely many minimal models (with respect to any divisor), it follows that, after passing to algebraic miniversal deformation spaces, $\psi_t$ is quasi-finite. 
Moreover, $\widetilde{\psi}_t$ is injective, and by \Cref{thm:beauville} both $\phi_t$ and $\widetilde{\phi}_t$ are smooth, of relative dimensions $h^{1,2}(X_t)$ and $h^{1,2}(\widetilde{X}_t)$ respectively. In particular,
\[
h^{1,2}(X_t)\ \ge\  h^{1,2}(\widetilde{X}_t) \ = \ h^{1,2}(\tX).
\]
\end{rem}

As a direct consequence of Theorem \ref{thm:smoothable} we have the following result.

\begin{cor}
Let $X$ be a K-semistable $\bQ$-Fano threefold with $V:=\vol(X) \geq 22$. Then the following conditions are equivalent.
\begin{enumerate}[label=(\roman*)]
    \item $X$ is Gorenstein canonical;
    \item $X$ is $\bQ$-Gorenstein smoothable;
    \item $X$ does not admit quotient singularities of type $\frac{1}{2}(1,1,1)$.
\end{enumerate}
Moreover, the locus of $X$ satisfying one (and hence all) of the above conditions is both open and closed in $\cM_{3,V}^{\K}$.
\end{cor}

\begin{proof}
By Theorem \ref{thm:smoothable} we have (i) implies (ii). The direction that (ii) implies (iii) follows from the rigidity of isolated quotient singularities \cite{Sch71}. By \cite[Theorem 1.3]{Liu25} we have (iii) implies (i).  Finally, the last statement holds as $\bQ$-Gorenstein smoothability is a closed condition, while being Gorenstein canonical is an open condition.
\end{proof}

The following example shows that the condition $\vol(X) \geq 22$ in \Cref{thm:smoothable} is nearly optimal.

\begin{exam}\label{ex:nonsmoothable_Kps_volume_18}

There exists a K-polystable toric Gorenstein canonical Fano threefold $X_0$ of volume $18$, corresponding to the spanning fan of the triangular prism, or equivalently, the normal fan of the dual bipyramid polytope (see \Cref{fig:polytopes}). 
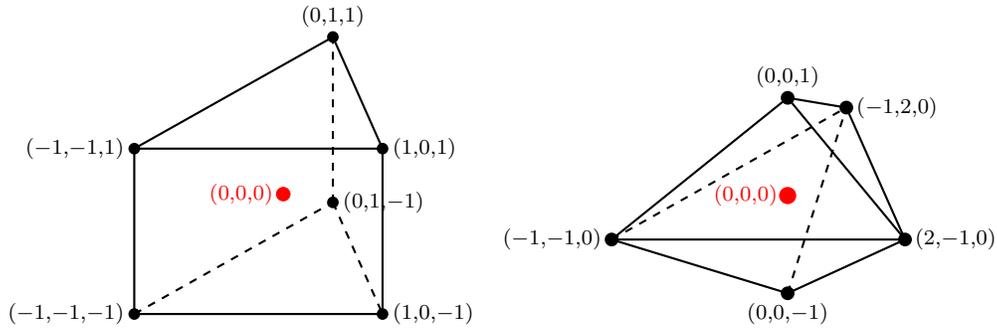
\begin{figure}[H]
\centering
\begin{tikzpicture}[
  scale=1.1,
  x={(1cm,0cm)}, y={(0.8cm,0.45cm)}, z={(0cm,1cm)}
]

% vertices (triangle x [-1,1])
\coordinate (A) at (0,0,-1);
\coordinate (B) at (3,0,-1);
\coordinate (C) at (0,3,-1);

\coordinate (D) at (0,0, 1);
\coordinate (E) at (3,0, 1);
\coordinate (F) at (0,3, 1);

% centroid
\coordinate (G) at (1,1,0);

% edges
\draw[thick] (A)--(B);
\draw[thick, dashed] (C)--(B);
\draw[thick, dashed] (A)--(C);

\draw[thick] (D)--(E);
\draw[thick] (D)--(F);
\draw[thick] (E)--(F);

\draw[thick] (A)--(D);
\draw[thick] (B)--(E);
\draw[thick, dashed] (C)--(F);

% points
\fill (A) circle (2pt);
\fill (B) circle (2pt);
\fill (C) circle (2pt);
\fill (D) circle (2pt);
\fill (E) circle (2pt);
\fill (F) circle (2pt);

% centroid (red)
\fill[red] (G) circle (2.5pt);
\node[red, left] at (G) {$\scriptstyle (0,0,0)$};

% labels
\node[left]  at (A) {$\scriptstyle (-1,-1,-1)$};
\node[right] at (B) {$\scriptstyle (1,0,-1)$};
\node[right] at (C) {$\scriptstyle (0,1,-1)$};

\node[left]  at (D) {$\scriptstyle (-1,-1,1)$};
\node[right] at (E) {$\scriptstyle (1,0,1)$};
\node[above] at (F) {$\scriptstyle (0,1,1)$};
\end{tikzpicture}
\centering
\begin{tikzpicture}[
  scale=1.3,
  x={(1cm,0cm)}, y={(0.8cm,0.45cm)}, z={(0cm,1cm)}
]

% vertices
\coordinate (A) at (2,-1,0);
\coordinate (B) at (-1,2,0);
\coordinate (C) at (-1,-1,0);

\coordinate (D) at (0,0,1);
\coordinate (E) at (0,0,-1);
\coordinate (G) at (0,0,0);

% edges of base triangle
\draw[thick] (A)--(B);
\draw[thick, dashed] (C)--(B);
\draw[thick] (A)--(C);

% edges to top apex
\draw[thick] (D)--(A);
\draw[thick] (D)--(B);
\draw[thick] (D)--(C);

% edges to bottom apex
\draw[thick] (E)--(A);
\draw[thick, dashed] (E)--(B);
\draw[thick] (E)--(C);

% vertices
\fill (A) circle (2pt);
\fill (B) circle (2pt);
\fill (C) circle (2pt);
\fill (D) circle (2pt);
\fill (E) circle (2pt);
\fill[red] (G) circle (2.5pt);
\node[red, left] at (G) {$\scriptstyle (0,0,0)$};
% labels
\node[right] at (A) {$\scriptstyle (2,-1,0)$};
\node[right]  at (B) {$\scriptstyle (-1,2,0)$};
\node[left]  at (C) {$\scriptstyle (-1,-1,0)$};

\node[above] at (D) {$\scriptstyle (0,0,1)$};
\node[below] at (E) {$\scriptstyle (0,0,-1)$};
\end{tikzpicture}
\caption{The triangular prism (left) and its dual bipyramid (right).}
\label{fig:polytopes}
\end{figure}
\noindent The toric threefold $X_0$ can be realized as %the complete intersection in $\bP^5$ defined by $xyz = w^3$ and $uv = w^2$ \AP{this is false, I think}, or equivalently as 
the weighted hypersurface
\[
V\big(y_0y_1 - (x_0x_1x_2)^2\big) \ \subseteq \   \bP(1,1,1,3,3).
\]
It has two isolated quotient singularities of type $\frac{1}{3}(1,1,1)$ and generically $A_{\infty}$-singularities along a cycle of three $\bP^1$'s. In particular, $X_0$ is not $\bQ$-Gorenstein smoothable. On the other hand, $X_0$ is K-polystable, since the barycenter of the weight polytope (the red point in \Cref{fig:polytopes}) is the origin. This example shows that the volume bound in \Cref{thm:smoothable} is close to optimal.

Notice that $X_0$ deforms to a general weighted hypersurface $X_t \subset \bP(1,1,1,3,3)$ of class $\cO(6)$, which can be realized as a double cover of $\bP \coloneqq \bP(1,1,1,3)$ branched along an anticanonical divisor $S_t \in |-K_{\bP}|$. Such a threefold $X_t$ has two isolated quotient singularities of type $\frac{1}{3}(1,1,1)$. By \cite[Theorem 1.2(2)]{LZ22} and \cite{Zhu21}, the threefold $X_t$ is K-polystable if and only if the pair $\big(\bP,\frac{1}{2}S_t\big)$ is K-polystable. Viewing $\bP$ as the projective cone $C_p(\bP^2,\cO_{\bP^2}(3))$, the divisor $S_t$ is a double cover of $\bP^2$ branched along a plane sextic curve $C_t$. By \cite[Theorem 5.2]{ADL21}, the pair $\big(\bP,\frac{1}{2}S_t\big)$ is K-polystable if and only if the pair $\big(\bP^2,\frac{1}{8}C_t\big)$ is K-polystable. 

Furthermore, by \cite[Theorem 1.5]{ADL19}, the K-moduli space of pairs $\big(\bP^2,\frac{1}{8}C_t\big)$ is isomorphic to the GIT moduli space of plane sextic curves $|\cO_{\bP^2}(6)|^{\rm ss}\sslash \PGL(3)$,
which is also the GIT moduli space of degree $2$ K3 surfaces. Consequently, the K-moduli space of weighted hypersurfaces in $\bP(1,1,1,3,3)$ of class $\cO(6)$ is isomorphic to this GIT quotient.

\end{exam}

\subsection{K-moduli of pairs and forgetful maps}

In this subsection, we establish a general framework for studying the forgetful morphism from the K-moduli of pairs to the moduli of K3 surfaces.

Fix a deformation family of smooth Fano threefolds (not necessarily containing a K-semistable member), and denote its number by \textnumero$\star$ and its anticanonical volume by $V$. Let $\cM^\K_{\textup{\textnumero}\star}$ denote the K-moduli stack of Fano threefolds of type \textnumero$\star$ (which may be empty), and let $\cM^{\Gor}_{\textup{\textnumero}\star}$ denote the moduli stack of Gorenstein canonical degenerations of smooth Fano threefolds in the family \textnumero$\star$. Since these Fano threefolds form a bounded family and the Gorenstein canonical condition is open in $\bQ$-Gorenstein families, it follows that $\cM^{\Gor}_{\textup{\textnumero}\star}$ is an Artin stack of finite type over $\bC$; see e.g. \cite[Theorem 7.36]{Xu25}. For any rational number $c\in (0,1)$, let $\cP_{\textup{\textnumero}\star}^{\K}(c)$ be the irreducible component of $\cP^{\K}_{3,V,1}(c)$, endowed with its reduced stack structure, whose general point parametrizes a smooth Fano threefold of family \textnumero$\star$ together with a smooth anticanonical K3 surface. Let $\fP_{\textup{\textnumero}\star}^{\K}(c)$ be the corresponding irreducible component of the K-moduli space $\fP^{\K}_{3,V,1}(c)$. Let $(X,cS)$ be a log smooth pair parametrized by $\cP_{\textup{\textnumero}\star}^{\K}(c)$. Denote by $h^{1,2}\coloneqq h^{1,2}(X)$ the third Betti number of $X$, by $r$ the Fano index of $X$, and fix a Cartier divisor $H$ such that $-K_X\sim rH$. For any line bundle $L\in \Pic(X)$, set $L_S\coloneqq L|_S$.

\begin{lem}\label{lem:primitivity}
    The class $c_1(L_S)$ is primitive in $H^2(S,\bZ)$.
\end{lem}

\begin{proof}
If $c_1(L_S)$ were not primitive in $H^2(S,\bZ)$, then the image of $        i^*:\Pic(X)\simeq H^2(X,\bZ)\to H^2(S,\bZ)$ would not be saturated, so $\coker(i^*)$ would have torsion. 
    However, as $S\in |-K_X|$ is a smooth ample divisor, the integral Lefschetz hyperplane theorem implies that $\coker(i^*)$ is torsion free; see \cite[Example 3.1.18]{Laz04a}. 
    Hence $c_1(L_S)$ is primitive in $H^2(S,\bZ)$.
\end{proof}

Let $\Lambda$ be $\im(H^2(X,\bZ)\rightarrow H^2(S,\bZ))$, which is a primitive sublattice of $H^2(S,\bZ)\simeq \Lambda_{\rm K3}$, $h\in \Lambda$ be $c_1(H_S)$, and $d$ be $(h^2)$. Let $\cF_{d,\Lambda}$ be the Noether--Lefschetz locus of $\cF_{d}$ associated to $\Lambda$, $\fF_{d,\Lambda}$ be its coarse moduli space, and $\ove{\fF}^{\BB}_{d,\Lambda}$ be its Baily--Borel compactification, i.e. the closure of $\fF_{d,\Lambda}$ in $\ove{\fF}^{\BB}_d$.

Let $\cP^{\Kst}_{\textup{\textnumero}\star}(1-\epsilon)$ be the open substack of 
$\cP^{\K}_{\textup{\textnumero}\star}(1-\epsilon)$ parametrizing K-stable pairs 
$\big(X,(1-\epsilon)S\big)$, and let $\cP^{\plt}_{\textup{\textnumero}\star}$ be the open substack of 
$\cP^{\K}_{\textup{\textnumero}\star}(1-\epsilon)$ parametrizing pairs 
$\big(X,(1-\epsilon)S\big)$ such that $(X,S)$ is plt. Then $\cP^{\plt}_{\textup{\textnumero}\star}$ is contained in  $\cP^{\Kst}_{\textup{\textnumero}\star}(1-\epsilon)$ by \cite[Theorem 2.10]{ADL21}. Let 
$\cP^{\ter}_{\textup{\textnumero}\star}$ (resp. $\cP^{\can}_{\textup{\textnumero}\star}$) be the open substack, by the inversion of adjunction, of 
$\cP^{\plt}_{\textup{\textnumero}\star}$ parametrizing pairs 
$\big(X,(1-\epsilon)S\big)$ for which $X$ is Gorenstein terminal (resp. Gorenstein canonical) and $S$ is an ADE K3 surface. 
Then $\cP^{\ter}_{\textup{\textnumero}\star}$ is a smooth Deligne--Mumford stack. 

\renewcommand{\arraystretch}{1.5}
\begin{longtable}{| >{\centering\arraybackslash}p{.18\textwidth} 
                  | >{\centering\arraybackslash}p{.20\textwidth} 
                  | >{\centering\arraybackslash}p{.16\textwidth} 
                  | >{\centering\arraybackslash}p{.16\textwidth} 
                  | >{\centering\arraybackslash}p{.16\textwidth} |}
    \hline 
    $\cP^{\ter}_{\textup{\textnumero}\star}$ & $\cP^{\can}_{\textup{\textnumero}\star}$ & $\cP^{\plt}_{\textup{\textnumero}\star}$ & $\cP^{\Kst}_{\textup{\textnumero}\star}(1-\epsilon)$ & $\cP^{\K}_{\textup{\textnumero}\star}(1-\epsilon)$ \\ \hline 
    $X$ Gorenstein terminal, $S$ ADE & $X$ Gorenstein canonical, $S$ ADE & $(X,S)$ plt & $(X,(1-\epsilon)S)$ K-stable & $(X,(1-\epsilon)S)$ K-semistable \\ \hline 
  
    \caption{Inclusions of open substacks}
    \label{tab:Inclusion of stacks}
\end{longtable}

\begin{prop}\label{prop:plt=Gor can}
For any $(X,S)\in \cP^{\plt}_{\textup{\textnumero}\star}$, the variety $X$ is Gorenstein canonical. In other words, one has $\cP^{\can}_{\textup{\textnumero}\star}=\cP^{\plt}_{\textup{\textnumero}\star}$.%\footnote{\YL{we should add one thing: every Gorenstein canonical degeneration together with an ADE K3 lies in $\cP^{\can}$; in other words, the forgetful map $\cP^{\can} \to \cM^{\Gor}$ is surjective. This is a quick consequence of \cite[Theorem 2.10]{ADL21}}}
\end{prop}

\begin{proof}
Let $(\sX,\sS)\rightarrow (0\in T)$ be a $\bQ$-Gorenstein smoothing of $(X,S)$ over a smooth pointed curve. By \cite[Lemma~2.11]{ABB+}, one has $
K_X+S\sim 0$, and hence $K_X$ is Cartier away from $S$. By \cite[Theorem~A.1]{HLS24}, it therefore suffices to show that $-K_X|_S=-K_{\sX}|_S$ is Cartier. As $\sX$ is smooth in codimension $2$, then $-K_{\sX}|_{\sS}$ is a $\bQ$-Cartier Weil divisor, and hence $-K_{\sX}|_{\sS}$ is Cartier by \Cref{lem:line bundle}.
\end{proof}

\begin{prop}\label{prop:extension to BB compactification}
The forgetful map $\cP^{\K}_{\textup{\textnumero}\star}(1-\epsilon)\dashrightarrow \cF_{d,\Lambda}$
extends to a surjective proper morphism
\[
\fP^{\K}_{\textup{\textnumero}\star}(1-\epsilon)^{\nu} \longrightarrow \overline{\fF}^{\BB}_{d,\Lambda}
\]
from the normalization of the K-moduli space to the Baily--Borel compactification of the Noether--Lefschetz locus.
\end{prop}

\begin{proof}
By \Cref{prop:plt=Gor can}, the rational map $\cP^{\K}_{\textup{\textnumero}\star}(1-\epsilon)\dashrightarrow \cF_{d,\Lambda}$ is regular on $\cP^{\plt}_{\textup{\textnumero}\star}$. Hence it suffices to consider pairs $(X,S)$ that are strictly log canonical, i.e.\ such that $S$ is strictly (semi-)log canonical. By \cite[Lemma~3.18]{AET23}, it is enough to prove the following statement: fix a pair $(X,S)$; for any one-parameter family of $(1-\epsilon)$-K-semistable pairs $(\sX,\sS)\to (0\in C)$ over a smooth pointed curve whose central fiber is $(X,S)$, the associated Baily--Borel limit depends only on $(X,S)$. More precisely, the central fiber determines whether the limit is of Type~II or Type~III, and in the Type~II case the $j$-invariant of the associated elliptic curve is uniquely determined.

Let $(\sX,\sS)\to (0\in C)$ be such a family, and denote by $(\sX^\circ,\sS^\circ)\to C^\circ$ its restriction over the punctured curve $C^\circ:=C\setminus\{0\}$. For every $c\in C^\circ$, the fiber $\sX_c$ is Gorenstein canonical and $\sS_c$ is an ADE K3 surface. After a possible finite base change, we can take a Kulikov model $\sS^*\to C$ of $\sS^\circ\to C^\circ$. Then $(\sS,\sS_0)\to C$ and $(\sS^*,\sS^*_0)\to C$ are two birational \emph{crepant log structures} in the sense of \cite[Definition~2]{Kol16}, and in both cases the closed point $\{0\}\subseteq C$ is the unique \emph{lc center} of $C$. By \cite[Theorem~1]{Kol16}, the crepant birational equivalence classes of minimal lc centers of $(\sS,\sS_0)$ and $(\sS^*,\sS^*_0)$ coincide.

In particular, if a (hence every) minimal lc center of $S=\sS_0$ is a point, then the Kulikov model $\sS^*\to C$ is of Type~III. If instead a (hence every) minimal lc center $Z$ of $S=\sS_0$ is a curve, then $Z$ is birational to the minimal lc center of $\sS^*_0$, which is an elliptic curve. Consequently, $\sS^*\to C$ is of Type~II, and its $j$-invariant is determined by the birational class of $Z$.
\end{proof}

\begin{cor}\label{cor:surjective forgetful maps}
There exist natural forgetful morphisms
\[
\begin{tikzcd}[ampersand replacement=\&]
\cM^{\Gor}_{\textup{\textnumero}\star} \& \cP^{\plt}_{\textup{\textnumero}\star} \& \cF_{d,\Lambda}
\arrow[from=1-2, to=1-1]
\arrow[from=1-2, to=1-3]
\end{tikzcd}
\]
both of which are surjective. Moreover, $\cP^{\plt}_{\textup{\textnumero}\star} \rightarrow\cF_{d,\Lambda}$ is proper.
\end{cor}

\begin{proof}
The existence of both morphisms follows immediately from \Cref{prop:plt=Gor can}. For any $X \in \cM^{\Gor}_{\textup{\textnumero}\star}$, there exists an ADE K3 surface $S \in |-K_X|$ by \Cref{thm:generalelephant}. Then the pair $(X,S)$ is plt by inversion of adjunction, and hence $\big(X,(1-\epsilon)S\big)$ is K-stable by \cite[Theorem 2.10]{ADL21}. This proves the surjectivity of the morphism $\cP^{\plt}_{\textup{\textnumero}\star} \longrightarrow \cM^{\Gor}_{\textup{\textnumero}\star}$. For the second morphism, by \Cref{prop:extension to BB compactification}, for any $\big(X,(1-\epsilon)S\big)\in \cP^{\K}_{\textup{\textnumero}\star}(1-\epsilon)$ such that $(X,S)$ is not plt (equivalently, $S$ is not ADE), the image of points representing $(X,S)$ under the extended morphism $
\beta:\fP^{\K}_{\textup{\textnumero}\star}(1-\epsilon)^{\nu} \longrightarrow \overline{\fF}^{\BB}_{d,\Lambda}$ lies in the boundary $\overline{\fF}^{\BB}_{d,\Lambda}\setminus \fF_{d,\Lambda}$. Since $\beta$ is surjective and proper, it follows that every point of $\fF_{d,\Lambda}$ arises from a pair $(X,S)$ with $(X,S)$ plt. Therefore, the morphism
$\cP^{\plt}_{\textup{\textnumero}\star} \longrightarrow \cF_{d,\Lambda}$ is also surjective and proper.
\end{proof}

\smallskip

Let $(\Lambda,h)$ be as above. We define the \emph{moduli stack of marked Fano threefold pairs} of family $\textup{\textnumero}\star$, denoted by $\cN_{\textup{\textnumero}\star}$. %\footnote{\AS{Is that the pertinent object for families of Picard rank $1$? E.g. for $V_22$ we'd almost want to have markings for class groups rather than Picard groups. But maybe not relevant here?}}. 
An object of $\cN_{\textup{\textnumero}\star}$ over a scheme $B$ consists of a triple $(\sX,\sS;\rho)$ where

\[
\left\{
(\sX,\sS;\rho)/B \ \middle| \
\begin{aligned}
    \bullet\;& f:\sX \to B \text{ is a family of Gorenstein terminal Fano threefolds} \\
         & \text{deforming to the family \textnumero}\star;\\
     \bullet\;& \sS \sim_B -K_{\sX/B} \text{ is a relative anticanonical divisor whose fibers} \\
         & \text{are ADE K3 surfaces;}\\
     \bullet\;& \rho:(\underline{\Lambda}_B, h_B) \rightarrow(\Pic_{\sX/B}, -K_{\sX/B})
    \text{ is fiberwise an isometry.}
\end{aligned}
\right\}
\] Here $\Pic_{\sX/B}$ denotes the relative Picard functor, which is a sheaf on $B$ in the fppf topology and is representable by a group scheme over $B$, and the intersection pairing is defined fiberwise by
\[
(L_1,L_2) \ \mapsto \ \big(L_1 \cdot L_2 \cdot (-K_{\sX/B})\big) \in \underline{\bZ}_B,
\]
which equips $\Pic_{\sX/B}$ with a bilinear form. The marking $\rho$ is required to be an isometry sending $h_B$ to $-K_{\sX/B}$. In particular, over a geometric point, $\cN_{\textup{\textnumero}\star}$ parametrizes triples $(X,S;\rho)$ where $X$ is a Gorenstein terminal degeneration of the family $\textup{\textnumero}\star$, $S \in |-K_X|$ is an ADE K3 surface, and $\rho:(\Lambda,h) \xrightarrow{\sim} (\Pic(X), -K_X)$ is a lattice isometry; two triples $(X,S;\rho)$ and $(X',S';\rho')$ are isomorphic if and only if there exists an isomorphism $f:X\rightarrow X'$ which sends $S$ to $S'$ and the pull-back $f^*:\Pic(X')\rightarrow \Pic(X)$ satisfies $f^*\circ \rho' =\rho$. 

Let $G\coloneqq \Aut(\Lambda,h)$ be the finite group of isometries of $\Lambda$ preserving the class $h$. Then $G$ acts freely on $\cN_{\textup{\textnumero}\star}$ by
\[
g:(X,S;\rho)\ \mapsto \ (X,S;\rho\circ g),
\]
and the forgetful morphism $\cN_{\textup{\textnumero}\star}\rightarrow \cP^{\ter}_{\textup{\textnumero}\star}$ is $G$-equivariant. Therefore, it descends to a morphism $
\Psi_{\textup{\textnumero}\star}:[\cN_{\textup{\textnumero}\star}/G]\rightarrow \cP^{\ter}_{\textup{\textnumero}\star}$.

\begin{prop}\label{prop:isomorphism of stacks 1}
The morphism $\Psi_{\textup{\textnumero}\star}$ is an isomorphism.
\end{prop}

\begin{proof}
By \cite[Theorem~1.4]{JR11}, the relative Picard sheaf is locally constant in families of Gorenstein terminal Fano varieties. It follows that $\Psi_{\textup{\textnumero}\star}$ is bijective on geometric points and preserves stabilizer groups. Since $\cP^{\ter}_{\textup{\textnumero}\star}$ is a smooth Deligne--Mumford stack, Zariski's main theorem for stacks implies that $\Psi_{\textup{\textnumero}\star}$ is an isomorphism.
\end{proof}

\smallskip

The moduli theory of lattice-polarized K3 surfaces is more subtle; we follow the construction in \cite{AE25}. Let $(\Lambda,h)$ be a primitive sublattice of $\Lambda_{\K3}$ together with a positive vector as above. Since $h$ is in general not very irrational (cf.\ \cite[Definition~4.1]{AE25}), one first fixes a \emph{small cone} $\tau\subset \Lambda_{\bR}$ (cf.\ \cite[Definition~4.9]{AE25}) whose closure contains $h$, and chooses a very irrational vector $h'\in \tau$. The cone $\tau$ is an open cone contained in the positive part of the positive cone $\{v\in \Lambda_{\bR}\mid (v^2)>0\}$. Its role is to ensure that the moduli stack $\cF_{(\Lambda,h')}$ and its universal family are independent of the choice of very irrational vector $h'\in \tau$. If $h$ is contained in $\tau$, even if $h$ is not very irrational, \cite[Theorem 5.5]{AE25} shows that $\cF_{(\Lambda,h)}$ and $\cF_{(\Lambda,h')}$ are isomorphic smooth separated DM stacks with isomorphic universal families. If $h$ lies on the boundary of $\tau$, then for any K3 surface $(X,j)$ parametrized by $\cF_{(\Lambda,h')}$, the class $j(h)$ is nef and big. However, the Picard group of the ample model $\overline{X}\coloneqq \Proj R(X,j(h))$ may fail to contain $\Lambda$ as a primitive sublattice. This is precisely the subtlety that necessitates fixing the small cone $\tau$.

One can then define the moduli functor $\cF_{(\Lambda,h)}$ of $(\Lambda,h)$-polarized K3 surfaces as follows: to each scheme $B$, it assigns the groupoid 
\[ \left\{ (\wt{\sS}\rightarrow\sS,j)/B \ \middle| \ \begin{aligned}
\bullet\;& f:\sS\to B \text{ is a family of ADE K3 surfaces};\\
\bullet\;& \pi:\wt{\sS}\to \sS \text{ is a simultaneous partial resolution};\\
\bullet\;& j:\underline{\Lambda}_B\hookrightarrow \Pic_{\wt{\sS}/B}
\text{ is a primitive embedding with } j(h'_B) \text{ ample over } B;\\
\bullet\;& \pi:\wt{\sS}\to \sS \text{ is the ample model morphism associated to } j(h_B).
\end{aligned} \right\} \]
In particular, each $\bC$-point of $\cF_{(\Lambda,h)}$ corresponds to a triple $(\pi:\wt{S}\to S,j)$, where $S$ is an ADE K3 surface, $\pi$ is a partial resolution, and $j:\Lambda\hookrightarrow \Pic(\wt{S})$ is a primitive isometric embedding such that $j(h')$ is ample and $\pi$ is the ample model morphism associated to the nef and big line bundle $j(h)$. If $(S,j)$ is an ADE K3 surface together with a primitive isometric embedding
$j:\Lambda\hookrightarrow \Pic(S)$ such that $j(h)$ is ample, then by openness of ampleness,
$j(h')$ is also ample for any $h'\in \tau$ sufficiently close to $h$, and hence for all
$h'\in \tau$.

\begin{thm}[\textup{cf.\ \cite[Theorem~5.9 and Corollary~5.10]{AE25}}]
The stack $\cF_{(\Lambda,h)}$ of $(\Lambda,h)$-polarized K3 surfaces is a smooth separated Deligne--Mumford stack. Both $\cF_{(\Lambda,h)}$ and the universal family $\sS\to \cF_{(\Lambda,h)}$ are independent of the choice of small cone $\tau$ containing $h$ in its closure. Furthermore, $\sS$ admits a simultaneous crepant resolution to a family of $(\Lambda,h')$-polarized K3 surfaces for any $h'\in \tau$.
\end{thm}

\smallskip

There is a natural forgetful morphism
\[
\Phi_{\textup{\textnumero}\star}:\cN_{\textup{\textnumero}\star}\longrightarrow \cF_{(\Lambda,h)},
\]
which sends a family $(\sX,\sS;\rho)\to B$ to $(\sS,j)$, where
\[
j \coloneqq i^*\circ \rho\ :\ \Lambda\ \hookrightarrow\  \Pic(\sS/B)
\]
is a primitive isometric embedding and $i:\sS\hookrightarrow \sX$ denotes the inclusion. This is well-defined since the restriction map $\Pic(X)\to \Pic(S)$ is injective for any $(X,S;\rho)\in \cN_{\textup{\textnumero}\star}(\bC)$ by \cite[XII, Corollary~3.6]{SGA2}.

\begin{cor}\label{thm:smooth dominant forgetful map}
The forgetful morphism $\Phi_{\textup{\textnumero}\star}$ is smooth and dominant, of relative dimension $h^{1,2}$.
\end{cor}

\begin{proof}
This follows immediately from \Cref{thm:beauville}.
\end{proof}

\smallskip
The finite group $G\coloneqq \Aut(\Lambda,h)$ does not act naturally on $\cF_{(\Lambda,h)}$, since its elements do not necessarily preserve the chosen small cone $\tau$. However, in our situation the K3 surfaces of interest arise as anticanonical divisors of Fano threefolds, and thus the primitive embedding of $\Lambda$ into the Picard lattice is defined directly on the surface, without passing to a partial resolution. We therefore consider the open substack $\cF^{\circ}_{(\Lambda,h)}\subseteq \cF_{(\Lambda,h)}$ consisting of triples $(\pi:\wt{S}\to S,j)$ such that $\pi$ is the identity morphism. Equivalently, $\cF^{\circ}_{(\Lambda,h)}$ parametrizes ADE K3 surfaces $S$ equipped with a primitive isometric embedding $j:\Lambda\hookrightarrow \Pic(S)$. We simply denote an object in $\cF^{\circ}_{(\Lambda,h)}$ by $(S,j)$. The stack $\cF^{\circ}_{(\Lambda,h)}$ admits a natural free $G$-action $g:(S,j)\mapsto (S,j\circ g)$.

The forgetful morphism
\[
\cF^\circ_{(\Lambda,h)}\ \longrightarrow  \ \cF_d,
\qquad
(S,j)\ \mapsto\  (S,j(h)),
\]
which sends a lattice-polarized K3 surface to its underlying polarized K3 surface, has image contained in the open substack $\cF^\circ_{d,\Lambda}\subseteq \cF_{d,\Lambda}$ parametrizing polarized K3 surfaces whose Picard lattices contain $(\Lambda,h)$ as a primitive sublattice. This morphism is $G$-equivariant, and therefore descends to a morphism $[\cF^\circ_{(\Lambda,h)}/G]\rightarrow \cF^\circ_{d,\Lambda}$. In particular, it induces a morphism
\[
\Psi_{(\Lambda,h)}:[\cF^\circ_{(\Lambda,h)}/G]\ \longrightarrow \ (\cF^\circ_{d,\Lambda})^\nu\ (\subseteq\  \cF^\nu_{d,\Lambda}),
\]
where $\cF^\nu_{d,\Lambda}$ (resp.\ $\cF^{\circ \nu}_{d,\Lambda}$) denotes the normalization of $\cF_{d,\Lambda}$ (resp.\ $\cF^\circ_{d,\Lambda}$). Moreover, $\Psi_{(\Lambda,h)}$ is surjective by definition, and since a general polarized K3 surface parametrized by $\cF^\circ_{d,\Lambda}$ has Picard lattice isometric to $\Lambda$, the morphism $\Psi_{(\Lambda,h)}$ is birational.

\begin{prop}\label{prop:isomorphism of stacks 2}
    The morphism $\Psi_{(\Lambda,h)}$ is an isomorphism.
\end{prop}

\begin{proof}
    By the surjectivity of $\Psi_{(\Lambda,h)}$ and the Zariski's main theorem, it suffices to show that $\Psi_{(\Lambda,h)}$ is representable and quasi-finite. For any polarized K3 surface $(S,H)$ such that there exists a primitive embedding $(\Lambda,h)\hookrightarrow (\Pic(S),H)$, there are only finitely many different embeddings $j:(\Lambda,h)\hookrightarrow (\Pic(S),H)$: since $j(h)=H$, then $j$ is uniquely determined by its restriction $j|_{h^\perp}:\ h^\perp \hookrightarrow H^\perp$ between two negative definite lattices. However, primitive embedding between negative definite lattices admits only finitely many possibilities because vectors of bounded norm are finite. Therefore, there are only finitely many such embeddings $j$ and hence $\Psi_{(\Lambda,h)}$ is quasi-finite. For any $(S,j)\in \cF^\circ_{(\Lambda,h)}$, the automorphism group $\Aut(S,j)$ is naturally a subgroup of $\Aut(S,j(h))$, and hence $\Psi_{(\Lambda,h)}$ is representable.  
\end{proof}

\begin{cor}
    There exists a natural forgetful morphism \[
\ove{\Phi}_{\textup{\textnumero}\star}:\cP^{\ter}_{\textup{\textnumero}\star}\ \longrightarrow \ \cF^\nu_{d,\Lambda},
\] which is smooth and dominant of relative dimension $h^{1,2}$.
\end{cor}

\begin{proof}
    The natural forgetful morphism $\Phi_{\textup{\textnumero}\star}:\cN_{\textup{\textnumero}\star}\rightarrow \cF^\circ_{(\Lambda,h)}$ is $G$-equivariant, and hence it descends to a morphism $$\ove{\Phi}\ :\ \cP^{\ter}_{\textup{\textnumero}\star} \ \simeq\  [\cN_{\textup{\textnumero}\star}/G] \ \longrightarrow \  [\cF^\circ_{(\Lambda,h)}/G] \ \simeq \ (\cF^\circ_{d,\Lambda})^\nu \ \subseteq \ \cF^{\nu}_{d,\Lambda}$$ by \Cref{prop:isomorphism of stacks 1} and \Cref{prop:isomorphism of stacks 2}, and it is smooth and dominant of relative dimension $h^{1,2}$ by \Cref{thm:smooth dominant forgetful map}.
\end{proof}

\medskip
\section{Boundary components of K-moduli of $V_{22}$}

In this section, we apply the deformation package developed in the previous section to study the K-moduli of $V_{22}$. Our goal is to prove Theorems \ref{thm:main theorem 1}, \ref{thm:Fano-K3}, and \ref{thm:rationality of F22}.

\subsection{Open immersion of the forgetful map}

In this subsection, we prove the open immersion statement of Theorem \ref{thm:Fano-K3}. 

\begin{thm}\label{thm:open immersion forgetful map}
    The forgetful morphism $\Phi:\cP^{\K,\ADE}\rightarrow \cF_{22}$ is an open immersion.
\end{thm}

\begin{comment}

\begin{prop}
Let $X$ be a K-semistable degeneration of $V_{22}$. Assume that $X$ is not terminal. Then there exists a small deformation $X_t$ of $X$ %which lifts to a deformation of a terminalization of $X$ 
such that $X_t$ is a singular Gorenstein terminal degeneration of $V_{22}$. 
\end{prop}

\begin{proof}
By Theorem \ref{thm:smoothable} we know that either such $X_t$ exists, or $X$ has a $\bQ$-Gorenstein smoothing with Picard rank at least $2$. In the later case, $X$ is a degeneration of the Fano threefolds family \textnumero 2.15, \textnumero 2.16 or \textnumero 3.6. It was shown in \cite{LZ25} that K-moduli stack $\cM^\K_{{\textnumero 2.15}}$ of Fano \textnumero 2.15 only contains Gorenstein terminal Fano threefolds. By \Cref{thm:K-ss limit of 2.16}, the moduli stack $\cM^\K_{\textup{\textnumero 2.16}}$ of Fano \textnumero 2.16 is smooth, so it does not intersect $\cM^\K$ as a substack of $\cM^{\Kss}_{3,22}$.
\end{proof}
\end{comment}

\begin{lemma}\label{lem:forgetful map representable}
    The forgetful morphism $\Phi:\cP^{\K,\ADE}\rightarrow \cF_{22}$ is representable.
\end{lemma}

\begin{proof}
Let $[(X,S)]\in \cP^{\K,\ADE}(\bC)$ be a pair. First note that $\Aut(X,S)$ is finite. 
Indeed, for any $0<\epsilon\ll1$, the pair $(X,(1-\epsilon)S)$ is K-stable by 
\cite[Theorem 2.10]{ADL21}, hence its automorphism group is finite. 
Let $G\subseteq \Aut(X,S)$ be the subgroup consisting of automorphisms whose induced 
automorphism on $S$ is the identity. We will show that $m\coloneqq |G|=1$.

Since $S\in |-K_X|$ is an anticanonical K3 surface, there is an exact sequence
\[
0\ \longrightarrow \ H^0(X,\cO_X)\ \longrightarrow\  H^0(X,-K_X)
\ \longrightarrow\  H^0(S,-K_X|_S)\ \longrightarrow\  0,
\]
where the first map is multiplication by a section $s\in H^0(X,-K_X)$ whose zero locus is $S$.
As $G$ acts trivially on $H^0(S,-K_X|_S)$, the induced action of $G$ on $H^0(X,-K_X)$ is 
diagonalizable of the form $\diag(\zeta_m,1,\dots,1)$, where $\zeta_m$ is an $m$-th root of unity 
and the defining section $s$ of $S$ spans the $\zeta_m$-eigenspace. In particular, $G$ is a finite 
subgroup of $\bC^*$ and hence cyclic; write $G=\langle g\rangle$ with $g(s)=\zeta_m s$.

Let $\pi:X\to Y:=X/G$ be the quotient morphism, and denote by $B:=\pi(S)$ the image of $S$. 
Since $-K_X$ is very ample by \Cref{thm:Kss-very-ample}, the anticanonical linear system embeds
\[
X \ \hookrightarrow\  \mathbf P(H^0(X,-K_X)^\vee)\ =\ \mathbf P^{13}_{[x_0:\cdots:x_{13}]}.
\]
Under the above diagonal action, the quotient $\mathbf P^{13}/G$ is the weighted projective space 
$\bP(1^{13},m)$, and $Y$ is naturally a subvariety of $\bP(1^{13},m)$. Moreover, $B$ is the restriction 
to $Y$ of the divisor $(y=0)$, where $y=x_0^m$, and hence $B$ is Cartier.

We now show that $Y$ is Gorenstein. Since $S\in |-K_X|$, we have $K_X+S\sim0$. 
Consider the residue map
\[
\operatorname{res}:H^0(X,K_X+S)\longrightarrow H^0(S,K_S),
\]
which is a $G$-equivariant isomorphism. As $g|_S=\id_S$, the group $G$ acts trivially on 
$H^0(S,K_S)$ and hence also on $H^0(X,K_X+S)$. Therefore there exists a nonzero 
$G$-invariant section of $K_X+S$, which descends to a nonzero section of $K_Y+B$. 
Thus $K_Y+B\sim0$. Since $B$ is Cartier, it follows that $K_Y$ is Cartier, and hence 
$Y$ is Gorenstein.

Finally, since $\pi^*(-K_Y)=-mK_X$, we obtain
\[
(-K_Y)^3 \ =\ m^2(-K_X)^3\ =\ 22m^2.
\]
If $g\neq \id_X$, then $m\ge2$, so $(-K_Y)^3\ge88$, contradicting \cite[Theorem 1.5]{Pro05}. 
Therefore $m=1$, and hence $G$ is trivial.
\end{proof}

\begin{remark}
In general, the natural homomorphism $\Aut(X,S)\to \Aut(S)$ need not be injective for a smooth Fano threefold $X$ and an anticanonical ADE K3 surface $S\subseteq X$. For example, let $X\subseteq \mathbb{P}^4$ be the smooth quartic threefold defined by $x^4+f(y,z,u,v)=0$, where $f(y,z,u,v)$ is a general quartic form. Then $S:=X\cap (x=0)\subseteq \mathbb{P}^3$ is a smooth quartic K3 surface. The automorphism $[x:y:z:u:v]\mapsto [\zeta_4 x:y:z:u:v]$ preserves $X$ and acts trivially on $S$, but is nontrivial on $X$. Equivalently, $X$ is a cyclic cover of $\mathbb{P}^3$ of degree $4$ branched along the quartic surface $S$.
\end{remark}

\begin{lem}\label{lem:birationality of forgetful}
The forgetful morphism $\Phi:\cP^{\K,\ADE}\to \cF_{22}$ is birational.
\end{lem}

\begin{proof}
It suffices to show that, for a very general polarized K3 surface $(S,L)$ of genus $12$, there exists, up to automorphisms, a unique smooth $V_{22}$ containing $S$ as an anticanonical divisor. This is precisely the uniqueness statement in \cite[Theorem~1.3(a)]{BKM25} for $g=12$.
\end{proof}

\begin{proof}[Proof of \Cref{thm:open immersion forgetful map}]
Since $\Phi$ is a representable (cf. \Cref{lem:forgetful map representable}) and birational (cf. \Cref{lem:birationality of forgetful}) morphism between separated Deligne--Mumford stacks of finite type over $\bC$, and $\cF_{22}$ is smooth and $\cP^{\K,\ADE}$ is reduced, Zariski's main theorem for Deligne--Mumford stacks (see e.g.\ \cite[Theorem 5.5.9]{AlperModuli}) reduces us to proving that $\Phi$ is quasi-finite.

To this end, we introduce a locally closed stratification of $\cP^{\K,\ADE}$ as follows. Let $\cP_0$ be the open substack parametrizing pairs $(X,S)$ such that $X$ is Gorenstein terminal, and let $\cP^{\mathrm{int}}$ be the closed substack parametrizing pairs $(X,S)$ for which $X$ is a degeneration of a family of smooth Fano threefolds distinct from $V_{22}$. By \cite[Proposition 3]{Nam97} (see also the proof of \Cref{lem:exclude the other families}), these two substacks are disjoint. We choose a locally closed stratification
\[
\cP^{\K,\ADE}\setminus (\cP_0\sqcup \cP^{\mathrm{int}})
\ =\ 
\bigsqcup_{i=1}^r \cP_i
\]
such that for each stratum $\cP_i$ the following hold:
\begin{enumerate}
    \item the one-dimensional singular locus $\sZ_i$ of the threefold part of the universal family $(\sX_i,\sS_i)\to \cP_i$ is flat over $\cP_i$;
    \item the exceptional divisor of the blowup $\Bl_{\sZ_i}\sX_i\to \sX_i$ is flat over $\cP_i$.
\end{enumerate}

\begin{lemma}
For any $t\in \cP_i$, one has
\[
\bigl(\Bl_{\sZ_i}\sX_i\bigr)_t \simeq \Bl_{(\sZ_i)_t}(\sX_i)_t.
\]
\end{lemma}

\begin{proof}
Since the claim is local in the smooth topology, we may assume that $T:=\cP_i$ is an integral affine scheme $\Spec R$, where $(R,\fm)$ is a local ring, $\sX=\Spec A$, and $\sZ$ is defined by an ideal $I\subseteq A$. By the choice of the stratification, $A$, $A/I$, and $I^k/I^{k+1}$ are all flat over $R$ for all $k\geq 1$. Using the exact sequence
\[
0\longrightarrow I^k/I^{k+1}\longrightarrow A/I^{k+1}\longrightarrow A/I^k\longrightarrow 0,
\]
we deduce inductively that $A/I^k$ is flat over $R$ for all $k\ge 1$. The desired compatibility of blowups is equivalent to the equality
\[
(I/\fm I)^k \ =\  I^k/\fm I^k
\]
for all sufficiently large $k$. Since $(I/\fm I)^k$ is the image of the natural map
$\phi \colon I^k/\fm I^k \longrightarrow A/\fm A$, it suffices to show that $\phi$ is injective. This follows from the flatness of $A/I^k$ over $R$ and the exact sequence
\[
0\longrightarrow I^k \longrightarrow A \longrightarrow A/I^k \longrightarrow 0,
\]
which remains exact after tensoring with $R/\fm$.
\end{proof}
%$(R,\fm)$ is a local ring, that $\sX\subseteq \bA^4_R$ is a hypersurface defined by an irreducible polynomial $f\in R[x_1,x_2,x_3,x_4]$, and that $\sZ = V(x_1,x_2,x_3)\subset \sX$. As by construction $\sX$ has double points along $\sZ$, then every monomial of $f$ has degree at least $2$ with respect to $x_1,x_2,x_3$, and some monomial of $f$ has degree exactly $2$ with respect to $x_1,x_2,x_3$. It suffices to show that for $k\gg 0$ the natural map
%\[\frac{(x_1,x_2,x_3)^k}{\fm(x_1,x_2,x_3)^k+(f)} \ \longrightarrow\ \left(\frac{(x_1,x_2,x_3)+\fm}{\fm+(f)}\right)^k \] is an isomorphism. Let $g(\underline{x})\in R[x_1,x_2,x_3,x_4]$ be a polynomial such that every monomial has degree at least $k$ with respect to $x_1,x_2,x_3$. Write $\ove{(\cdot)}$ for reduction modulo $\fm$, so that $\ove{R}=R/\fm\simeq \bC$. If $\ove g$ lies in the ideal $(\ove f)\subset \bC[x_1,x_2,x_3,x_4]$, then we may write \[\ove g\ =\ \ove f\cdot h \] for some polynomial $h$ whose monomials have degree at least $k-2$ with respect to $x_1,x_2,x_3$ (since $\bar f$ has $(x_1,x_2,x_3)$-degree exactly $2$). Viewing $h$ as a polynomial in $R[x_1,x_2,x_3,x_4]$, still denoted by $h$, we obtain \[ g-f\cdot h \ =\  p,\] where $p$ has coefficients in $\fm$ and every monomial of $p$ has degree at least $k$ with respect to $x_1,x_2,x_3$. This shows that the above map is injective, hence an isomorphism, and the desired compatibility of blow-ups with the fiber over $t$ follows.

In the following lemmas, we show that the restriction of $\Phi$ to each stratum is quasi-finite.

\begin{lemma}\label{lem:restriction on Gor term locus q-finite}
The restriction $\Phi|_{\cP_0}$ is quasi-finite.
\end{lemma}

\begin{proof}
By \cite[Theorem 1.4]{JR11}, for every $t\in \cP_0$ the image of the restriction map $\Pic(X_t)\longrightarrow \Pic(S_t)$ is a locally constant lattice, and in particular is isomorphic to the rank-one lattice $\langle \ell_{22}\rangle$ with $(\ell_{22}^2)=22$. By \Cref{thm:beauville}, the morphism $\Phi$ is smooth at any point $[(X,S)]\in \cP_0$ of relative dimension $h^{1,2}(X)$. By \Cref{thm:beauville}, this number is zero, since $h^{1,2}$ vanishes for a smooth $V_{22}$. Hence $\Phi$ has zero-dimensional fibers at $[(X,S)]$, and therefore is quasi-finite at such points. The lemma follows.
\end{proof}

In particular, every Gorenstein terminal $V_{22}$ has vanishing Hodge number $h^{1,2}$.

\begin{lemma}\label{lem:restriction on each strata}
The restriction $\Phi|_{\cP_i}$ is quasi-finite for every $i=1,\dots,r$.
\end{lemma}

\begin{proof}
Fix $i\ge 1$. By construction of the stratification, there exists a partial resolution
\[
g_i:\wt{\sX}_i\to \sX_i
\]
obtained by blowing up the one-dimensional singular locus, such that
$(\wt{\sX}_i,\wt{\sS}_i)\to \cP_i$ is a family of Gorenstein terminal weak Fano threefolds with anticanonical ADE K3 surfaces $\wt{\sS}_i=g_i^*\sS_i$.

Let $(X,S)\in \cP_i$ and let $g:(\tX,\tS)\to (X,S)$ be the corresponding blowup. By
\Cref{cor:small deformation small anti-can morphism}, $\tX$ admits a small deformation $\tX_t$ whose anticanonical model $X_t$ is Gorenstein terminal, and the forgetful map $\Def(\tX,\tS)\to\Def(\tS)$ has relative dimension $0$, since $h^{1,2}(X_t)=0$. Because the deformation space of $(X,S)$ in $\cP_i$ maps to a subspace of $\Def(\tX,\tS)$, the fibers of $\Phi$ over $\cP_i$ are zero-dimensional at $(X,S)$. Hence $\Phi|_{\cP_i}$ is quasi-finite.
\end{proof}

\begin{lemma}\label{lem:exclude the other families}
The restriction of $\Phi|_{\cP^{\rm int}}$ is quasi-finite.
\end{lemma}

\begin{proof}
If this substack is non-empty, then for any pair $(X,S)$ in this locus, the threefold $X$ is a common degeneration of $V_{22}$ and another family of smooth Fano threefolds of volume $22$. In particular, $\cM^\K_{3,22}$ is singular at the point $[X]$. Hence, by \cite[Main Theorem]{Min01} and \Cref{thm:volume bound}(1), $X$ is Gorenstein canonical but not terminal.

By the classification of smooth Fano threefolds (cf.\ \cite{Fano}), there are exactly four families of volume $22$: \textnumero1.10 (i.e.\ $V_{22}$), \textnumero2.15, \textnumero2.16, and \textnumero3.6. By \cite[Theorem 1.2]{LZ25}, every K-semistable Fano degeneration of family \textnumero2.15 has only ADE singularities and is therefore terminal. In \Cref{appendix B} we show that the K-moduli stack $\cM^\K_{\textup{\textnumero2.16}}$ is smooth; see \Cref{thm:K-ss limit of 2.16}. Consequently, $X$ must be a degeneration of family \textnumero3.6.

Consider the stack $\cP^{\K,\ADE}_{\textup{\textnumero3.6}}$ parametrizing pairs $(Y,T)$ such that $Y$ is a K-semistable degeneration of the Fano family \textnumero3.6 and $T\in |-K_Y|$ is an ADE K3 surface. Since $\dim \mtc{M}^\K_{\textup{\textnumero3.6}}=5$ and $h^0(Y,-K_Y)=14$ for any $Y\in \mtc{M}^\K_{\textup{\textnumero3.6}}$, the intersection
\[\cP^{\mathrm{int}} \ = \
\cP^{\K,\ADE}\ \cap\  \cP^{\K,\ADE}_{\textup{\textnumero3.6}}
\]
is non-empty by assumption and has dimension at most $17$. Here, when taking the intersection, we can view both of them as substacks of the K-moduli stack of pairs $\cP^\K_{3,22}(c)$ for any $c\in (0,1)$.

Since the forgetful morphism $\Phi:\cP^{\K,\ADE}\to \cF_{22}$
is a birational morphism to a smooth Deligne--Mumford stack, by the purity of the exceptional locus for birational morphisms (see e.g.\ \cite[1.40]{Deb01}), the restriction $\Phi|_{\cP^{\mathrm{int}}}$ is quasi-finite. This proves the lemma.
\end{proof}

Therefore, $\Phi$ is quasi-finite and thus an open immersion.
\end{proof}
\begin{cor}\label{cor:reduced K-moduli of V22 is smooth}
    The moduli stack $\cM^\K$ is smooth.
\end{cor}

\begin{proof}
Let $\cP^{\K,\ADE}_{3,22}$ be the open substack of $\cP^\K_{3,22}(c)$, for $0<c<1$, consisting of pairs $(X,S)$ such that $X\in\cM^\K_{3,22}$ and $S\in |-K_X|$ is an ADE K3 surface. 
There is a natural forgetful morphism $(X,S)\mapsto X$ from $\cP^{\K,\ADE}_{3,22}$ to $\cM^\K_{3,22}$, which is smooth by the proof of \Cref{thm:beauville}. Let $\beta:\cP^{\K,\ADE}_{\diamondsuit}\to \cM^\K$ be the pullback of this morphism along $\cM^\K\to \cM^\K_{3,22}$. 
Then $\beta$ is also smooth, and we obtain the following cartesian diagram
\[
\begin{tikzcd}[ampersand replacement=\&]
	{\cP^{\K,\ADE}} \&\& {\cP^{\K,\ADE}_{\diamondsuit}} \&\& {\cP^{\K,\ADE}_{3,22}} \\
	\&\& {\cM^\K} \&\& {\cM^\K_{3,22}}
	\arrow["\alpha", from=1-1, to=1-3]
	\arrow["\Psi"{description}, from=1-1, to=2-3]
	\arrow[from=1-3, to=1-5]
	\arrow["\beta", from=1-3, to=2-3]
    \arrow[from=1-3, to=2-5, phantom, "\lrcorner", very near start]
	\arrow[from=1-5, to=2-5]
	\arrow[from=2-3, to=2-5]
\end{tikzcd}.
\] Since $\cM^\K$ is reduced, the stack $\cP^{\K,\ADE}_{\diamondsuit}$ is also reduced. 
It follows that $\alpha$ is an isomorphism, and hence $\Psi$ is smooth. Finally, as $\Phi:\cP^{\K,\ADE}\hookrightarrow \cF_{22}$ is an open immersion and $\cF_{22}$ is smooth by \Cref{isommoduli}, the stack $\cP^{\K,\ADE}$ is smooth. 
Therefore $\cM^\K$ is smooth as well.
\end{proof}

\smallskip

\begin{comment}

\begin{lem}\label{lem:Pic of sing quadric}
    The Picard group of a singular quadric surface in $\bP^3$ is isomorphic to $\bZ$. 
\end{lem}

\begin{proof}
  {\color{red} Decide whether we want to add this.}
\end{proof}

\end{comment}

\subsection{Nodality of terminal K-semistable $V_{22}$}

In this subsection, we prove that every terminal K-semistable $V_{22}$ has at worst nodal singularities. 

\begin{thm}\label{cor:deforms to a one-nodal}
    Every K-semistable singular terminal degeneration of $V_{22}$ has only $A_1$-singularities. In particular, it deforms to one-nodal Fano $V_{22}$.
\end{thm}

\begin{prop}\label{thm:terminal-nodal}
Every $\bQ$-Gorenstein Fano degeneration of $V_{22}$ with only isolated $cA_{\leq 2}$-singularities has at worst $A_1$-singularities.
%\footnote{\YL{Perhaps this can be strengthened to Gorenstein terminal Fano threefolds of volume $22$ and Picard rank $1$, since by \cite{KS25} they admit smoothings to $V_{22}$.}}
\end{prop}

\begin{proof}
Let $X$ be such a degeneration of $V_{22}$, and let $p\in X$ be a singular point. By \Cref{lem:Picard rank one} and \cite[Theorem~2.2]{Pro15}, the anticanonical divisor $-K_X$ is very ample. Let $\phi:\wt{X}\to X$ be the blow-up of $X$ at $p$, with exceptional divisor $E$. As $X$ has $cA_{\leq 2}$-singularity at $p$, then $\wt{X}$ is Gorenstein terminal, and $-K_{\wt{X}} = \phi^*(-K_X) - E$ is base-point free with positive top self-intersection, so $\wt{X}$ is weak Fano.

We claim there is a natural exact sequence
$$0 \ \longrightarrow\  \phi^*\Pic(X) \ \longrightarrow\   {\Pic(\wt{X})} \ \longrightarrow\   {\Pic(E)}.%\footnote{\YL{One may wish to replace $E$ by $E_{\red}$ to include the $cD$ and $cE$ cases.}}
$$ It suffices to show that if $L\in \Pic(\wt{X})$ satisfies $L|_E \simeq \mathcal{O}_E$, then $L \simeq \phi^*M$ for some $M\in \Pic(X)$. Since the Mori cone of $\wt{X}$ is finitely generated, we may choose $m\gg 0$ such that $L - m\phi^*K_X$ is big and nef. Moreover, since $L-E$ is $\phi$-relatively ample, we may assume $L - m\phi^*K_X - E$ is ample. Consider the short exact sequence
\[
0 \longrightarrow \mathcal{O}_{\wt{X}}(L - m\phi^*K_X - E)
\longrightarrow \mathcal{O}_{\wt{X}}(L - m\phi^*K_X)
\longrightarrow \mathcal{O}_E(L - m\phi^*K_X)\simeq \mathcal{O}_E
\longrightarrow 0.
\]
By Kawamata--Viehweg vanishing, this yields a surjection
\[
H^0(\wt{X}, L - m\phi^*K_X) \twoheadrightarrow H^0(E,\mathcal{O}_E).
\]
Thus we may choose a section of $\mathcal{O}_{\wt{X}}(L - m\phi^*K_X)$ not vanishing along $E$, and it follows that $L - m\phi^*K_X \in \phi^*\Pic(X)$, proving the claim.

Since $p\in X$ is a $cA_{\leq 2}$-singularity, the exceptional divisor $E$ is a quadric surface in $\mathbb{P}^3$, and $p$ is an $A_1$-singularity if and only if $E$ is smooth. Note that $\rho(\mathbb{P}^1\times \mathbb{P}^1)=2$, while the Picard ranks of the singular reduced quadric surfaces are equal to $1$. Therefore, by \Cref{lem:Picard rank one}, it suffices to show that $\rho(\wt{X})\neq 2$. Suppose instead that $\rho(\wt{X})=2$. In the rest of the proof, we will derive a contradiction. 

Let $\wt{S}\in |-K_{\wt{X}}|$ be a general elephant, which is smooth. Let $\Lambda$ be the saturation of $\im(\Pic(\wt{X})\to \Pic(\wt{S}))$ in $H^2(\wt{S},\mathbb{Z})$. Then $\rk \Lambda = 2$, and the classes
\[
e_1 \coloneqq \phi^*(-K_X)|_{\wt{S}}, 
\qquad 
e_2 \coloneqq E|_{\wt{S}}
\]
satisfy
\[
(e_1^2)=22, \qquad (e_1\cdot e_2)=0, \qquad (e_2^2)=-2.
\]

\begin{lemma}
Let $\Lambda=\langle e_1,e_2\rangle$ be the rank-$2$ lattice with Gram matrix
\[
\begin{pmatrix}
22 & 0 \\
0 & -2
\end{pmatrix}.
\]
Then any embedding $\Lambda \hookrightarrow \Lambda_{K3}$ is primitive.
\end{lemma}

\begin{proof}
Let $\Lambda^\vee$ be the dual lattice. Then
\[
\Lambda^\vee\ =\  \Big\langle \frac{e_1}{22},\ \frac{e_2}{2}\Big\rangle,
\qquad 
A_\Lambda\ \coloneqq\ \Lambda^\vee/\Lambda\ \simeq \ \mathbb{Z}/22 \oplus \mathbb{Z}/2.
\]
The discriminant quadratic form is
\[
q_\Lambda(a,b)\ =\ \frac{a^2}{22}-\frac{b^2}{2}
\in \mathbb{Q}/2\mathbb{Z},
\qquad a\in \mathbb{Z}/22,\ b\in \mathbb{Z}/2.
\]
By Nikulin's correspondence \cite[Proposition~1.4.1]{Nik79}, even overlattices of $\Lambda$ are in bijection with isotropic subgroups of $(A_\Lambda,q_\Lambda)$. It therefore suffices to show that $A_\Lambda$ has no nonzero isotropic element.
\begin{enumerate}
    \item If $b=0$, isotropy implies $a^2/22 \equiv 0 \pmod{2}$, hence $a^2 \equiv 0 \pmod{44}$, so $a\equiv 0$ in $\mathbb{Z}/22$.
    \item If $b=1$, isotropy would require $\frac{a^2}{22} \equiv \frac{1}{2} \pmod{2}$, i.e. $a^2 \equiv 11 \pmod{44}$, which is impossible since $11\equiv 3\pmod{4}$, whereas a square modulo $4$ is $0$ or $1$.
\end{enumerate} Thus $(A_\Lambda,q_\Lambda)$ has no nonzero isotropic element, so $\Lambda$ admits no proper even overlattice. Hence any embedding $\Lambda\hookrightarrow \Lambda_{K3}$ is primitive.
\end{proof}

By \Cref{thm:beauville}, the forgetful morphism
\[\Def_{(\wt{X},\wt{S})} \longrightarrow \Def_{(\wt{S},\Lambda)}\]
is smooth and surjective. Therefore we may choose a small deformation $(\wt{X}_t,\wt{S}_t)$ such that $\wt{S}_t$ is a very general element of $\Def_{(S,\Lambda)}$. Let $\wt{L}_t\in \Pic(\wt{X}_t)$ be the deformation of $\phi^*(-K_X)$; it remains globally generated and big. The $\wt{L}_t$-ample model $(\overline{X}_t,\overline{S}_t)$ is then a deformation of $(X,S)$, where $S$ is the image of $\wt{S}$ in $X$. In particular, $\overline{X}_t$ is a K-semistable degeneration of $V_{22}$ and $\overline{S}_t$ is a very general point in the nodal divisor $\cD^{22}_{0,-2}$ of $\cF_{22}$. 

We claim that $\overline{X}_t$ must be singular. Suppose otherwise that $\overline{X}_t$ is smooth. Then $\overline{X}_t$ is a smooth $V_{22}$ and $\rho(\overline{X}_t)=1$, since $\Def(X)$ is smooth by \cite[Main Theorem]{Min01}. In particular, there exists an exceptional prime divisor $E_t$ of $\wt{X}_t \to \overline{X}_t$. Let $E_0$ be the degeneration of $E_t$ on $\wt{X}$; its support is contained in $E$. Since the center satisfies 
$c_X(E_0) \subseteq c_X(E) = p$, it follows that the center $c_{\overline{X}_t}(E_t)$ is also a point on $\overline{X}_t$. As $E$ is reduced, we have $A_X(E_{0,\mathrm{red}}) = 2$,
and hence $
A_{\overline{X}_t}(E_t) \leq 2$.
However, the minimal log discrepancy of a smooth closed point on a threefold is $3$, yielding a contradiction.

\begin{lemma}\label{lem:nodal NL divisor}
A K3 surface parametrized by a very general point on the nodal Noether--Lefschetz divisor of $\cF_{22}$ is an anticanonical divisor of a smooth K-semistable $V_{22}$.
\end{lemma}

\begin{proof}
Let $X$ be a smooth K-semistable $V_{22}$. Since $X \subset \bP^{13}$ is not a scroll, a very general pencil of hypersurfaces in $|\cO_X(1)|$ is a Lefschetz pencil and therefore contains at least one singular member, denoted by $S$. The claim then follows immediately from \Cref{thm:open immersion forgetful map}.
\end{proof}
Finally, combining \Cref{lem:nodal NL divisor} with \Cref{thm:open immersion forgetful map}, we obtain a contradiction. Therefore $\rho(\wt{X}) = 3$, and $X$ has only $A_1$-singularities.
\end{proof}

\smallskip

\begin{proof}[Proof of \Cref{cor:deforms to a one-nodal}]
    By \Cref{thm:volume bound}(2), any K-semistable terminal degeneration $X$ of $V_{22}$ has only isolated $cA_{\leq 2}$-singularities. Hence, by \Cref{thm:terminal-nodal}, $X$ has only $A_1$-singularities. Then \cite[Proposition 4]{Nam97} implies that $H^2(X, T_X) = 0$, so there are no local-to-global obstructions to deforming $X$. Consequently, $X$ deforms to a one-nodal Fano $V_{22}$.
\end{proof}

\smallskip

\subsection{Proofs of main theorems}

We conclude this section by proving Theorems \ref{thm:main theorem 1}, \ref{thm:Fano-K3}, and \ref{thm:rationality of F22}.

\smallskip

\begin{proof}[Proof of \Cref{thm:main theorem 1}]

By \Cref{thm:smoothable}, \Cref{cor:deforms to a one-nodal}, and the proof of \Cref{lem:exclude the other families}, it suffices to show the following.

\begin{lem}
Let $X_0$ be a common K-semistable $\bQ$-Gorenstein degeneration of the Fano threefold families $V_{22}$ and {\rm\textnumero3.6}, assuming such a degeneration exists. Then $X_0$ deforms to a Type~\rm{I} one-nodal $V_{22}$.
\end{lem}

\begin{proof}
A smooth Fano threefold of the family \textnumero3.6 is the blowup of $\bP^3$ along the disjoint union of a line and an elliptic normal curve. Consider the forgetful map
\[
\Phi_{\textup{\textnumero3.6}}:\cP^{\K,\ADE}_{\textup{\textnumero3.6}}\longrightarrow \cF_{22}.
\]
Its image is contained in the Noether--Lefschetz locus associated to the rank-three lattice
\[
\begin{pmatrix}
22 & 11 & 6 \\
11 & 4 & 1 \\
6  &  1 & -2
\end{pmatrix}.
\] 
Since this lattice contains
\[
\begin{pmatrix}
22 & 11 \\
11 & 4
\end{pmatrix}
\]
as a primitive sublattice, any ADE K3 surface $S_0\in |-K_{X_0}|$ is a degeneration of a family of anticanonical K3 surfaces of Type~I $V_{22}$; see \Cref{lem:K3 lattice of Type1-4}. In particular, the image of the restricted forgetful map
\[
\Phi|_{\cP^{\mathrm{int}}}:\cP^{\mathrm{int}}\longrightarrow \cF_{22},
\]
where $\cP^{\mathrm{int}}=\cP^{\K,\ADE}\cap \cP^{\K,\ADE}_{\textup{\textnumero3.6}}$, is contained in the closure of the image of $\Phi|_{\cP^{\K,\ADE}_{\mathrm{I}}}$, where $\cP^{\K,\ADE}_{\mathrm{I}}$ denotes the closed substack of $\cP^{\K,\ADE}$ parametrizing pairs $(X,S)$ such that $X$ is a degeneration of a family of Type~I $V_{22}$. Therefore, by \Cref{thm:open immersion forgetful map}, we conclude that $(X_0,S_0)\in \cP^{\K,\ADE}_{\mathrm{I}}$, and hence $X_0$ is a degeneration of Type~I $V_{22}$.
\end{proof}
\end{proof}

\begin{proof}[Proof of \Cref{thm:Fano-K3}]

We first record the following.

\begin{lem}\label{lem:BN general of anticanonical K3 of smooth V22}
Let $X$ be a smooth $V_{22}$. Then any ADE K3 surface $S\in |-K_X|$ is Brill--Noether general.
\end{lem}

\begin{proof}
Choose a very general $S_g\in |-K_X|$ such that the curve $C\coloneqq S_g\cap S$ is smooth and contained in the smooth locus of $S$. By \cite{Laz86}, the polarized K3 surface $(S_g,\cO_{S_g}(C))$ is Brill--Noether general, hence $C$ is a Brill--Noether general curve by \cite[Theorem~1]{Hab24}. It then follows from \cite[Theorem~2.10]{BKM25} that $(S,\cO_S(C))=(S,-K_X|_S)$ is Brill--Noether general.
\end{proof}

By \Cref{thm:main theorem 1}, the open substack $\cP^{\K,\ADE,\circ}$ of $\cP^{\K,\ADE}$ parametrizing pairs $(X,S)$ with $X$ a smooth K-semistable $V_{22}$ is 
\[
\cP^{\K,\ADE}
\ \setminus \
\big(\cP^{\K,\ADE}_{\mathrm{I}}\cup \cP^{\K,\ADE}_{\mathrm{II}}
\cup \cP^{\K,\ADE}_{\mathrm{III}}\cup \cP^{\K,\ADE}_{\mathrm{IV}}\big),
\] where $\cP^{\K,\ADE}_{\mathrm{I}}, \ldots, \cP^{\K,\ADE}_{\mathrm{IV}}$
denote the closures of the loci parametrizing pairs $(X,S)$ such that 
$X$ is a one-nodal $V_{22}$ of Types I–IV, respectively. The images of these four divisors are contained respectively in the Noether--Lefschetz divisors
\[
\cF_{22,\Lambda_{\mathrm{I}}}=\cD^{22}_{11,4},\quad
\cF_{22,\Lambda_{\mathrm{II}}}=\cD^{22}_{9,2},\quad
\cF_{22,\Lambda_{\mathrm{III}}}=\cD^{22}_{6,0},\quad
\cF_{22,\Lambda_{\mathrm{IV}}}=\cD^{22}_{5,0}.
\] By \Cref{lem:BN general of anticanonical K3 of smooth V22}, the image of
$\cP^{\K,\ADE,\circ}$ consists of Brill--Noether general K3 surfaces; hence by
\Cref{lem:equivalent condition of BN general} it is disjoint from the seven Noether--Lefschetz divisors. Combining this with \Cref{thm:open immersion forgetful map}, we obtain the desired result.
\end{proof}

\smallskip

\begin{proof}[Proof of \Cref{thm:rationality of F22}]
For a prime K3 surface $(S,L)\in \cF_{22}$, the automorphism group is trivial (cf.~\cite[Corollary~15.2.12]{Huy16}). It follows that the stack $\cF_{22}$ is birational to its coarse moduli space $\fF_{22}$. Thus, it suffices to prove that $\cF_{22}$ is rational.

By \Cref{thm:open immersion forgetful map}, it further suffices to show that $\cP^{\K,\ADE}$ is rational. Since $\Aut(X)=0$ for a very general $X\in \cM^\K$ by \cite[Theorem~A.1]{DM22}, the universal family $\pi:\sX\to \cM^\K$ induces a birational morphism
\[
\cP^{\K,\ADE} \dashrightarrow \Proj_{\cM^\K}\!\big(\Sym^{\bullet}\pi_*\omega^{\vee}_{\sX/\cM^\K}\big),
\]
to a projective bundle over $\cM^\K$. Therefore, $\cP^{\K,\ADE}$ is rational provided that $\cM^\K$ is rational. Finally, by \cite{Muk92}, the stack $\cM^\K$ is birational to the moduli stack $\cM_3$ of smooth curves of genus three, which is known to be rational. This completes the proof.
\end{proof}

%\footnote{\YL{As another approach, somehow we should try to bound the genus of $C_i$. This controls the dimension of the fiber of $\Def_{(\tX,\tS)}\to \Def_{\tS}$. Phil said we should use the vanishing cycle sequence. Need to think about it more.}}

\medskip
\section{Reconstruction of Fano threefolds from K3 surfaces}\label{sec:Reconstruction}

By Theorem \ref{thm:Fano-K3}, the forgetful morphism
\[
\Phi:\cP^{\K,\ADE}\hookrightarrow \cF_{22}
\]
is an open immersion. In particular, by Lemma \ref{lem:equivalent condition of BN general} the image of $\Phi$ is contained in the complement of the following seven Noether--Lefschetz divisors:
\[
\cD_{1,0}^{22},\ \ 
\cD_{2,0}^{22},\ \ 
\cD_{3,0}^{22},\ \ 
\cD_{4,0}^{22},\ \ 
\cD_{7,2}^{22},\ \ 
\cD_{8,2}^{22},\ \ 
\cD_{10,4}^{22}.
\] 

This suggests a natural reconstruction problem: given a polarized K3 surface $(S,L)$ of degree $22$, can one recover a Fano threefold $X$ admitting $S$ as an anticanonical divisor? In this section, we investigate this question for K3 surfaces lying in the seven Noether--Lefschetz divisors above.

For a general K3 surface $S$ parametrized by each of these divisors, we construct a Gorenstein canonical Fano threefold $X$ of degree $22$ containing $S$ as an anticanonical divisor. Somewhat surprisingly, the resulting Fano threefold is essentially rigid: for a general $S$ in each divisor, the construction produces a unique such $X$, although the resulting threefold is K-unstable.

%In this section, we show that for a general K3 surface parametrized by each of these seven Noether--Lefschetz divisors, one can construct a K-unstable Gorenstein canonical Fano threefold $X$ of degree $22$ containing $S$ as an anticanonical divisor. 

On the other hand, there is also a global moduli-theoretic input for plt pairs.
Denote by $\cP^{\plt}$ the moduli stack of plt pairs $(X,S)$ such that $X$ is a Gorenstein canonical Fano degeneration of $V_{22}$ and $S\in |-K_X|$ is an ADE K3 surface. The following corollary follows immediately from \Cref{cor:surjective forgetful maps} and Zariski's Main Theorem.

\begin{cor}
The forgetful map
\[
\cP^{\plt} \longrightarrow \cF_{22}, \qquad (X,S) \longmapsto (S, -K_X|_S),
\]
is proper, surjective, birational, and has connected fibers.
\end{cor}

Taken together, these results suggest that the geometry of $V_{22}$ may be largely governed by its anticanonical K3 surface. This leads to the following conjecture.

\begin{conj}\label{conj:Fano-K3-inj}
For each ADE K3 surface $(S,L)\in \cF_{22}$, there exists a unique Gorenstein canonical Fano threefold $X$ admitting a $\bQ$-Gorenstein smoothing to $V_{22}$ such that $S \in |-K_X|$ and $L = -K_X|_S$. Moreover, the forgetful map
\[
\cP^{\plt} \longrightarrow \cF_{22}, \qquad (X,S) \longmapsto (S, -K_X|_S),
\]
is an isomorphism. 
\end{conj}

We note that an analogous statement holds in degree $4$, where $V_{22}$ is replaced by $\bP^3$ and quartic K3 surfaces arise as anticanonical divisors; see \cite[Theorem 1.1(3)]{ADL21}.

We conclude this discussion with a heuristic observation about the Picard rank of the relevant moduli spaces.

%We show that assuming Conjecture \ref{conj:Fano-K3-inj}, the possible Noether--Lefschetz divisor for which $X$ is K-semistable are the singular Fano threefolds constructed by Prokhorov \cite{Pro16}.
\begin{comment}

If $(S,L)$ is a smooth prime K3 surface of genus 12, then by \cite{Muk89,BKM25} $S$ is embedded into a smooth $V_{22}$. For non-BN general K3 surfaces $S$, there are special divisors on $S$ as classified by Greer-Li-Tian \cite{GLT15}, and we will use results from Johnsen-Knutsen \cite{JK04} and Hana \cite{Hana} to embed $S$ into a threefold inside a rational normal scroll, then perform some birational transformation.

Recall that $D_{d,n}^g$ represents the Noether--Lefschetz divisor in $\sF_{2g-2}$ parametrizing quasi-polarized K3 surface $(S,L)$ whose Picard lattice contains a primitive rank $2$ sublattice $\bZ L +\bZ \beta$ where 
\[
(L^2)=2g-2, \quad (L\cdot \beta)=d, \quad (\beta^2)=n.
\]
When $n=0$, the Noether--Lefschetz divisor $D_{d,0}^g$ parametrizes quasi-polarized elliptic K3 surfaces of genus $g$ with multisection index $d$.
\end{comment}

\begin{remark}\label{rmk:picrank}
By \cite[Theorem 0.1]{GLT15}, the Picard group of $\cF_{22}$ is generated by twelve Noether--Lefschetz divisors subject to one relation. Among them, a general K3 surface in the nodal divisor arises as an anticanonical section of a smooth $V_{22}$, while four divisors correspond to one-nodal $V_{22}$. For a general K3 surface in each of the remaining seven divisors,
Section~5 constructs a K-unstable Fano threefold containing it as an
anticanonical divisor.

This suggests that the Picard rank of $\cP^{\K,\ADE}$ should be four. Since $\Psi: \cP^{\K,\ADE}\to \cM^{\K}$ is generically a $\bP^{13}$-bundle, the class group of the K-moduli space $\fM^{\rm K}$ of $V_{22}$ is expected to have rank three.
% \begin{conj}\label{conj:picard rank 3}
% The class group of $\fM^{\rm K}$ has rank $3$.
% \end{conj}
\end{remark}

\subsection{General reconstruction}\label{subsec:general reconstruction}
Let $(S,L)$ be a smooth polarized K3 surface embedded in a Gorenstein canonical threefold $Y$ (not necessarily Fano) such that
\[
S \ \sim\  -K_Y ,\ \ \ \textup{and} \ \  -K_Y|_S \sim\  L + B,
\]
where $B$ is an effective divisor. View $B$ as a subscheme of $S$, i.e.\ an algebraic curve (possibly reducible and non-reduced) with at worst planar singularities, and hence also as a subscheme of $Y$. Let $\phi:\wt{Y} \coloneqq \Bl_B Y \to Y$ be the blowup of $Y$ along $B$, with exceptional divisor $F$, and let $\wt{S}$ denote the proper transform of $S$. Then $\wt{S} \simeq S$ since $B$ is a Cartier divisor on $S$.

\begin{thm}
With the above notation, $\wt{Y}$ is a Gorenstein canonical weak Fano threefold satisfying $\wt{S}\sim -K_{\wt{Y}}$. The anticanonical morphism $\psi:\wt{Y}\to X$ contracts every prime divisor $\wt{E}$ on $\wt{Y}$ disjoint from $\wt{S}$ to a non-terminal point $x_0\in X$, and it maps $\wt{S}$ isomorphically onto its image $\overline{S}$.  Under this isomorphism, one has $-K_X|_{\overline{S}} \simeq L$.
\end{thm}

\begin{proof}
Since $S\sim -K_Y$ and $B\subseteq S$, one has $\wt{S}\sim -K_{\wt{Y}}$, which is a Cartier divisor. Moreover, as $(Y,S)$ has purely log terminal singularities, so does $(\wt{Y},\wt{S})$, and hence $\wt{Y}$ has Gorenstein canonical singularities. Under the natural isomorphism $\wt{S}\simeq S$, one has
\[
-K_{\wt{Y}}|_{\wt{S}} \ \sim\  -K_Y|_S - B \ \sim\  L.
\]
In particular, $-K_{\wt{Y}}$ is nef, and it is big since $(-K_{\wt{Y}})^3 = (L^2) > 0.$ As $\tS$ is a smooth Cartier divisor, and $B\subseteq \wt{S}$, then $\tY\setminus F$ is isomorphic to $Y\setminus B$, and $\tY$ has $cA$-singularities along $F$; see e.g.\ \cite[Proof of Lemma 5.11]{LZ25}. Thus $\wt{Y}$ is a Gorenstein canonical weak Fano threefold. On the other hand, since $\wt{E}\cap \wt{S}=\emptyset$, one has $-K_{\wt{Y}}|_{\wt{E}}\sim 0$, and hence $\psi$ contracts $\wt{E}$ to a point $x_0$.%\footnote{\AS{I do not understand why the anticanonical morphism cannot contract divisors to curves- why can't we have $S^2E=0$ if $S$ and $E$ are not disjoint on $\wt Y$?}} %We claim that these are precisely the curves contracted by the anticanonical morphism. Indeed, since $-K_{\wt{Y}} \sim \phi^*L + \wt{E}$, if $C\subseteq \wt{Y}$ is a curve not contained in $\wt{E}$, then
%\[(-K_{\wt{Y}}\cdot C) \ =  \ (\phi^*L\cdot C)+(\wt{E}.C) \ \ge\  (\phi^*L\cdot C)\ >\  0,\]
%since $L$ is ample on $Y$. It follows that the anticanonical morphism $\psi:\wt{Y}\to X$ contracts exactly $\wt{E}$.
Since $\psi$ is crepant, the image point $x_0\in X$ is necessarily singular, and the singularity of $X$ at $x_0$ is non-cDV by \cite[Theorem 2.19(1)]{Liu25}. Finally, since the normal bundle $N_{\wt{S}/\tY}\simeq L$ is ample, the morphism $\psi$ restricts to an isomorphism $\wt{S}\xrightarrow{\sim}\overline{S}$, and under this identification one has $-K_X|_{\overline{S}} \simeq L$.
\end{proof}

\begin{rem}
   If $E$ is an effective divisor on $X$ whose support does not contain $S$ such that $E|_S\leq B$, then the proper transform of $E$ on $\tY$ is disjoint with $\wt{S}$.%\footnote{\YL{I think we only need $E|_S\leq B$, this is probably if and only if condition for $\tE$ to be disjoint from $\tS$.}}
\end{rem}

Therefore, we obtain a Gorenstein canonical Fano threefold $X$ of volume $(L^2)$ containing $S$ as an anticanonical divisor and satisfying $-K_X|_S \simeq L$. Since $X$ is Gorenstein, the singular point $x_0$ cannot be a $\frac{1}{2}(1,1,1)$ quotient singularity. It then follows from \cite[Theorem 1.3(2)]{Liu25} that if $\vol(X)\geq 22$, then $X$ is K-unstable.

\begin{comment}
\subsection{Prokhorov's nodal Fano threefolds}

There are four types of Prokhorov's Fano threefolds, each general member contains one $A_1$-singularity \cite{Pro16}. They correspond to four  Noether--Lefschetz divisors according to the following table. By \cite{DK25} we know that each general member of Prokhorov's list is K-stable.

\begin{table}[htbp!]\renewcommand{\arraystretch}{1.5}
\caption{Prokhorov's list of nodal $V_{22}$}\label{table:singularities}
\begin{tabular}{|c|c|l|c|}
\hline 
Noether--Lefschetz divisor & \cite{Pro16}'s type   & Small resolution & K-ss\\ \hline \hline 
$D_{11,4}^{12}$ & (I)  & \specialcell{blowup of $\bP^3$ along a smooth rational quintic\\ curve not contained in a
quadric} &   Yes $\star$ \\ \hline
$D_{9,2}^{12}$ & (II) & \specialcell{blowup of $Q^3$ along a smooth rational quintic \\curve not contained in a hyperplane section} & Yes $\star$ \\\hline
$D_{6,0}^{12}$ &  (III)  & \specialcell{blowup of $V_5$ along a smooth rational quartic curve}     &  Yes $\star$ \\ 
    \hline

$D_{5,0}^{12}$ &  (IV) & \specialcell{del Pezzo fibration of degree 5}  &  Yes $\star$\tablefootnote{$\star$ means true for general member. Type (IV) is done in \cite{ACC+}.}\\ \hline

\end{tabular}
\end{table}
\end{comment}

\subsection{Application to Noether--Lefschetz divisors on $\cF_{22}$}

Now we apply the construction in \Cref{subsec:general reconstruction} to the seven Noether--Lefschetz divisors on $\cF_{22}$. For a general member $(S,L)$ in each of these divisors, we construct a Gorenstein canonical Fano threefold $X$ containing $S$ as an anticanonical divisor such that $-K_X|_S \sim L$. Moreover, each such $X$ is K-unstable. We note that some of the constructions are similar to those in \cite{ADL21,JR06,CS05}.

\subsubsection{\texorpdfstring{Unigonal divisor $\cD_{1,0}^{22}$}{\cD_{1,0}^{12}}}

We first start from the unigonal K3 surfaces, i.e.\ surfaces $(S,L)$ parametrized by $\cD_{1,0}^{22}$. Let $\pi: S\to \bP^1$ be the elliptic fibration with a section $B\subset S$. Let $F$ be a fiber of $\pi$. Then we know that $L=12F+B$. From e.g.\ \cite[Section 11]{Huy16}, we know that 
\[ \pi_* \cO_S(B) \simeq \cO_{\bP^1}, \quad \pi_* \cO_S(2B) \simeq \cO_{\bP^1}(-4) \oplus \cO_{\bP^1}, \quad \pi_* \cO_S(3B) \simeq \cO_{\bP^1}(-6) \oplus \cO_{\bP^1}(-4) \oplus \cO_{\bP^1}. \]  Let $\cE\coloneqq \pi_*\cO_S(3B)$ be the rank three bundle on $\bP^1$ and \[p\ : \ Y\ \coloneqq \ \Proj_{\bP^1} \Sym\,\cE \ \longrightarrow \ \bP^1 \] be the weighted projective bundle, where we assign degrees $1,2,3$ for $\cO_{\bP^1}$, $\cO_{\bP^1}(-4)$, and $\cO_{\bP^1}(-6)$. Then $\omega^*_Y\simeq \cO_{Y}(12)\otimes p^*\mathcal{O}_{\bP^1}(6)$, and $S$ is embedded into $Y$ as an anti-canonical divisor; see \cite[Section 4.3]{ADL21}. Let $E$ be the only divisor in $|\cO_Y(1)|$. Then \[B\ =\ E|_S \ \ \ \textup{and} \ \ \ -K_Y|_S\ \sim\  12 F+ 6B \ \sim\   L+5B.\] By the construction in \Cref{subsec:general reconstruction}, one can take $\tY$ to be the blowup of $Y$ along the non-reduced curve $S\cap 5E$, or equivalently, the $(5,1)$-weighted blowup along the divisors $(S, E)$. Then the anticanonical ample model $X$ of $\tY$ is a K-unstable Gorenstein canonical Fano threefold containing $S$ as an anticanonical divisor.  %Note that $\tY$ depends on the choice of $S$ which results a $5$-dimensional moduli. Denote by $\tS$ and $\tE$ the strict transforms of $S$ and $E$ in $\tY$. It is clear that $\tS\simeq S$. 

\subsubsection{Hyperelliptic divisor $\cD_{2,0}^{22}$}
Let $(S,L)$ be a \emph{general} elliptic K3 surface in $\cD_{2,0}^{22}$, the \emph{hyperelliptic divisor}. Then $S$ admits an elliptic fibration $\pi:S\rightarrow \bP^1$ with fiber class $F$. We have that $(L-6F)^2=-2$, so there is a unique bisection $(-2)$-curve $B\in |L-6F|$. By Leray spectral sequence, we have
\[
0 \ \longrightarrow \  H^1(\bP^1, \pi_*\cO_S(mB))\ \longrightarrow \  H^1(S, \cO_S(mB))\ \longrightarrow \   H^1(\bP^1, R^1\pi_*\cO_S(mB)).
\]
If $m\geq 1$, then  $R^1\pi_*\cO_S(mB)=0$, and hence $H^1(\bP^1, \pi_*\cO_S(mB))\simeq H^1(S, \cO_S(mB))$ which has dimension $m^2-1$ by the Riemann-Roch. Thus one can show that 
\[
\pi_*\cO_S(B)\simeq \cO_{\bP^1}\oplus \cO_{\bP^1}(-1),\quad \pi_*\cO_S(2B)\simeq \cO_{\bP^1}\oplus \cO_{\bP^1}(-1)\oplus \cO_{\bP^1}(-2)\oplus \cO_{\bP^1}(-3).
\]
Let $\cE\coloneqq \cO_{\bP^1}\oplus \cO_{\bP^1}(-1)\oplus\cO_{\bP^1}(-3)$ where we assign degrees $1,1,2$ for the direct summands.
Let $Y\coloneqq \Proj \Sym \cE$ be the $\bP(1,1,2)$-bundle over $\bP^1$, which contains $S$ as an anti-canonical divisor. Let $E$ be the unique effective divisor in $|\cO_Y(1)|$. Then $B=E|_S$. Let $\tY$ be the $(3,1)$-weighted blow up of $Y$ along the divisors $(S,E)$. Note that $-K_Y\sim 6F_Y+4E$, where $F_Y$ is the fiber class of $Y\rightarrow \bP^1$, and $(E+6F_Y)|_S\sim L$. Then by the construction in \Cref{subsec:general reconstruction}, one can take $\tY$ to be the blowup of $Y$ along the non-reduced curve $S\cap 3E$, and the anticanonical ample model $X$ of $\tY$ is a K-unstable Gorenstein canonical Fano threefold containing $S$ as an anticanonical divisor.

\subsubsection{Trigonal divisor $\cD_{3,0}^{22}$}

Let $(S,L)$ be a general polarized K3 surface in $\cD_{3,0}^{22}$. Then $S$ admits an elliptic fibration $\pi \colon S \to \bP^1$ with fiber class $F$, and one has 
$(L - 4F)^2 = -2$. Thus there exists a unique $(-2)$-curve $B \in |L - 4F|$, which is a trisection of $\pi$.

By \cite{Hana}, since $(S,L)$ is general, it admits an embedding into the scroll
\[
Y = \bP(\cE), \qquad \textup{where} \qquad  \cE = \cO_{\bP^1}(4) \oplus \cO_{\bP^1}(3) \oplus \cO_{\bP^1}(3).
\]
Let $E \subseteq Y$ be the distinguished $\bP^1$-bundle over $\bP^1$ corresponding to the inclusion $\cO_{\bP^1}(4) \hookrightarrow \cE$, and let $F_Y$ denote a fiber of the projection $Y \to \bP^1$. Then the divisor classes satisfy
\[
-K_Y \sim \cO_Y(3) - 8F_Y, \ \ 
E \sim \cO_Y(1) - 4F_Y, \ \ 
L = \cO_Y(1)|_S, \ \
B = E|_S, \ \ \textup{and} \ \ -K_Y|_S = L + 2B.
\] By the construction in \Cref{subsec:general reconstruction}, one can take $\tY$ to be the blow-up of $Y$ along the non-reduced curve $S \cap 2E$. The anticanonical ample model $X$ of $\tY$ is then a Gorenstein canonical Fano threefold that is K-unstable and contains $S$ as an anticanonical divisor.

\begin{comment}

\begin{prop}
With the above notation, $\tY$ is a Gorenstein canonical weak Fano threefold. The anti-canonical model $\tY\to X$ contracts $\tE$ to an isolated singularity $x$ on $X$. Moreover, $X$ is a \textcolor{blue}{$\bQ$-Gorenstein smoothable} Gorenstein canonical Fano threefold of degree $22$ and Picard rank $1$, and $\hvol(x,X)\leq \frac{3}{2}$.
\end{prop}

\begin{proof}
The proof is similar to before, just notice that $\tS|_{\tS}=4\tF+\tB$ which is the pull-back of $L$. In this case, $(\tE,\Delta_{\tE})\simeq (\bP^1\times\bP^1, \frac{1}{2}B)$ where $B\in |\cO(1,3)|$ is a smooth rational curve. Thus $\hvol(x,X)\leq (-K_{\tE}-\Delta_{\tE})^2=\frac{3}{2}$. \textcolor{blue}{Need to think about $\bQ$-Gorenstein smoothability.}
\end{proof}
\end{comment}

\subsubsection{Tetragonal divisor $\cD_{4,0}^{22}$}

Let $(S,L)$ be a general polarized K3 surface in $\cD_{4,0}^{22}$. Then $S$ admits an elliptic fibration $\pi \colon S \to \bP^1$ with fiber class $F$, and one has $
(L - 3F)^2 = -2$. Thus there exists a unique $(-2)$-curve $B \in |L - 3F|$, which is a degree $4$ multisection of $\pi$. By \cites{JK04, Hana}, $S$ admits an embedding into the scroll $\bP(\cE)$, where
\[
\cE \  =\  \cO_{\bP^1}(3) \oplus \cO_{\bP^1}(2) \oplus \cO_{\bP^1}(2) \oplus \cO_{\bP^1}(2).
\]
Let $H_{\cE} \coloneqq \cO_{\bP(\cE)}(1)$ and let $F_{\cE}$ denote a fiber of the projection $\bP(\cE) \to \bP^1$. By \cites{JK04, Hana}, the ideal sheaf of $S$ admits a locally free resolution
\[
0 \longrightarrow \cO_{\bP(\cE)}(-4H_{\cE} + 11F_{\cE})
\longrightarrow  \cO_{\bP(\cE)}(-2H_{\cE} + 4F_{\cE}) \oplus \cO_{\bP(\cE)}(-2H_{\cE} + 3F_{\cE})
\longrightarrow  \cI_{S/\bP\cE} \longrightarrow  0.
\]
In particular, $S$ is a complete intersection of two divisors in the linear systems $|2H_{\cE} - 4F_{\cE}|$ and $|2H_{\cE} - 3F_{\cE}|$. It follows that there exists a unique divisor $
Y \in |2H_{\cE} - 4F_{\cE}|$ containing $S$, and the induced morphism $Y \to \bP^1$ is a quadric surface fibration. From this construction, one has $H_{\cE}|_S = L$, and therefore
\[
-K_Y|_S \sim (2H_{\cE} - 3F_{\cE})|_S = 2L - 3F \sim L + B.
\] Since $\cE$ is of type $(3,2,2,2)$, there exists a unique divisor $E_{\cE} \in |H_{\cE} - 3F_{\cE}|$. Let $E \coloneqq E_{\cE}|_Y$. Then $E \to \bP^1$ is a conic bundle, and one has 
$E|_S = B$. By the construction in \Cref{subsec:general reconstruction}, one can take $\tY$ to be the blow-up of $Y$ along $S \cap E$. The anticanonical ample model $X$ of $\tY$ is then a Gorenstein canonical Fano threefold that is K-unstable and contains $S$ as an anticanonical divisor.

\begin{comment}
\begin{prop}
With the above notation, $\tY$ is a Gorenstein canonical weak Fano threefold. The anti-canonical model $\tY\to X$ contracts $\tE$ to an isolated singularity $x$ on $X$. Moreover, $X$ is a $\bQ$-Gorenstein smoothable Gorenstein canonical Fano threefold of degree $22$ and Picard rank $1$, and $\hvol(x,X)\leq 8$.
\end{prop}

\begin{proof}
The proof is similar to before, just notice that $\tS|_{\tS}=3\tF+\tB$ which is the pull-back of $L$. In this case, $(\tE,\Delta_{\tE})\simeq (\bP^1\times\bP^1, 0)$. Thus $\hvol(x,X)\leq (-K_{\tE}-\Delta_{\tE})^2=8$.
\end{proof}
\end{comment}

\subsubsection{Tritangent divisor $\cD_{7,2}^{22}$}
Let $(S,L)$ be a \emph{general} polarized K3 surface in $\cD_{7,2}^{22}$. Then there exists a divisor $D$ on $S$ such that $(L\cdot D)=7$ and $(D^2)=2$. It follows that
\[
(L-3D)^2=-2 \quad \text{and} \quad (4D-L)^2=-2.
\]
Let $B \in |L-3D|$ and $ B' \in |4D-L|$ be the corresponding $(-2)$-curves. Then
\[
D \sim B+B', \qquad (D\cdot B)=(D\cdot B')=1, \qquad (B\cdot B')=3.
\]
In particular, the linear system $|D|$ induces a double cover $\pi:S \to \bP^2$ branched along a sextic curve $C_6 \subset \bP^2$, and $B+B'$ is the pullback of a tritangent line $\ell \subset \bP^2$ to $C_6$. As $S$ is a double cover of $\bP^2$, we may embed $S$ into $Y \coloneqq \bP(1,1,1,3)$ as an anticanonical divisor, where $S$ is defined by an equation
\[
w^2 = f_6(x,y,z),
\]
with $C_6=V(f_6)$. %Let $E_1 \simeq \bP(1,1,3)$ denote the cone over the line $l$. Then $E_1|_S = B+B'$. 
Moreover, we have
\[
-K_Y|_S \ \sim\  6D \ \sim\  L + 2B + 3B'
\]
in $Y$. By the construction in \Cref{subsec:general reconstruction}, one can take $\tY$ to be the blow-up of $Y$ along the non-reduced and reducible curve $2B+3B'$. The anticanonical ample model $X$ of $\tY$ is then a Gorenstein canonical Fano threefold that is K-unstable and contains $S$ as an anticanonical divisor.
\subsubsection{The conic divisor $\cD_{8,2}^{22}$}

Let $(S,L)$ be a \emph{general} polarized K3 surface in $\cD_{8,2}^{22}$. Then there exists a divisor $D$ on $S$ such that $(L\cdot D)=8$ and $(D^2)=2$. It follows that
\[
(L-D)^2=8 \qquad \text{and} \qquad (L-2D)^2=-2.
\]
Let $F \in |L-D|$ and $B \in |L-2D|$. Then $B$ is a $(-2)$-curve and $F^2=8$. The linear system $|F|$ defines an embedding $
S \hookrightarrow \bP^5$ as a complete intersection of three quadric hypersurfaces. Under this embedding, the curve $B \subset S$ is realized as a smooth conic in $\bP^5$. Let
\[
\Pi\ \coloneqq\ \langle B \rangle \ \simeq\  \bP^2\ \subseteq \ \bP^5
\]
be the plane spanned by $B$. Among the $2$-dimensional linear system of quadrics containing $S$, there is a distinguished pencil $\{Q_\lambda\}$ consisting of quadrics that contain the plane $\Pi$. Let $Y$ be the base locus of this pencil, which is a quartic del Pezzo threefold with three $A_1$-singularities located on $\Pi$; see \Cref{lem:containing-a-plane}. Then $S \in |-K_Y|$. Let
$\pi:\hY \to Y$ be a small resolution of the $A_1$-singularity such that the strict transform $\widehat{\Pi}$ of $\Pi$ contains the exceptional rational curves. Then $\widehat{\Pi}$ is a smooth del Pezzo surface of degree $6$ and is Cartier in $\hY$. Set $\hB\coloneqq \widehat{\Pi}\cap \hS$. Then $\hS$ is isomorphic to $S$, under which one has $\hB=B$ and $-K_{\hY}|_{\hS}= L+B$. By the construction in \Cref{subsec:general reconstruction}, one can take $\tY$ to be the blow-up of $\hY$ along $\hB$. The anticanonical ample model $X$ of $\tY$ is then a Gorenstein canonical Fano threefold that is K-unstable and contains $S$ as an anticanonical divisor.
\begin{comment}

\begin{prop}
With the above notation, $\tY$ is a Gorenstein canonical weak Fano threefold. The anti-canonical model $\tY\to X$ contracts $\tE$ to an isolated singularity $x$ on $X$. Moreover, $X$ is a \textcolor{blue}{$\bQ$-Gorenstein smoothable} Gorenstein canonical Fano threefold of degree $22$ and Picard rank $1$, and $\hvol(x,X)<8$.
\end{prop}

\begin{proof}
Notice that $S|_S=2F\sim L+B$, so $\tS|_{\tS}$ is the pull-back of $L$. Then the proof is similar to before. We have $\hvol(x,X)\leq (-K_{\tE})^2=(-K_{\bF_1})^2=8$. Moreover, the equality cannot hold by \cite{LX20} because $\bF_1$ is K-unstable.
\end{proof}
\end{comment}

\subsubsection{The nodal quadric divisor $\cD_{10,4}^{22}$}

Let $(S,L)$ be a \emph{general} polarized K3 surface in $\cD_{10,4}^{22}$. Then there exists a divisor $D$ on $S$ such that $ (L\cdot D)=10$ and $(D^2)=4$. Then
\[
(L-D)^2=6, \qquad (L-2D)^2=-2, \qquad (L-D.L-2D)=0.
\]
Let $F \in |L-D|$ and $B \in |L-2D|$. Then $B$ is a $(-2)$-curve and $F^2=6$. The linear system $|F|$ induces a morphism $
\varphi_{|F|}\colon S \to \bP^4$ which is birational onto its image. The image $S' \coloneqq \varphi_{|F|}(S)$ is a $(2,3)$-complete intersection K3 surface, and $\varphi_{|F|}$ contracts $B$ to a single $A_1$-singularity $p \in S'$. Since $(S,L)$ is general, we may assume that $S'$ is contained in a smooth quadric threefold $Q \subset \bP^4$. Let $Y \coloneqq \Bl_p Q$ be the blow-up of $Q$ at $p$, and denote by $E_1 \subset Y$ the exceptional divisor. Then 
$S \hookrightarrow Y$ as an anticanonical divisor, and $E_1|_S = B$. Let $E_2'$ be the intersection of $Q$ with its tangent hyperplane $T_p Q$. Then $E_2'$ is isomorphic to a quadric cone. Let $E_2$ be the strict transform of $E_2'$ in $Y$. The pull-back of $E_2'$ to $Y$ is linearly equivalent to $E_2 + 2E_1$, and hence
\[
E_2|_S \ \sim\  F - 2E_1|_S \ =\  F - 2B \ \sim\  3D - L.
\]
Set $C \coloneqq E_2|_S$.
Then $(C^2)=-2$, $(C. B)=4$, and $D \sim B + C$. In particular, $D$ decomposes uniquely as the sum of two $(-2)$-curves. Let $B_2 \coloneqq E_1 \cap E_2$, which is a smooth conic in $E_1 \simeq \bP^2$. The three surfaces $S, E_1,$ and $E_2$ intersect transversely, with four triple intersection points. One has $$-K_Y|_S\ \sim\ (3E_2+4E_1)|_S \ \sim \ L+B+C .$$ By the construction in \Cref{subsec:general reconstruction}, one can take $\tY$ to be the blow-up of $Y$ along the reducible curve $B\cup C$. The anticanonical ample model $X$ of $\tY$ is then a Gorenstein canonical Fano threefold that is K-unstable and contains $S$ as an anticanonical divisor. 

\medskip

At the end of this section, we propose the following conjecture. By explicit computation, we are able to verify the conjecture for several of the seven divisors, providing supporting evidence for its validity. However, since we are currently unable to resolve the conjecture in full, and as it is not the primary focus of this paper, we leave it for future investigation.

\begin{conj}\label{conj:qG-smoothing}
The Gorenstein canonical Fano threefolds constructed for the above seven Noether--Lefschetz divisors are of Picard rank 1 and admit $\bQ$-Gorenstein smoothing to $V_{22}$.
\end{conj}

\begin{remark}
The reconstruction of a Fano threefold from a polarized K3 surface need not be unique in general.
For instance, a polarized K3 surface in the trigonal divisor $\cD_{3,0}^{22}$ may be embedded into several rational normal scrolls $Y=\bP(\cE)$ with $L\simeq \cO_{\bP\cE}(1)|_S$ (see \cite[Table on p.~103]{JK04}). Applying the reconstruction procedure of \Cref{subsec:general reconstruction} to these scrolls produces distinct families of Gorenstein canonical Fano threefolds of degree $22$ containing $S$ as an anticanonical divisor.

This phenomenon occurs only for K3 surfaces lying in deeper Noether--Lefschetz strata inside the divisor; for a general K3 surface in each Noether--Lefschetz divisor considered above, the reconstruction produces a unique Fano threefold. Nevertheless, we expect that uniqueness should hold after restricting to Fano threefolds that admit a $\bQ$-Gorenstein smoothing to $V_{22}$, which motivates Conjecture~\ref{conj:Fano-K3-inj}.
\end{remark}

\smallskip
\appendix
\section[tocentry={K-moduli of the Fano threefolds \textnumero2.16}]{K-moduli of the Fano threefolds \textnumero2.16}\label{appendix A}

In this appendix, we study the K-moduli stack of the family \textnumero2.16 of Fano threefolds of volume $22$. 
A smooth member of this family is obtained by blowing up a smooth $(2,2)$-complete intersection in $\mathbb{P}^5$ along a smooth conic curve. 
We prove the following result, which is used in \Cref{lem:exclude the other families}.%\footnote{\YL{mention related work from many people including Vanya and Junyan. Also say a few words on how far are we from a complete description of K-moduli of 2.16.}}

\begin{theorem}\label{thm:K-ss limit of 2.16}
    Every K-semistable degeneration $X$ of Fano threefolds {\rm\textnumero2.16} is the blowup of a $(2,2)$-complete intersection in $\bP^5$ along a conic. Moreover, the K-moduli stack $\mtc{M}^{\K}_{\textup{\textnumero2.16}}$ is a smooth connected component of $\mtc{M}^{\K}_{3,22}$.
\end{theorem}

In the recent work \cite{LastFano}, the authors prove the K-stability of certain Fano threefolds in the family \textnumero2.16 and outline a strategy to describe the K-moduli of this family. \Cref{thm:K-ss limit of 2.16} confirms \cite[Conjecture~6.1.1]{LastFano}, which is the most technical step in this approach. To complete the description of the K-moduli, one needs to study the parameter space $W$ of quartic del Pezzo--conic pairs in $\bP^5$ and compute the CM line bundle $\cL_{\rm CM}$ associated with the universal family obtained by blowing up quartic del Pezzo threefolds along conics. The expected outcome is an identification of the K-moduli stack (resp.\ space) with the VGIT quotient
\[
[W^{\rm ss}/\PGL(6)]
\qquad (\text{resp.\ } W\sslash_{\cL_{\rm CM}}\PGL(6)),
\]
together with a study of the corresponding GIT (semi/poly)stability; see \cite[Conjecture~6.3.1]{LastFano}. 
As this direction is tangential to the main focus of the present paper, we leave it for future work.

\subsection{Geometry of blowups of (2,2)-complete intersections}

In this subsection, we study the geometry of the blowup of a $(2,2)$-complete intersection in $\bP^5$ along a conic, as preparation for the next subsection. Throughout this subsection, we assume that
\begin{itemize}
    \item $V = Q_1 \cap Q_2$ is a Gorenstein canonical $(2,2)$-complete intersection in $\bP^5$, and
    \item $C \subseteq V$ is a conic curve, i.e.\ a subscheme of $\bP^5$ with Hilbert polynomial $p_C(t) = 2t + 1$.
\end{itemize}
Such a curve is contained in a unique $2$-plane and is either a smooth conic, the nodal union of two lines, or a double line. We further assume that $V$ is generically smooth along every irreducible component of $C$ and has hypersurface singularities along $C$.

\begin{lem}\label{lem:complete-intersection-after-blowup}
Let $\pi:\Bl_{C}\bP^5\rightarrow \bP^5$ be the blow-up and let $F$ be the exceptional divisor. Then the following hold:
\begin{enumerate}
    \item[\textup{(1)}] the blow-up $X\coloneqq\Bl_{C}V$ coincides with the proper transform $\wt{V}$ of $V$ under $\pi$; and 
    \item[\textup{(2)}] the scheme-theoretic intersection $E\coloneqq F\cap X$ is a complete intersection in $F$.
\end{enumerate}
\end{lem}

\begin{proof}
For (1), both $X=\Bl_C V$ and $\wt{V}$ are integral subvarieties of $\Bl_C\bP^5$, and there exists an open subscheme $U\subseteq \Bl_C\bP^5$ contained in both of them. Hence they coincide. Consequently, $X$ is a complete intersection of two divisors of class $2H-F$, where $H=\pi^*\mtc{O}_{\bP^5}(1)$. This proves (2), since $F\cap X$ has dimension $2$.
\end{proof}

\begin{lemma}\label{lem:blowup-conic-weak-Fano}
Set $Y\coloneqq \Bl_{C}\mathbb{P}^5$. Then $Y$ is weak Fano. 
Moreover, any curve on $Y$ intersecting $-K_Y$ trivially is contained in the strict transform of the $2$-plane spanned by $C$.
\end{lemma}

\begin{proof}
Let $\pi\colon Y\to \mathbb{P}^5$ be the blow-up with exceptional divisor $E$, and set 
$H\coloneqq \pi^*\mathcal{O}_{\mathbb{P}^5}(1)$. Let $P\subset \mathbb{P}^5$ be the unique plane containing $C$, and let $\widetilde{P}$ be its proper transform in $Y$. Since $C$ has codimension $3$, we have $-K_Y \sim 6H-3E$. The normal bundle of $C$ in $\mathbb{P}^5$ is
\[
N_{C/\mathbb{P}^5}\simeq \mathcal{O}_C(1)^{\oplus 3}\oplus \mathcal{O}_C(2),
\]
so under the identification $E\simeq \mathbb{P}N^*\simeq \mathbb{P}N^*(2)$, the restriction
$(2H-E)|_E$ corresponds to $\mathcal{O}_{\mathbb{P}N^*(2)}(1)$. 
Since $N^*(2)\simeq \mathcal{O}_C(1)^{\oplus 3}\oplus \mathcal{O}_C$ is nef, the tautological bundle $\mathcal{O}_{\mathbb{P}N^*(2)}(1)$ is also nef. Moreover, there is a unique section of $\mathbb{P}N^*(2)\to C$ on which 
$\mathcal{O}_{\mathbb{P}N^*(2)}(1)$ has degree zero, namely the section corresponding to the unique nontrivial morphism $N^*(2)\to \mathcal{O}_C$; this section lies in $\widetilde{P}$.

Suppose there exists an integral curve $\Gamma\subset Y$ with $(-K_Y\cdot \Gamma)\le 0$ that is not contained in $E$. Then by degree considerations, its image $\overline{\Gamma}\coloneqq \pi(\Gamma)$ must lie in every hyperplane of $\mathbb{P}^5$ containing $\overline{\Gamma}$, hence $\overline{\Gamma}\subseteq P$. 

Finally, since $(6H-3E)|_{\widetilde{P}}\sim 0$, the divisor $6H-3E$ is nef on $Y$. It is also big because $(2H-E)^3>0$. Consequently, the ample model of $Y$ with respect to $-K_Y$ contracts precisely $\widetilde{P}$ to a point, and no other curves.
\end{proof}

\begin{lem}\label{lem:weakFano-plane}
 If $X\coloneqq \Bl_C V$ is not Fano, then
    $V$ contains the plane $P$ spanned by $C$, the variety $X$ is weak Fano,
    and the anticanonical ample model $X\to \overline{X}$ contracts precisely
    the proper transform $\widetilde{P}$ of $P$ to a point.
\end{lem}

\begin{proof}
    Let $Y\coloneqq \Bl_C\mathbb P^5$. Then $X\subseteq Y$ is a complete
    intersection of two divisors of class $-\tfrac{1}{3}K_Y$, and hence
    $-K_X\sim \tfrac{1}{3}(-K_Y)|_X$ is big and nef. By \Cref{lem:blowup-conic-weak-Fano}, if there exists a curve
    $\Gamma\subset X$ with $(-K_X\cdot \Gamma)=0$, then $\Gamma$ is contained in
    the proper transform $\widetilde{P}$ of the plane $P$ spanned by $C$; equivalently,
    its image $\overline{\Gamma}=\pi(\Gamma)\subset V$ lies in $P$, where
    $\pi\colon X\to V$ is the blowup map. Moreover, as $V$ is generically smooth along $C$, then $\ove{\Gamma}\neq C$.

    Now let $C'\subset P$ be any curve. Since $C\subset P$ is a conic,
    we have
    \[
        C'\cdot Q_i = \deg C'\cdot \deg C > 2\deg C'
    \]
    for $i=1,2$. Thus $C'$ meets $Q_i$ in more than $2\deg C'$ points, which forces $C'\subset Q_i$. Hence $P\subset Q_1\cap Q_2=V$, and by \Cref{lem:blowup-conic-weak-Fano} the anticanonical morphism of $X$ contracts precisely $\widetilde{P}$ to a point.
\end{proof}

\begin{lemma}\label{lem:containing-a-plane}
Let $P \subseteq \bP^5$ be a $2$-plane, and let $V$ be a general $(2,2)$-complete intersection containing $P$. Then the singular locus of $V$ consists of three $A_1$-singularities, all contained in $P$.
\end{lemma}

\begin{proof}
Consider the blow-up $f\colon Y\coloneqq \Bl_{P}\mathbb{P}^5\to \mathbb{P}^5$, and let $H=f^*\mathcal{O}_{\mathbb{P}^5}(1)$ and $E$ be the exceptional divisor. A general $(2,2)$-complete intersection $V\subset \mathbb{P}^5$ containing $P$ is the image of a general complete intersection $\widetilde{V}\subset Y$ of two globally generated divisors of class $2H-E$. In particular, $\widetilde{V}$ is smooth, and hence $V$ is smooth away from $P$. Since $E\simeq \mathbb{P}^2\times \mathbb{P}^2$, the restriction $(2H-E)|_E$ has class $\mathcal{O}_{\mathbb{P}^2\times \mathbb{P}^2}(1,1)$. Thus $\widetilde{P}\coloneqq \widetilde{V}\cap E$
is a smooth Fano surface with
\[
(-K_{\widetilde{P}})^2
\ =\ (\mathcal{O}_{\mathbb{P}^2\times \mathbb{P}^2}(1,1)^4) \ =\ 6,
\]
so $\widetilde{P}$ is a smooth del Pezzo surface of degree $6$. Consequently, the morphism $\widetilde{P}\to P\simeq \mathbb{P}^2$ is the blow-up of three general points. The three exceptional curves of this blow-up have trivial intersection with $-K_Y$, and their images in $V$ are precisely three $A_1$-singularities on $P$.
\end{proof}

\begin{cor}\label{cor:blowup-conic-unstable}
Let $V$ contains a plane $P$, and let $C\subset P$ be a conic. Then the blow-up $X\coloneqq \Bl_{C}V$ is K-unstable.
\end{cor}

\begin{proof}
By the openness of K-semistability, it suffices to prove the statement for a general $(2,2)$-complete intersection $V$ containing $P$ and a general conic $C\subset P$. In this case, by \Cref{lem:containing-a-plane}, the conic $C$ lies in the smooth locus of $V$. By \Cref{lem:blowup-conic-weak-Fano}, the anticanonical model morphism $X\longrightarrow \overline{X}$ is crepant and contracts the strict transform $\widetilde{P}\simeq \mathbb{P}^2$ to a point $p\in \overline{X}$, which is an isolated singularity. If $\overline{X}$ were K-semistable, then all its isolated singularities would be terminal by \Cref{thm:generalelephant}. This contradicts the fact that $X\to \overline{X}$ is crepant and contracts a divisor. Hence $X$ is K-unstable.
\end{proof}

\begin{comment}

\begin{lemma}\label{lem:blowup-nonreduced-conic-unstable}
Let $V\subset \mathbb{P}^5$ be a $(2,2)$-complete intersection and let 
$C\subset V$ be a non-reduced conic curve. Then the blow-up $X\coloneqq \Bl_{C}V$ is K-unstable.
\end{lemma}

\begin{proof}
A direct local calculation shows that $X$ is singular along a rational curve 
$\Gamma$ with $(-K_X\cdot \Gamma)=1$. By \Cref{thm:sing-line}, any such variety is 
K-unstable. Hence $X$ is K-unstable.
\end{proof}
\end{comment}

\subsection{K-semistable limits of family \textnumero2.16}

Let $X$ be a K-semistable $\bQ$-Fano variety that admits a $\bQ$-Gorenstein smoothing $\pi:\mts{X}\rightarrow T$  over a smooth pointed curve $0\in T$ such that $\mts{X}_0\simeq X$ and every fiber $\mts{X}_t$ over $t\in T\setminus \{0\}$ is a smooth Fano threefold in the family №2.16. Up to a finite base change, we may assume that the restricted family $\mts{X}^{\circ}\rightarrow T^{\circ}\coloneqq T\setminus \{0\}$ is isomorphic to $\Bl_{\mts{C}^{\circ}}(\mts{V}^\circ)$, where $$\mts{V}^{\circ}\subseteq \bP^5\times T^\circ \rightarrow T^{\circ}$$ is a family of $(2,2)$-complete intersections in $\mathbb{P}^5$, and $\mts{C}^{\circ}\subseteq \bP^5\times T^\circ$ is a family of smooth conic curves. Let $\mts{E}^{\circ}\subseteq \mts{X}^{\circ}$ be the exceptional divisor of $\sX^\circ\rightarrow \sV^\circ$, and $\cL^{\circ}$ be a divisor on $\mts{X}^{\circ}$ which is linearly equivalent to the pull-back of $\mtc{O}_{\mathbb{P}^5}(1)$. Then $\mts{E}^{\circ}$ (resp. $\cL^{\circ}$) extends to $\mts{E}$ (resp. $\cL$) on $\mts{X}$ as a Weil divisor on $\mts{X}$ by taking Zariski closure. We denote by $E$ (resp. $L$) the restriction of $\mts{E}$ (resp. $\cL$) on the central fiber $\mts{X}_0\simeq X$.

Let $\theta:\mts{Y}\rightarrow \mts{X}$ be a small $\mathbb{Q}$-factorialization of $\mts{X}$. Since $\mts{X}$ is klt and $K_{\mts{Y}}=\theta^* K_{\mts{X}}$, we know that $\mts{Y}$ is $\bQ$-factorial of Fano type over $\mts{X}$. By \cite{BCHM10} we can run a minimal model program for $\theta^{-1}_{*}\cL$ on $\mts{Y}$ over $\mts{X}$. As a result, we obtain a log canonical model $\mts{Y} \dashrightarrow \wt{\mts{X}}$ that fits into a commutative diagram
$$\xymatrix{
 & \wt{\mts{X}} \ar[rr]^{f} \ar[dr]_{\wt{\pi}}  &  &  \mts{X} \ar[dl]^{\pi}\\
 & & T  &\\
 }$$ satisfying the following conditions:
\begin{enumerate}
    \item $f$ is a small contraction, and is an isomorphism over $T^{\circ}$;
    \item $-K_{\wt{\mts{X}}}=f^{*}(-K_{\mts{X}})$ is a  $\wt{\pi}$-big and $\wt{\pi}$-nef Cartier divisor; 
    \item $\wt{\mts{X}}_0$ is a $\bQ$-Gorenstein smoothable Gorenstein canonical weak Fano variety whose anti-canonical model is isomorphic to $\mts{X}_0$;
    \item $\wt{\cL}\coloneqq f^{-1}_* \cL$ is an $f$-ample Cartier divisor (cf. \Cref{thm:volume bound}(3)); and
    \item\label{item:5} $-K_{\wt{\mts{X}}} + \epsilon \wt{\cL}$ is $\wt{\pi}$-ample, for any real number $0<\epsilon\ll1$.
\end{enumerate}
To ease our notation, we denote by $$\wt{X}\coloneqq \wt{\mts{X}}_0,\ \ \  g = f|_{\wt{X}}: \wt{X} \to X,\ \ \ \textup{and}\ \ \ \wt{\mts{E}}\coloneqq f_{*}^{-1} \mts{E}.$$ We also denote by $\wt{E}$ (resp. $\wt{L}$) the restriction of $\wt{\mts{E}}$ (resp. $\wt{\cL}$) to $\wt{\mts{X}}_0 = \wt{X}$. We will show that $\wt{\cL}$ is relatively big and semiample over $T$; see Proposition \ref{nefness}.

\begin{lemma}\label{quasipolarized}
Let $\wt{S}\in |-K_{\wt{X}}|$ be a general member. Then 
$(\wt{S},\wt{L}|_{\wt{S}})$ is a quasi-polarized K3 surface of degree $8$, and 
$|\wt{L}|_{\wt{S}}|$ is base-point free.
\end{lemma}

\begin{proof}
For any $0<\epsilon\ll 1$, the condition (\ref{item:5}) above implies that $(-K_{\wt{X}}+\epsilon\wt{L})|_{\wt{S}}$ is ample. By deforming to a family of K3 surfaces in $|-K_{\mts{X}_t}|$ 
(cf.~\cite[Lemma 4.4]{LZ25}), one sees that the divisors 
$-K_{\wt{X}}|_{\wt{S}}$ and $\wt{L}|_{\wt{S}}$ generate a primitive sublattice 
of $\Pic(\wt{S})$ isometric to the rank-two lattice $\Lambda$ with Gram matrix
\[
\begin{pmatrix}
22 & 14\\
14 & 8
\end{pmatrix}
\]
with respect to generators $e_1,e_2$. Let $\mts{F}_{\Lambda,\epsilon}$ be the moduli stack of $\Lambda$-polarized ADE 
K3 surfaces (see \cite[ Theorem 5.5]{AE25}), where the very irrational positive vector is 
$2e_1+\epsilon e_2$ for $0<\epsilon\ll1$. We claim that $\wt{L}|_{\wt{S}}$ is nef; 
it is then big since $(\wt{L}|_{\wt{S}})^2=8$.

Suppose otherwise. By \cite[Proposition 4.14]{AE25}, there exists a class 
$v\in \Pic(\wt{S})$ with $(v^2)=-2$ such that the rank-three lattice 
$\Lambda_v\coloneqq \langle v,e_1,e_2\rangle$ is hyperbolic, $(v, e_1+\epsilon e_2) >0$, and $(v,e_2)<0$. 
Write $a\coloneqq (e_1.v)$ and $b\coloneqq (e_2.v)$. Then
\[
\det(\Lambda_v)\ =\ -8a^2+28ab-22b^2+40\ >\ 0.
\]
However, under the constraints $a+\epsilon b>0$ and $b<0$, this inequality has no 
integral solutions, a contradiction. Thus $\wt{L}|_{\wt{S}}$ is nef.

To prove base-point freeness, it suffices to rule out that 
$(\wt{S},\wt{L}|_{\wt{S}})$ is unigonal. If it were, there would exist classes 
$\Sigma,F\in\Pic(\wt{S})$ with
\[
(\Sigma^2)=-2,\quad (F^2)=0,\quad (\Sigma.F)=1,\quad 
\wt{L}|_{\wt{S}}=\Sigma+5F.
\]
The Gram matrix of the lattice generated by 
$-K_{\wt{X}}|_{\wt{S}},\ \wt{L}|_{\wt{S}},\ F$ is then
\[
\begin{pmatrix}
22 & 14 & c\\
14 & 8 & 1\\
c & 1 & 0
\end{pmatrix},
\]
whose determinant is $-22+28c-8c^2>0$, forcing $c=2$. Let $S \in |-K_X|$ be the image of $\wt{S}$ in $X$ under $g:\tX\rightarrow X$. By \Cref{thm:Kss-very-ample}, the linear system $|-K_X|$ is very ample. Hence
\[
c \;=\; (-K_{\wt{X}}|_{\wt{S}} \cdot F)
   \;=\; (-K_X|_S \cdot g_*(F))
   \;\ge\; 3 .
\]
Indeed, any integral curve of degree $\leq 2$ in $\bP H^0(X,-K_X)\simeq \bP^{13}$ is a smooth rational curve, whereas $g(C)$ is an elliptic curve for a general $C \in |F|$. This contradiction shows that $|\wt{L}|_{\wt{S}}|$ is base-point free.
\end{proof}

\begin{lemma}\label{isomorphismonsection}
    For a general K3 surface $\wt{S}\in |-K_{\wt{X}}|$, the restriction map $$H^0(\wt{X},\mtc{O}_{\wt{X}}({\wt{L}}))\ \longrightarrow \ H^0(\wt{S},\mtc{O}_{\wt{S}}({\wt{L}|_{\wt{S}}}))$$ is an isomorphism. In particular, we have that $h^0(\wt{X},\mtc{O}_{\wt{X}}({\wt{L}}))=6$. 
\end{lemma}

\begin{proof}
    Let $\wt{S}\in|-K_{\wt{X}}|$ be a K3 surface. By Lemma \ref{quasipolarized} we see that $(\wt{S},\wt{L}|_{\wt{S}})$ is a quasi-polarized degree $8$ K3 surface, and hence $$\textstyle h^0(\wt{S},\wt{L}|_{\wt{S}})\ =\ \frac{1}{2}({\wt{L}}|_{\wt{S}})^2+2\ =\ 6.$$
    Since $\wt{S}\sim -K_{\wt{X}}$ is Cartier, we have a short exact sequence 
    \[\textstyle 
    0 \ \longrightarrow \ \mtc{O}_{\wt{X}}({\wt{L}}- \wt{S}) \ \longrightarrow \  \mtc{O}_{\wt{X}}({\wt{L}}) \ \longrightarrow \ \mtc{O}_{\wt{S}}({\wt{L}}|_{\wt{S}})\ \longrightarrow \  0.
    \]
    As ${\wt{L}}-\wt{S}\sim \wt{L}+K_{\wt{X}}$ is not effective, then taking the long exact sequence on cohomology we see that  $H^0(\wt{X},\mtc{O}_{\wt{X}}({\wt{L}}))\hookrightarrow H^0(\wt{S},{\wt{L}}|_{\wt{S}})$ is injective, and thus $h^0(\wt{X},\mtc{O}_{\wt{X}}({\wt{L}}))\leq 6$. On the other hand, by upper semi-continuity, we have that $$\textstyle h^0(\wt{X},\mtc{O}_{\wt{X}}({\wt{L}}))\ \geq\  h^0(\wt{\mts{X}}_t,\mtc{O}_{\wt{\mts{X}}_t}({\wt{\cL}_t}))\ =\ 6$$ for a general $t\in T$. Therefore, one has $h^0(\wt{X},\mtc{O}_{\wt{X}}({\wt{L}}))=6$, and the restriction map is an isomorphism.
    
\end{proof}

\begin{prop}\label{nefness} The Cartier divisor $\wt{\cL}$ is $\wt{\pi}$-semiample and  $\wt{\pi}$-big.

\end{prop}

\begin{proof}
    We first prove that $\wt{L}$ is a nef divisor. Recall that $\wt{L}$ is $g$-ample, so we have that $(\wt{L}.C)>0$ for any $g$-exceptional curve $C\subseteq \wt{X}$. We claim that the base locus of the linear series $|\wt{L}|$ is either some isolated points, or is contained in the $g$-exceptional locus, so $\wt{L}$ is nef. By Lemma \ref{quasipolarized}, we know that $|\wt{L}|_{\wt{S}}|$ is base-point-free, where $\wt{S}\in|-K_{\wt{X}}|$ is a general elephant. Suppose that $\wt{C}\subseteq \Bs|\wt{L}|$ is a curve which is not contracted by $g$. Then the intersection $\wt{C}\cap \wt{S}$ is non-empty and consists of finitely many points, which are all base points of $|\wt{L}|_{\wt{S}}|$. This contradicts Lemma \ref{quasipolarized}.

     Since $\wt{L}= \wt{\cL}|_{\wt{\mts{X}}_0}$ is nef, and $\wt{\cL}|_{\wt{\mts{X}}_t}$ is nef for any $t\in T\setminus\{0\}$ as $\wt{\mts{X}}_t\simeq \mts{X}_t$ is a smooth Fano threefold in the family \textnumero 2.16, we conclude that $\wt{\cL}$ is $\wt{\pi}$-nef. This implies the $\wt{\pi}$-semiampleness of $\wt{\cL}$ by the Kawamata--Shokurov base-point free theorem, as $\wt{\mts{X}}$ is of Fano type over $T$. Since $\wt{\cL}|_{\wt{\mts{X}}_t}$ is big for a general $t\in T$, we obtain the $\wt{\pi}$-bigness of $\wt{\cL}$.
    
\end{proof}

Taking the $\wt{\cL}$-ample model over $T$ yields a birational morphism $\phi:\wt{\mts{X}}\to \mts{V}$ that fits into a commutative diagram $$\xymatrix{
 & \wt{\mts{X}} \ar[rr]^{\phi \quad \quad \quad \quad \quad} \ar[dr]_{\wt{\pi}}  &  &  \mts{V}=\Proj_{T}\bigoplus_{m\in \bN} \wt{\pi}_* \big(\wt{\cL}^{\otimes m}\big)
 \ar[dl]^{\pi_{\mts{V}}} \\
 & & T&\\
 }.$$ By the base-point free theorem, as $\wt{\cL}$ is Cartier, it descends to a Cartier divisor $\cL_{\mts{V}}$ on $\mts{V}$ which is $\pi_{\mts{V}}$-ample. For any $0\neq t\in T$, the morphism $\wt{\mts{X}}_t\rightarrow \mts{V}_t$ contracts precisely the ruled surface $\wt{\mts{E}}_t$ to a smooth conic contained in $\mts{V}_t$, which can be embedded into $\mathbb{P}^5$ as a complete intersection of two quadrics by the line bundle $\cL_{\mts{V}}|_{\mts{V}_t}$. Now consider the restriction of the morphism $\phi$ to the central fiber $$\phi_0:\wt{\mts{X}}_0=\wt{X}\longrightarrow V\coloneqq \mts{V}_0.$$ Let $L_V\coloneqq \cL_{\mts{V}}|_{\mts{V}_0}$ be the Cartier divisor, which coincides with $(\phi_{0})_*\wt{L}$. 

\begin{lemma}\label{lem:contract-Q}
    The central fiber $V$ of $\pi_{\mts{V}}$ is a normal projective variety. Moreover, 
    the morphism $\phi_0: \wt{X}\to V$ is birational, contracts $\wt{E}$ to a curve $C_V$ of $V$, and is an isomorphism on $\wt{X}\setminus \wt{E}$.
\end{lemma}

\begin{proof}
    We first show that $V$ is normal and  $\phi_0$ is birational. Since both $\wt{\cL}$ and $-K_{\wt{\mts{X}}}$ are nef and big over $T$ and $\wt{\mts{X}}$ is klt, Kawamata--Viehweg vanishing theorem implies that $R^i \wt{\pi}_* \wt{\cL}^{\otimes m} \ =\  0$ for any $i> 0$ and $m\in \bN$. Thus, by cohomology and base change, the sheaf $\wt{\pi}_* \wt{\cL}^{\otimes m}$ is locally free and satisfies that $$\big(\wt{\pi}_* \wt{\cL}^{\otimes m}\big)\otimes k(0)\ \simeq\ H^0(\wt{X}, \wt{L}^{\otimes m}).$$ As a result, one has that $V\ =\ \mts{V}_0\ \simeq\ \Proj \bigoplus_{m\in \bN} H^0(\wt{X}, \wt{L}^{\otimes m})$ is the ample model of $\wt{L}$ on $\wt{X}$, which implies the normality of $V$ and the birationality of $\phi_0$. 
    
    Consider the restriction $\phi|_{\wt{\mts{E}}}:\wt{\mts{E}}\rightarrow \mts{W}\coloneqq \phi(\wt{\mts{E}})$. Since $\mts{E}$ is an integral scheme and a general fiber of $\mts{W}\rightarrow T$ is a smooth conic, one knows that $\mts{W}$ is an integral surface. In particular, $\mts{W}\rightarrow T$ is flat, and hence $C_V\coloneqq \mts{W}_0$ is also a curve of degree $2$ with respect to $L_V$ and arithmetic genus $0$. On the other hand, if $C\subseteq\wt{X}$ is a curve such that $(C.\wt{L})=0$, then $C\subseteq \wt{E}$ because $2\wt{L}-(1-\epsilon)\wt{E}$ is ample for $0<\epsilon \ll 1$. Thus the last statement is proved.
    
\end{proof}

\begin{comment}
    On a Q-Fano variety, the $\pi_1(X)$ is trivial, so if there is a Cartier divisor which is torsion, then it has to be trivial: otherwise can take an etale cover. However, it could happen that a torsion Q-Cartier Weil divisor is not trivial. For example, on the cubic surface $S$ with 3 $A_2$-singularities, the three lines are only Q-linearly equivalent but not linearly equivalent, and the triple cover $\bP^2\rightarrow S$ is ramified over the 3 singularities.
\end{comment}

\begin{prop}\label{Gorensteincan}
The variety $V$ is a Gorenstein canonical $(2,2)$-complete intersection in $\bP^5$, and $C_V$ is a (possibly singular) conic curve. Moreover, $V$ is generically smooth along each irreducible component of $C_V$, and if $p\in C_V$ is a singularity of $V$, then it is a hypersurface singularity with multiplicity 2.
\end{prop}

\begin{proof}
From the linear equivalence $-K_{\wt{X}}\sim 2\wt{L}-\wt{E}$ and \Cref{lem:contract-Q}, we obtain
\[
-K_{V} \;=\; (\phi_0)_*(-K_{\wt{X}}) \;\sim\; 2L_V ,
\]
which is ample. Since $\wt{X}$ is klt and $\phi_0^*K_V = K_{\wt{X}}-\wt{E}$, it follows that $V$ is also klt. As $L_V$ is Cartier, $V$ is a Gorenstein canonical Fano threefold of Fano index $2$. Therefore, by \cite{Fuj90}, $V$ is embedded in $\bP^5$ via $L_V$ as a $(2,2)$ complete intersection.

Let $\wt{S}\in |-K_{\wt{X}}|$ be a general member, and set $S_V:=(\phi_0)_*\wt{S}$. Then $S_V\sim -K_V$ is Cartier as $V$ is Gorenstein. Since $\wt{S}$ has ADE singularities, the pair $(\wt{X},\wt{S})$ is a plt log Calabi--Yau pair. Hence $(V,S_V)$ is also a plt log Calabi--Yau pair, which in particular implies that $S_V$ is an ADE K3 surface and that $\phi_0|_{\wt{S}}\colon \wt{S}\to S_V$ is birational. Let $\wt{S}_t\in |-K_{\wt{\mts{X}}_t}|$ be a deformation of $\wt{S}$. Since the image of $\wt{S}_t$ contains $\mts{W}_t$, it follows that $C_V\subseteq S_V$. 
Since the Cartier divisor $S_V$ is an ADE K3 surface and is generically smooth along $C_V$, we conclude that $V$ is also generically smooth along each irreducible component of $C_V$, and if $p\in C_V$ is a singularity of $V$, then it is a hypersurface singularity with multiplicity 2.
\end{proof}

\begin{lem}\label{lem:commutativity-blowup-fiber}
For any $t\in T$, one has $(\Bl_{\sW}\sV)_t \simeq \Bl_{\sW_t}\sV_t$.
\end{lem}

\begin{proof}
It suffices to prove the statement for $t=0$. In the proof of \Cref{lem:contract-Q}, we showed that the sheaf $\wt{\pi}_* \wt{\cL}$ is locally free of rank $6$. By shrinking the base $T$, we may assume that $T=\Spec R$, where $R$ is a DVR with uniformizer $t$, and that $\mts{V}\subseteq \mathbb{P}^5_T\coloneqq \mathbb{P}^5\times T$. Since $V$ has hypersurface singularities, the same holds for $\sV$. Hence the exceptional divisor of $\Bl_{\sW}\sV\rightarrow \sV$, denoted by $\sG$, has the property that every fiber of $\sG\rightarrow\sW$ has dimension at most $2$.

We claim that $(\Bl_{\sW}\sV)_0$ is irreducible. Note that over any smooth point $p$ of $\sV_0$, the schemes $(\Bl_{\sW}\sV)_0$ and $\Bl_{\sW_0}\sV_0$ coincide, since $\sW_0$ is a local complete intersection in $\sV_0$. Hence, if $(\Bl_{\sW}\sV)_0$ were reducible, then $\Bl_{\sW}\sV$ would contain a divisor whose center on $\sV$ is a single point, which is a contradiction.

On the other hand, by the same argument as in \Cref{lem:complete-intersection-after-blowup}, $\Bl_{\sW}\sV$ is a complete intersection in $\Bl_{\sW}\bP^5_T$. Since $\sW$ is a complete intersection in $\bP^5_T$, one has $(\Bl_{\sW}\bP^5_T)_0=\Bl_{\sW_0}\bP^5$, and $\Bl_{\sW}\bP^5_T$ is Cohen--Macaulay. It follows that $\Bl_{\sW}\sV$ and $(\Bl_{\sW}\sV)_0$ are also Cohen--Macaulay. Since $(\Bl_{\sW}\sV)_0$ is irreducible and birational to $V=\sV_0$, it is generically reduced, and hence reduced. Therefore, $(\Bl_{\sW}\sV)_0$ is an integral variety.

Combining this with \Cref{lem:complete-intersection-after-blowup}, we see that both $(\Bl_{\sW}\sV)_0$ and $\Bl_{\sW_0}\sV_0$ are integral subvarieties of $\Bl_{\sW_0}\bP^5$ sharing a common open subset, and hence they coincide.
\end{proof}

\begin{prop}\label{prop:blow-up-iso}
    There exists a natural isomorphism $$\mts{X} \ \simeq\ \Bl_{\mts{W}}\mts{V}$$ over $T$. In particular, we have $X\simeq \Bl_{C_V}V$.
\end{prop}

\begin{proof}

By the proof of \Cref{lem:commutativity-blowup-fiber}, we see that $\Bl_{\mts{W}}\mts{V}$ is integral and Cohen--Macaulay, and hence normal. By \Cref{lem:blowup-conic-weak-Fano}, the morphism $\Bl_{\mts{W}}\mts{V}\rightarrow T$ is a family of weak Fano varieties. Let $\mts{X}'\rightarrow T$ be the relative anticanonical ample model; it is also normal.

Notice that $\mts{X}$ and $\mts{X}'$ are isomorphic over $T^{\circ}$, and hence they are birational and isomorphic in codimension one. Since both are normal and have anti-canonical divisors relatively ample over $T$, it follows that they are isomorphic. In particular, $\mts{X}'_0$ is K-semistable. If $\sX'$ were not isomorphic to $\Bl_{\sW}\sV$, then $\Bl_{\sW_0}\sV_0$ would not be Fano. Hence, by \Cref{lem:weakFano-plane}, $V$ would contain the plane spanned by $C_V$. However, \Cref{cor:blowup-conic-unstable} implies that $\mts{X}'_0$ is K-unstable, which is a contradiction. Therefore, $\Bl_{\mts{W}}\mts{V}\rightarrow T$ is a family of Fano varieties, and $\mts{X}'\simeq \Bl_{\mts{W}}\mts{V}\simeq \mts{X}$. The last statement follows immediately from \Cref{lem:commutativity-blowup-fiber}.
\end{proof}

\begin{prop}\label{prop:deformation}
    Let $X$ be the blow-up of a Gorenstein canonical $(2,2)$-complete intersection $V$ in $\bP^5$ along a conic curve $C$ which is not contained in the singular locus of $V$. Assume in addition that $V$ does not contain the 2-plane spanned by $C$. Then $\Ext^2(\Omega^1_{X},\mtc{O}_X)=0$. In particular, there are no obstructions to deformation of $X$.
\end{prop}

\begin{proof}
   The blow-up morphism $\phi:X\rightarrow V\subseteq \mathbb{P}^5$ sits in a commutative diagram $$\xymatrix{
 & X \ar[rr]^{\phi} \ar@{^(->}[d]  &  &  V \ar@{^(->}[d]\\
 & \wt{\mathbb{P}}\coloneqq \Bl_{C}\mathbb{P}^5  \ar[rr]^{\quad \psi}  &   & \mathbb{P}^5. \\
 }$$ Let $L\coloneqq \psi^{*}\mtc{O}_{\mathbb{P}^5}(1)$, $F$ be the $\psi$-exceptional divisor, i.e. the closed subscheme defined by the ideal $\psi^{-1}I_{C/\bP^5}\cdot \cO_{\wt{\bP}}$, and $E\coloneqq F\cap X$ be the $\phi$-exceptional locus. Since $C$ is not contained in the singular locus of $V$, $X$ is a complete intersection in $\wt{\bP}$, and hence $$N_{X/\wt{\mathbb{P}}}\ \simeq \ \mtc{O}_X(2L-F)\oplus\mtc{O}_X(2L-F),$$ which is big and nef. By taking $R\Hom(\cdot, \mtc{O}_X)$ of the short exact sequence $$0 \ \longrightarrow \ (N_{X/\wt{\mathbb{P}}})^*\ \longrightarrow\ \Omega^1_{\wt{\mathbb{P}}}|_X\  \longrightarrow\   \Omega^1_X \ \longrightarrow\   0,$$ it suffices to show that $H^1(X,N_{X/\wt{\mathbb{P}}})=0$ and $\Ext^2(\Omega^1_{\wt{\mathbb{P}}}|_X,\mtc{O}_X)=0$. The first vanishing follows immediately from Kawamata-Viehweg vanishing.

 To show the second vanishing, let us consider the short exact sequence $$0 \longrightarrow \psi^*\Omega^1_{\mathbb{P}^5}|_X \ \longrightarrow\ \Omega^1_{\wt{\mathbb{P}}}|_X \ \longrightarrow\  \Omega^1_{F/C}|_X \ \longrightarrow \ 0,$$ where we use the fact that $F$ intersects generically transversely with $X$. It suffices to show that $\Ext^2(\psi^*\Omega^1_{\mathbb{P}^5}|_X,\mtc{O}_X)=0$ and $\Ext^2(\Omega^1_{F/C}|_X,\mtc{O}_X)=0$. Consider the pull-back of the Euler sequence  $$0 \ \longrightarrow\  \mtc{O}_X \ \longrightarrow\ (L|_X)^{\oplus 6 } \ \longrightarrow \  \psi^{*}T_{\mathbb{P}^5}|_X \ \longrightarrow \  0.$$ One has that $H^3(X,\mtc{O}_X)=0$ and $H^2(X,L|_X)=0$ by Kawamata-Viehweg vanishing, and hence $$\Ext^2(\psi^*\Omega^1_{\mathbb{P}^5}|_X,\mtc{O}_X)\ \simeq\  H^2(X,\psi^{*}T_{\mathbb{P}^5}|_X)\ =\ 0.$$ Let $p:F\rightarrow C$ be the natural projection. As $C$ is a complete intersection of hypersurfaces in $\bP^5$ of multidegree $(1,1,1,2)$, then $N_{C/\bP^5}\simeq \cE\coloneqq (\mtc{O}_{\bP^1}(1)^{\oplus3}\oplus\mtc{O}_{\bP^1}(2))|_C$, and hence $F\rightarrow C$ is isomorphic to $\bP\cE^*\rightarrow C$. Consider the relative Euler sequence on $F\simeq \bP\cE^*\rightarrow C$ (each entry viewed as a torsion sheaf on $\wt{\mathbb{P}}$) $$0 \ \longrightarrow \ \Omega^1_{F/C}\ \longrightarrow \ (p^{*}\cE^*)(-1) \ \longrightarrow \   \mtc{O}_{F} \longrightarrow  0.$$ Twisting by $\mtc{O}_X$, one obtains $$0 \ \longrightarrow \ \Omega^1_{F/C}|_X\ \longrightarrow \ (p^{*}\cE^*)(-1)|_X \ \longrightarrow \   \mtc{O}_{F\cap X} \longrightarrow  0.$$ We now need to prove that $\Ext^2((p^{*}\cE^*)(-1)|_X,\mtc{O}_X)=0$ and $\Ext^3(\mtc{O}_F|_X,\mtc{O}_X)=0$. Since $\mtc{O}_F|_X=\mtc{O}_E$, then by Serre duality one has $$\Ext^3_X(\mtc{O}_F|_X,\mtc{O}_X)\ \simeq\  H^0(X,\omega_X\otimes \mtc{O}_F|_X)^{*} \ \simeq\  H^0(E,\omega_X|_E)^{*}\ =\ 0$$ as $-K_X|_E$ is big and nef. Similarly, for $m=1,2$ one has $$\Ext^2_X((p^{*}\mtc{O}_{C}(-m))(-1)|_X,\mtc{O}_X)\ \simeq\  H^1(X,\omega_X\otimes (p^{*}\mtc{O}_{C}(-m))(-1)|_X)^* \ \simeq\  H^1(E,(2E-(m+2)L)|_E)^{*},$$ where we use that $\mtc{O}_F(-1)=\mtc{O}_{F}(F)$ and $\omega_X\sim F-2L$. If $m=2$, then $((m+2)L-2E)|_E$ is big and nef, and hence $H^1(E,(2E-(m+2)L)|_E)=0$ by Kawamata-Viehweg vanishing. For $m=1$, consider the short exact sequence $$0\ \longrightarrow \mtc{O}_X(E-2L)\ \longrightarrow \mtc{O}_X(2E-3L)\ \longrightarrow \ \mtc{O}_E(2E-3L)\ \longrightarrow \ 0.$$ Since $2L-F$ is big and nef, then by Kawamata-Viehweg Vanishing $\mtc{O}_X(F-2L)$ has no middle cohomology, and hence $$H^1(E,(2E-3L)|_{E})\ =\ H^1(X,2E-3L)\ \simeq \ H^2(X,L-E)^* \ = \ 0$$ by Kawamata-Viehweg vanishing theorem, where we use the fact that $(L-F)|_X$ is nef. 
 %Consider the short exact sequence $$0\ \longrightarrow \mtc{O}_X(L-E)\ \longrightarrow \mtc{O}_X(L)\ \longrightarrow \ \mtc{O}_E(L)\ \longrightarrow \ 0.$$ Since $L$ is big and nef, then by Kawamata-Viehweg Vanishing $\mtc{O}_X(L)$ has no higher cohomology, and hence $$H^2(X,L-E) \ =\  H^1(E,L|_{E}) \ \simeq \ H^1(C,\mtc{O}_{\bP^5}(1)|_C)\ = \ 0 ,$$ where for the middle isomorphism we use $R^1p_*\mtc{O}_E=0$. 
\end{proof}

\begin{proof}[Proof of \Cref{thm:K-ss limit of 2.16}]
    It follows from \Cref{thm:volume bound}(1) that $X$ is Gorenstein canonical. By \Cref{prop:blow-up-iso}, we have $X\simeq \Bl_{C_V}V$ where $V$ is a Gorenstein canonical $(2,2)$-complete intersection in $\bP^5$ and $C_V\subseteq V$ is a conic. Thus by \Cref{prop:deformation}, the K-moduli stack $\mtc{M}^\K_{\textup{\textnumero2.16}}$ is smooth, and hence a smooth connected component of $\cM^\K_{3,22}$.
\end{proof}

\medskip
\section[tocentry={K-stability of special Type I degeneration}]{K-stability of special Type I degeneration}\label{appendix B}

In this appendix, we construct a three-dimensional family of K-stable $V_{22}$ with non-isolated singularities. 
These arise as degenerations of Type I one-nodal $V_{22}$.

\smallskip
Let $Q \subset \bP^3$ be the smooth quadric surface defined by $x_0x_3 = x_1x_2$. Fix the isomorphism $Q\simeq \bP^1\times\bP^1$ given by
\[
([u_0:u_1],[v_0:v_1]) \ \mapsto \ [u_0v_0:u_0v_1:u_1v_0:u_1v_1].
\]
Let $C\subset Q$ be a curve of bidegree $(1,4)$ with respect to the coordinates $([u_0:u_1],[v_0:v_1])$. 
Let 
\[
\pi\ :\ Y:=\Bl_C\bP^3\ \longrightarrow \ \bP^3
\]
be the blowup, and let $f:Y\rightarrow X$ be the contraction of the strict transform $\wt{Q}$ of $Q$. 
Then $X$ is a Gorenstein canonical Fano threefold with an $A_{\infty}$-singularity along a rational curve of $(-K_X)$-degree $3$. The Fano threefolds obtained in this way form a three-dimensional family. As Type I one-nodal $V_{22}$ are obtained by taking the anticanonical ample model of the blowup of $\bP^3$ along a general rational quintic curve (cf.\ \Cref{thm:Prokhorov's type I-IV}), $X$ can be viewed as a degeneration of Type I $V_{22}$.

\begin{thm}\label{thm:existence of K-stable with non-isolated sing}
For a general $(1,4)$-curve $C$, the Fano threefold $X$ is K-stable.
\end{thm}

Let $C_0\subset Q$ be the special $(1,4)$-curve defined by
\[
u_0^4v_0=u_1^4v_1.
\]
Then $C_0$ is a rational curve parametrized by
\[
[y_0:y_1] \ \mapsto \ \big([y_0:y_1],[y_1^4:y_0^4]\big),
\]
and its image in $\bP^3$ is the rational quintic curve
\[
[y_0:y_1]\ \mapsto \ [y_0y_1^4:y_0^5:y_1^5:y_0^4y_1].
\] Let $Y_0:=\Bl_{C_0}\bP^3$ and let $X_0$ be its anticanonical ample model. 
To prove \Cref{thm:existence of K-stable with non-isolated sing}, we first establish the following result using the equivariant K-stability and admissible flag method (cf.\ \cite{AZ22, Zhu21, ACC+}), and then conclude using deformation theory.

\begin{prop}\label{prop:non-iso-K-ps}
The Fano variety $X_0$ is K-polystable.
\end{prop}

\begin{proof}
Let $G$ be the subgroup in $\Aut(\bP^3)$ generated by the involution $$[x_0:x_1:x_2:x_3] \ \mapsto \ [x_3:x_2:x_1:x_0]$$ and
automorphisms $$[x_0:x_1:x_2:x_3] \ \mapsto \ [\lambda x_0:\lambda^5x_1:x_2:\lambda^4x_3],$$ where $\lambda\in \bG_m$. Then $G\simeq \bG_m\rtimes \mu_2$, and $C_0$, $Q$ are both $G$-invariant. Thus, the action of the group $G$ lifts to the threefold $X_0$.

\begin{lemma}\label{lem: no fixed points and all}
    With the notation above, $\bP^3$ does not contain $G$-fixed points or $G$-invariant planes; and the only $G$-invariant lines in $\bP^3$ are the lines $$\ell_1\coloneqq V(x_0,x_3)  \ \ \ \textup{and} \ \ \ \ell_2\coloneqq V(x_1,x_2).$$ Moreover, one has $$\ell_1\cap Q \ = \ \ell_1\cap C_0 \ = \ [0:1:0:0]\cup[0:0:1:0],$$ and $$\ell_2\cap Q \ = \ [1:0:0:0]\cup[0:0:0:1], \  \ \ \ \ell_2\cap C_0 \ = \ \emptyset.$$
\end{lemma}

\begin{proof}
    The proof is elementary.
\end{proof}

\begin{lemma}
    The Fano variety $X_0$ is K-semistable if and only if the weak Fano variety $Y_0$ is K-semistable.
\end{lemma}

\begin{proof}
    This is because $f:Y_0\rightarrow X_0$ is crepant, i.e. $f^*K_{X_0}=K_{Y_0}$.
\end{proof}

\begin{lemma}\label{lem:divisorial stability}
    The weak Fano variety $Y_0$ is divisorially K-semistable.
\end{lemma}

\begin{proof}
    It suffices to show that for any prime divisor $F$ on $Y_0$, one has $S_{Y_0}(F)<1$. Note that the pseudo-effective cone of $Y_0$ is generated by two effective divisors $E$ and $\wt{Q}$, where $E$ is the $\pi$-exceptional divisor. Denote by $H$ the pull-back on $Y_0$ of the hyperplane class of $\bP^3$. Then $$-K_{Y_0}-u\wt{Q} \ \sim_{\bQ} \ 4H-E-u(2H-E),$$ whose positive part of the Zariski decomposition is \begin{equation}\nonumber
        P(-K_{Y_0}-u\wt{Q}) \ = \ \begin{cases}
        (4-2u)H-(1-u)E & 0\leq u\leq 1\\
        (4-2u)H & 1\leq u \leq 2.
    \end{cases}
    \end{equation} It follows that 
    \begin{equation}\nonumber
    \begin{split}
         22\cdot S_{Y_0}(\wt{Q}) \ & = \  \int_0^1(-K_{Y_0}-u\wt{Q})^3 du + \int_1^2\big((4-2u)H\big)^3du \\ 
         &= \ \int_0^1(4-2u)^3-15(4-2u)(1-u)^2+18(1-u)^3 du + \int_1^2(4-2u)^3du \\
         & = \ 17+2 \ = 19, 
    \end{split}
    \end{equation} and hence $S_{Y_0}(\wt{Q})=\frac{19}{22}$. Similarly, one has $$-K_{Y_0}-uE \ \sim_{\bQ} \ 4H-(1+u)E,$$ whose positive part of the Zariski decomposition is \begin{equation}\nonumber
        P(-K_{Y_0}-uE) \ = \ (1-u)(4H-E) \ \sim_{\bQ} \ (1-u)(-K_{Y_0})
    \end{equation} for $0\leq u\leq 1$. It follows that 
    \begin{equation}\nonumber
         S_{Y_0}(E) \ = \  \int_0^1 (1-u)^3 du  \ = \ \frac{1}{4}\ < \ 1 ,
    \end{equation} and hence $Y_0$ is divisorial K-semistable.
\end{proof}

Suppose that $Y_0$ is K-unstable. Then there exists a prime divisor $F$ over $Y_0$, which is $G$-invariant, such that $\delta_{Y_0}(F)\leq 1$. Let $Z\coloneqq C_{Y_0}(F)$ be the center of $F$ on $Y_0$. By Lemma \ref{lem: no fixed points and all} and Lemma \ref{lem:divisorial stability}, $Z$ must be a curve on $Y_0$.

\begin{lemma}\label{lem:Z not in Q}
    The curve $Z$ is not contained in $\wt{Q}$.
\end{lemma}

\begin{proof}
    Suppose otherwise, then by \cite[Corollary 1.110]{ACC+}, one has that $S(W^{\wt{Q}}_{\bullet,\bullet},Z)>1$. We compute in Lemma \ref{lem:divisorial stability} that positive part of the Zariski decomposition of $-K_{Y_0}-u\wt{Q}$ is \begin{equation}\nonumber
        P(-K_{Y_0}-u\wt{Q}) \ = \ \begin{cases}
        (4-2u)H-(1-u)E & 0\leq u\leq 1\\
        (4-2u)H & 1\leq u \leq 2.
    \end{cases}
    \end{equation} It follows that \begin{equation}\nonumber
        P(-K_{Y_0}-u\wt{Q})|_{\wt{Q}} \ = \ \begin{cases}
        \mathcal{O}(3-u,2u) & 0\leq u\leq 1\\
        \mathcal{O}(4-2u,4-2u) & 1\leq u \leq 2,
    \end{cases}
    \end{equation} where we identify $\wt{Q}$ with $\bP^1\times\bP^1$. By Lemma \ref{lem: no fixed points and all} again, $Z-\Delta$ is an effective divisor on $\bP^1\times\bP^1$, where $\Delta$ is the diagonal. Thus we have that 
    \begin{equation}
        \begin{split}
S(W^{\wt{Q}}_{\bullet,\bullet},Z)\ & \leq \ S(W^{\wt{Q}}_{\bullet,\bullet},\Delta) \\
& = \ \frac{3}{22}\left(\int_0^1\int_0^{2u}2(3-u-v)(2u-v)dvdu+\int_1^2\int_0^{4-2u}2(4-2u-v)^2dvdu\right)\\
& = \ \frac{3}{22}\left(\frac{7}{3}+\frac{4}{3}\right) \ = \ \frac{1}{2} \ < \ 1.
        \end{split}
    \end{equation} This leads to a contradiction. 
\end{proof}

As we assume $Y_0$ is K-unstable, we have $\alpha_{G,Z}(Y_0)<\frac{3}{4}$, and hence there is a $G$-invariant effective $\bQ$-divisor $D$ on $Y_0$ and a positive rational number $\lambda<\frac{3}{4}$ such that $D\sim_{\bQ} -K_{Y_0}$ and $Z$ is contained in the non-klt locus of $(Y_0,\lambda D)$, denoted by $\Nklt(Y_0,\lambda D)$.

\begin{lemma}\label{lem:1-dim nklt locus}
    The locus $\Nklt(Y_0,\lambda D)$ does not contain any $G$-irreducible surface.
\end{lemma}

\begin{proof}
    Suppose that $S$ is a $G$-irreducible surface contained in $\Nklt(Y_0,\lambda D)$. We can write $D = \gamma S + D'$, where $\gamma$ is a rational number such that $\gamma \geq \frac{1}{\lambda}$, and $D'$ is an effective $\bQ$-divisor on $Y_0$ whose support does not contain $S$. If $S = E$, then $$2\wt{Q}+E \sim_{\bQ} -K_{Y_0} \ \sim_{\bQ} \ \gamma S + D',$$ which implies that $2\wt{Q}-(\gamma-1)E$ is pseudo-effective, which is impossible. Thus $S\sim aH-bE$ for some $a\in \bZ_{>0}$ and $b\in \bZ_{\geq0}$ with $b\leq \frac{a}{2}$. Moreover, we have $a\gamma\leq 4$, and hence either $a = 1 $ or $a = 2$. 
    
    If $a=2$ and $b=0$, then one has that $D'\sim_{\bQ}(4-2\gamma)H-E$ is an effective $\bQ$-divisor, which is a contradiction. If $a = 2$ and $b = 1$, then $S = \wt{Q}$, which is a contradiction as $Z\subsetneq \wt{Q}$ by Lemma \ref{lem:Z not in Q}. If $a = 1$ and $b = 0$, then $\pi(S)$ is a $G$-invariant plane, which is impossible by Lemma \ref{lem: no fixed points and all}.
\end{proof}

\begin{lemma}
    The curve $Z$ is rational.
\end{lemma}

\begin{proof}
    Let $\ove{D}\coloneqq f(D)$ and $\ove{Z}\coloneqq f(Z)$, where $f:Y_0\rightarrow X_0$ is the contraction of $\wt{Q}$. Since $Z\subsetneq \wt{Q}$, then we see that $Z$ is a $G$-invariant irreducible curve, the induced morphism $Z\rightarrow \ove{Z}$ is birational, and $\ove{Z}\subseteq \Nklt(X_0,\lambda\ove{D})$. As $\Nklt(X_0,\lambda\ove{D})$ has no 2-dimensional components, then $Z$ is a smooth rational curve by \cite[Corollary A.14]{ACC+}.
\end{proof}

\begin{lemma}
    The curve $Z$ is not contained in $E$.
\end{lemma}

\begin{proof}
    The normal bundle $N_{C/\bP^3}$ is balanced, i.e. isomorphic to $\mathcal{O}_{\bP^1}(9)\oplus\mathcal{O}_{\bP^1}(9)$, by \cite[Theorem 3.2]{CR18}, and hence $E$ is isomorphic to $\bP^1\times\bP^1$.

    Write $D = aE + \Delta$, where $\Delta$ is an effective $\bQ$-divisor whose support does not contain $E$, and $a$ is a non-negative rational number. Then $a\lambda\leq 1$ by Lemma \ref{lem:1-dim nklt locus}. Note that $-K_{Y_0}\sim 2\wt{Q}+E$ and $Z \not\subseteq \Nklt\big(Y_0, \lambda(2\wt{Q}+E)\big)$ as $Z\not\subseteq \wt{Q}$ by Lemma \ref{lem:Z not in Q}. Thus, by replacing $D$ by $(1+\mu)D-\mu(2\wt{Q}+E)$ for some $\mu\geq0$, one may assume that 
    \begin{itemize}
        \item $Z\subseteq \Nklt(Y_0,\lambda D)$, and
        \item $\Supp(\Delta)$ does not contain either $\wt{Q}$ or $E$. 
    \end{itemize} 
   However, if $a>0$, then $\Delta$ must have $\wt{Q}$ in its support. Thus $E\not\subseteq \Supp D$. By inversion of adjunction, one has that $Z\subseteq \Nklt(E,\lambda \Delta|_E)$. Since $\Delta|_E \sim \mathcal{O}_{\bP^1\times\bP^1}(1,11)$, $\lambda\leq \frac{3}{4}$, then $Z$ has to be a ruling of $E\rightarrow C$. However, this is impossible as $\bP^3$ has no $G$-fixed point by Lemma \ref{lem: no fixed points and all}.
\end{proof}

\begin{lemma}
    The curve $\pi(Z)$ is a $G$-invariant line in $\bP^3$.
\end{lemma}

\begin{proof}
    Let $\widehat{D}\coloneqq \pi(D)$ and $\widehat{Z}\coloneqq \pi(Z)$. Then $\widehat{Z}$ is a rational curve in $\bP^3$ such that $\widehat{Z}\subseteq  \Nklt(\bP^3,\lambda \widehat{D})$. As $\Nklt(\bP^3,\lambda \widehat{D})$ is 1-dimensional by Lemma \ref{lem:1-dim nklt locus}, then one can conclude by \cite[Corollary A.10]{ACC+}.
\end{proof}

Combining with Lemma \ref{lem: no fixed points and all}, one sees that either $\pi(Z)=\ell_1$ or $\pi(Z)=\ell_2$. 

\smallskip

Let $S$ be the proper transform on $Y_0$ of a general hyperplane in $\bP^3$ containing $\pi(Z)$. Then $S$ is a smooth del Pezzo surface of degree $4$. Let $t$ be a non-negative real number. Note that $-K_{Y_0}-tS\sim_{\bR}(2-t)H+\wt{Q}$ is pseudo-effective if and only if $0\leq t\leq 2$. Moreover, we have that the positive and negative parts of its Zariski decomposition are $$P(-K_{Y_0}-tS) \ = \ \frac{2-t}{2}(4H-E), \ \ \ \textup{and} \ \ \ N(-K_{Y_0}-tS) \ = \ \frac{t}{2}\wt{Q}$$ so that $Z\not\subseteq N(-K_{Y_0}-tS)$. It follows that $$S_{Y_0}(S) \ = \ \int_{0}^2 \left(1-\frac{t}{2}\right)^3 dt \ = \ \frac{1}{2}.$$

\medskip

Therefore, to deduce a contradiction and complete the proof of \Cref{prop:non-iso-K-ps}, by \cite[Corollary 1.110]{ACC+}, it suffices to show that $S(W^{S}_{\bullet,\bullet},\ell_i)< 1$ for $i=1,2$. This will be established in the remainder of the proof.

\medskip

Let $\phi\coloneqq \pi|_S:S\rightarrow T\simeq \bP^2$ be the birational morphism, which contracts 5 disjoint rational curves $e_1,...,e_5$. Then $E|_S = e_1+\cdots+e_5 $, and $$C\ \coloneqq \ \wt{Q}|_S\ \sim\ 2L-(e_1+\cdots+e_5)$$ is the strict transform of a conic such that $(C^2)=-1$, where $L$ is the class of a hyperplane section of $T$.

\begin{lemma}
    One has $S(W^{S}_{\bullet,\bullet},\ell_2)< 1$.
\end{lemma}

\begin{proof}
    For any $0\leq t\leq 2$ and $s\geq0$, one has that $$P(-K_{Y_0}-tS)|_S-s\ell_2 \ = \ \frac{2-t}{2}(4L-\sum e_i)-sL \ = \ (4-2t-s)L-\frac{2-t}{2}\sum e_i ,$$ which is pseudo-effective if and only if $0\leq s\leq 2-t$. Moreover, the positive part $P(s,t)$ of $P(-K_{Y_0}-tS)|_S-s\ell_2$ is \begin{equation}\nonumber
        P(s,t) \ = \ \begin{cases}
        (4-2t-s)L-\frac{2-t}{2}\sum e_i & \ \ \  0\leq s\leq \frac{6-3t}{4}\\
        (2-t-s)(5L-2\sum e_i) &  \ \ \ \frac{6-3t}{4}\leq s \leq 2-t,
    \end{cases}
    \end{equation} and hence \begin{equation}\nonumber
        \vol\big(P(s,t)\big) \ = \ \begin{cases}
        (4-2t-s)^2-\frac{5}{4}(2-t)^2 & \ \ \  0\leq s\leq \frac{6-3t}{4}\\
        5(2-t-s)^2 &  \ \ \ \frac{6-3t}{4}\leq s \leq 2-t.
    \end{cases}
    \end{equation} It follows that 
    \begin{equation}\nonumber
        \begin{split}
        S(W^{S}_{\bullet,\bullet},\ell_2)\
& = \ \frac{3}{22}\int_0^2dt\left(\int_0^{\frac{6-3t}{4}}(4-2t-s)^2-\frac{5}{4}(2-t)^2 ds+\int_{\frac{6-3t}{4}}^{2-t}5(2-t-s)^2 ds\right)\\
& = \  \frac{53}{88} \ < \ 1.
        \end{split}
    \end{equation}
\end{proof}

\begin{lemma}
    One has $S(W^{S}_{\bullet,\bullet},\ell_1)< 1$.
\end{lemma}

\begin{proof}
    We may assume that the class of $\ell_1$ is $L-e_1-e_2$. For any $0\leq t\leq 2$ and $s\geq0$, one has that $$P(-K_{Y_0}-tS)|_S-s\ell_1 \ = \ \frac{2-t}{2}(4L-\sum e_i)-s(L-e_1-e_2) ,$$ which is pseudo-effective if and only if $0\leq s\leq \frac{5}{4}(2-t)$. Moreover, the positive part $P(s,t)$ of $P(-K_{Y_0}-tS)|_S-s\ell_1$ is \begin{equation}\nonumber
        P(s,t) \ = \ \begin{cases}
        \frac{2-t}{2}(4L-\sum e_i)-s(L-e_1-e_2)  & \ \ \  0\leq s\leq \frac{2-t}{2}\\
        \frac{2}{3}(\frac{5(2-t)}{4}-s)(3L-e_3-e_4-e_5) &  \ \ \ \frac{2-t}{2}\leq s \leq \frac{5(2-t)}{4},
    \end{cases}
    \end{equation} and hence \begin{equation}\nonumber
        \vol\big(P(s,t)\big) \ = \ \begin{cases}
        \frac{11}{4}(2-t)^2-s^2-4s(2-t) & \ \ \  0\leq s\leq \frac{2-t}{2}\\
        \frac{8}{3}(\frac{5(2-t)}{4}-s)^2 &  \ \ \ \frac{2-t}{2}\leq s \leq \frac{5(2-t)}{4}.
    \end{cases}
    \end{equation} It follows that 
    \begin{equation}\nonumber
        \begin{split}
            S(W^{S}_{\bullet,\bullet},\ell_1)\
& = \ \frac{3}{22}\int_0^2dt\left(\int_0^{\frac{2-t}{2}}\left(\frac{11}{4}(2-t)^2-s^2-4s(2-t)\right)ds+\int_{\frac{2-t}{2}}^{\frac{5}{4}(2-t)} \frac{8}{3}\left(\frac{5(2-t)}{4}-s\right)^2 ds\right)\\
& = \  \frac{29}{44} \ < \ 1.
        \end{split}
    \end{equation}
\end{proof}
This contradicts the assumption that $X_0$ is not K-polystable. Hence $X_0$ is K-polystable.
\end{proof}

\begin{lemma}\label{lem:vanishing obstruction space}
The obstruction space for the deformation of $X$ is trivial. In particular, $\cM^\K_{3,22}$ is smooth at $[X]$.
\end{lemma}

\begin{proof}
Since $X$ has only $A_\infty$-singularities along a rational curve $\Gamma$, the tangent sheaf $T_X$ is Cohen--Macaulay. Hence, by Serre duality for Cohen--Macaulay sheaves (cf.\ \cite[Theorem~5.71]{KM98}),
\[
H^2(X,T_X)\ \simeq\  H^1(X,\Omega^{[1]}_X\otimes \omega_X)^{\vee},
\]
which vanishes by \cite[Proposition~4.3]{GKP14}. As $X$ is a smoothable degeneration of a smooth $V_{22}$, the sheaf $\cE\textup{xt}^1(\Omega_X,\cO_X)$ is a line bundle of non-negative degree on $\Gamma\simeq\bP^1$, and therefore $H^1(X,\cE \textup{xt}^1(\Omega_X,\cO_X))=0$. Moreover, since $X$ has l.c.i.\ singularities, $\cE\textup{xt}^2(\Omega_X,\cO_X)=0$, thus $H^0(X,\cE\textup{xt}^2(\Omega_X,\cO_X))=0$. By the local-to-global Ext spectral sequence, it follows that $\Ext^2(\Omega_X,\cO_X)=0$.
\end{proof}

\begin{lemma}
The Fano variety $X$ satisfies $h^1(X,T_X)=4$ and $h^0(X,\cE\textup{xt}^1(\Omega_X,\cO_X))=3$.
\end{lemma}

\begin{proof}
By the local-to-global Ext spectral sequence and \Cref{lem:vanishing obstruction space}, the smoothness of the 6-dimensional moduli at $X$, together with $\dim\Aut(X)=1$, yields
\[
h^1(X,T_X) \ +\ h^0(X,\cE\textup{xt}^1(\Omega_X,\cO_X))\ = \ 7.
\] For any small locally trivial deformation $X_t$ of $X$, a general anticanonical K3 surface $S_t\in |-K_{X_t}|\simeq \bP^{13}$ has three $A_1$-singularities. Hence, by \Cref{thm:open immersion forgetful map}, the space of locally trivial deformations of $X$ has dimension at most $19-3-13=3$. On the other hand, since the moduli of $(1,4)$-curves on $\bP^1\times\bP^1$ is three-dimensional, $X$ has a 3-dimensional equi-singular moduli. Therefore
\[
h^1(X,T_X)\ =\ 3+\dim\Aut(X)\ =\ 4,
\]
and consequently $h^0(X,\cE\textup{xt}^1(\Omega_X,\cO_X))=3$.
\end{proof}

\begin{proof}[Proof of \Cref{thm:existence of K-stable with non-isolated sing}]
The same argument as in the proof of \cite[Corollary~1.16]{ACC+} shows that a general locally trivial deformation of $X$ is K-polystable, applied here to $\Def^{\mathrm{lt}}(X)$ in place of $\Def(X)$. Since $\Def^{\mathrm{lt}}(X)$ is smooth, the argument carries over verbatim. Moreover, a general locally trivial deformation $X_t$ has finite automorphism group, and hence is K-stable.
\end{proof}

\bibliography{ref}

\end{document}